\documentclass[sn-mathphys,Numbered]{sn-jnl}% Math and Physical Sciences Reference Style
\usepackage{graphicx}%
\usepackage{multirow}%
\usepackage{amsmath}
\usepackage{amsthm}%
\usepackage{mathrsfs}%
\usepackage[title]{appendix}%
\usepackage{xcolor}%
\usepackage{textcomp}%
\usepackage{manyfoot}%
\usepackage{booktabs}%
\usepackage{algorithm}%
\usepackage{algorithmicx}%
\usepackage{algpseudocode}%
\usepackage{listings}%
\usepackage{dsfont}
\usepackage{float}
\usepackage{subfigure}
\usepackage{mathrsfs}
\usepackage{eepic}
\usepackage{amsfonts}
\usepackage{verbatim}
\usepackage{ragged2e}
\usepackage[numbers,sort&compress]{natbib}% ²Î¿¼ÎÄÏ×Á¬×Ö·û£¬×Ô¶¯ÅÅÐò
\usepackage{arydshln}

\usepackage{url}
\usepackage{color}
\usepackage{booktabs}
\usepackage{threeparttable}
\usepackage{amssymb}
\usepackage{bbding}
\usepackage{booktabs}
\usepackage{graphicx}
\usepackage{dsfont}
\usepackage{amsmath}
\usepackage{float}
\usepackage{subfigure}
\usepackage{multirow}
\usepackage{multicol}
\usepackage{lscape}
\usepackage{float}
\newtheorem{theorem}{Theorem}%  meant for continuous numbers
\newtheorem{assumption}{Assumption}
\newtheorem{remark}{Remark}%
\newtheorem{lemma}{Lemma}
\newtheorem{corollary}{Corollary}
\newtheorem{definition}{Definition}%

\begin{document}
	
	\title[Stochastic smoothing accelerated gradient method for general constrained nonsmooth  
	convex 
	composite optimization]{Stochastic smoothing accelerated gradient method for general constrained nonsmooth  
		convex 
		composite optimization}
	
	%%=============================================================%%
	%% Prefix	-> \pfx{Dr}
	%% GivenName	-> \fnm{Joergen W.}
	%% Particle	-> \spfx{van der} -> surname prefix
	%% FamilyName	-> \sur{Ploeg}
	%% Suffix	-> \sfx{IV}
	%% NatureName	-> \tanm{Poet Laureate} -> Title after name
	%% Degrees	-> \dgr{MSc, PhD}
	%% \author*[1,2]{\pfx{Dr} \fnm{Joergen W.} \spfx{van der} \sur{Ploeg} \sfx{IV} \tanm{Poet Laureate} 
		%%                 \dgr{MSc, PhD}}\email{iauthor@gmail.com}
	%%=============================================================%%
	
	\author[1]{\fnm{Ruyu} \sur{Wang}}\email{wangruyu@nwafu.edu.cn}
	
	\author*[1]{\fnm{Chao} \sur{Zhang}}\email{zc.njtu@163.com}
	\equalcont{These authors contributed equally to this work.}

%		\affil[1]{\orgdiv{College of Science}, \orgname{Northwest A\&F University}, \orgaddress{\street{No.3 Taicheng Road}, \city{Yangling}, \postcode{712100}, \state{Shaanxi}, \country{China}}}
		
	\affil[1]{\orgdiv{School of Mathematics and Statistics}, \orgname{Beijing Jiaotong University}, \orgaddress{\street{No.3 Shangyuancun}, \city{Haidian}, \postcode{100044}, \state{Beijing}, \country{China}}}

	%%==================================%%
	%% sample for unstructured abstract %%
	%%==================================%%
	
	\abstract{We propose a novel stochastic smoothing accelerated gradient (SSAG) method for general constrained nonsmooth convex composite optimization, and analyze the convergence rates. The SSAG method allows various smoothing techniques, and can deal with the nonsmooth term that is not easy to compute its proximal term, or that does not own the linear max structure. To the best of our knowledge, it is the first time to develop a stochastic approximation type method that treats the maximization of finite but numerous nonsmooth convex functions as a stochastic function, which significantly improves the computational efficiency. We prove that the SSAG method can simultaneously achieve the best-known order ${\cal{O}}(\frac{1}{\epsilon})$ of iteration complexity, and the optimal order ${\cal{O}}(\frac{1}{\epsilon^2})$ of $\cal{SFO}$ complexity, using variable sample-size. Numerical results on the application arising from the distributionally robust optimization demonstrate the effectiveness and efficiency of the proposed SSAG method.}

	\keywords{Smoothing method, Stochastic approximation, Accelerated gradient method, Constrained convex stochastic programming, Complexity, Distributionally robust optimization}
	
	%%\pacs[JEL Classification]{D8, H51}
	
	%%\pacs[MSC Classification]{	46N10, %Applications in optimization, convex analysis, mathematical programming, economics
		%	57R12, %Smooth approximations
		%	90C15%Stochastic programming}
	
	\maketitle
	
	\section{Introduction}\label{sec:introduction}
	In this paper, we develop and analyze a novel stochastic smoothing accelerated gradient (SSAG) method for the following general nonsmooth convex composite optimization problems
	\begin{eqnarray}\label{orip}
		\psi^{\rm{opt}}:=\min\limits_{x\in X}\left\lbrace \psi(x): = f(x) + h(x)\right\rbrace, 
	\end{eqnarray}
	where the assumptions on the feasible set and underlying functions are 
	\begin{itemize}
		\item $X$ is a closed convex set in the Euclidean space $\mathbb{R}^{d}$; 
		\item $f: X\rightarrow\mathbb{R}$ is a convex function and its gradient is Lipschitz continuous with  constant $L_f$;
		\item $h: X\rightarrow\mathbb{R}$ is a general continuous nonsmooth convex function that allows a smoothing function satisfying Definition \ref{smoothingdefinition} given in {Sect.} 2.
		\item $\psi(\cdot)$ is well defined and finite valued in $X$, and \eqref{orip} has at least one global minimizer $x^{\rm{opt}}$ and $\psi^{\rm{opt}} = \psi(x^{\rm{opt}})$ as its optimal value.
	\end{itemize}
	Problem (\ref{orip}) encapsulates a broad range of nonsmooth optimization problems. In this paper, we mainly focus on $h$ of the following forms, and show that various existing smoothing techniques can generate smoothing functions that satisfy Definition \ref{smoothingdefinition} and the induced smoothing function of $\psi$ owns a stochastic gradient that satisfies Assumption \ref{assumption 1} in {Sect.} 2 for guaranteeing complexity results of the SSAG method. In fact, the additive composite of the three types of $h$ also allows smoothing functions and Assumption \ref{assumption 1} as well. Here we just omit it for conciseness.
	\begin{itemize}
		\item The nonsmooth term $h$ can be the expectation of nonsmooth closed proper convex functions $\mathbf{H}(\cdot, \xi)$ as
		\begin{eqnarray}\label{orip-2}
			h(x):={\mathbb{E}}_{\xi}
			\left[\mathbf{H}(x, \xi)\right],
		\end{eqnarray}
		where $\xi$ follows a certain discrete or continuous probability distribution on the support set $\Xi$. Here $h$ does not need to have an easily obtainable proximal operator.
		
		\item The nonsmooth term $h$ can be the maximum of finite but numerous convex nonsmooth functions as
		\begin{equation}\label{max}
			h(x):=\max \left\lbrace h_{i}(x),~i\in \mathbb{I}_{q}:=\{1,\ldots,q\}\right\rbrace ,
		\end{equation}
		where $q$ can be a large integer greater than thousands, $h_{i}$ is a convex but possibly nonsmooth function over $X$. It is worth emphasizing that the index 
		$i \in \mathbb{I}_q$ in our SSAG method will be considered as a random variable.
		
		\item {The nonsmooth term $h$ can also be in the form that  combines  
			\eqref{orip-2} and \eqref{max}:
			\begin{equation}\label{Emax}
				h(x):=\mathbb{E}_{\xi}
				\left[ \max\limits_{i\in \mathbb{I}_{q}} \left\lbrace h_{i}(x,\xi)\right\rbrace\right],
			\end{equation}
			where $h_{i}(x,\xi)$ involves both the random vector $\xi$ and the random variable $i$.
		}
	\end{itemize}
	
	The three choices of $h$ and their additive composites make problem \eqref{orip} not narrow, and covers a lot of real applications arising in various problems, including machine learning and distributionally robust optimization (DRO) \cite{chapelle2008optimization,shivaswamy2008relative,li2020fast,crammer2001algorithmic}. 
	Problem \eqref{orip} is challenging mainly for two reasons. First, the nonsmooth term $h$ may be rather complex. For instance, $\mathbf{H}(\cdot,\xi)$ in \eqref{orip-2} has no explicit proximal operator, or the number of $h_{i}(x)$ in \eqref{max} or the number of $h_i(x,\xi)$ in \eqref{Emax} is very large, and $h_{i}(x)$ is highly nonsmooth for some $i\in \mathbb{I}_q$. Second, the objective value and/or the subgradient are hard or expensive to obtain. 
	
	To address the first difficulty mentioned above, smoothing techniques for the objective function and smoothing methods are promising \cite{nesterov2005smooth,chen2012smoothing,chen1996class,zhang2020smoothing,zhang2009smoothing,xu2014smoothing,xu2015smoothing}. The smoothing techniques 
	for deterministic nonsmooth convex programming problems are well developed, including the integral convolution \cite{chen2012smoothing}, the Nesterov's smoothing technique \cite{nesterov2005smooth},  the inf-conv smoothing approximation \cite{beck2012smoothing}, as well as the randomized smoothing (RS) technique \cite{lakshmanan2008decentralized,yousefian2012stochastic, duchi2012randomized}.
	
	To tackle the second difficulty, the stochastic approximation (SA) algorithms are appropriate. Ever since the pioneering work \cite{robbins1951stochastic}, SA algorithms have attracted much attention and well developed \cite{polyak1990new,polyak1992acceleration,nemirovski2009robust,lan2012optimal}.  Complexity results are important for SA algorithms in terms of either the number of ${\cal SFO}$ or the number of iterations to get an $\epsilon$-approximate solution. Recall that we say $y_k \in X$ is an $\epsilon$-approximate solution of \eqref{orip}, if
	\begin{eqnarray}\label{epsilon-approximate}
		\mathbb{E}\left[\psi(y_k)-\psi^{\rm opt}\right]\leq\epsilon,
	\end{eqnarray}
	where the expectation is taken to ${\cal F}_k$, the history of randomness up to the $k$-th {iteration}. For the nonsmooth stochastic convex optimization, the order of $\mathcal{SFO}$ complexity required to find an $\epsilon$-approximate solution for a predetermined accuracy $\epsilon>0$, cannot be smaller than $\mathcal{O} \left(\frac{1}{\epsilon^2}\right)$ for a first-order method,
	as pointed out in \cite{nemirovski2009robust}. Moreover, for first-order SA algorithms to solve the nonsmooth stochastic convex optimization, the best-known order of iteration complexity for finding an $\epsilon$-approximate solution is $\mathcal{O} \left(\frac{1}{\epsilon}\right)$ obtained by the smoothed variable sample-size accelerated proximal scheme (sVS-APM); see Theorem 4 of \cite{jalilzadeh2022smoothed} for detail. At the $k$-th {iteration}, the corresponding variable sample size is $\lfloor  k^{1+\delta}\rfloor$, where  $\delta$ can be an arbitrary positive real number, and the order of ${\cal SFO}$ complexity is ${\cal{O}}(\frac{1}{\epsilon^{2+\delta}})$ that can be arbitrarily near optimal. 
	
	Several SA algorithms such as RSPG \cite{ghadimi2016mini} and AM-SGD \cite{zhou2020amortized} can solve nonsmooth stochastic convex composite problems \eqref{orip}-\eqref{orip-2}, but they require that the nonsmooth
	term $h$ is relatively simple, e.g., $h(x) = \|x\|_1$, so that the proximal operator can be obtained easily. Recently, variance reduction techniques have been enrolled in SA algorithms. The popular stochastic variance reduction methods for nonsmooth problems such as SAGA \cite{defazio2014saga}, SVRG \cite{reddi2016stochastic}, SARAH \cite{nguyen2017sarah}, IPSG \cite{wang2017inexact}, and Katyusha \cite{allen2017katyusha} also require the proximal operator to be easily computed.

	For solving (\ref{orip})-(\ref{orip-2}) that the proximal operator can not be obtained easily, there exist several excellent works on combining smoothing techniques and SA methods. The RS method proposed by Duchi et al. in \cite{duchi2012randomized} is one of the pioneering works, which combines the RS technique, with Tseng's accelerated gradient method \cite{tseng2008accelerated} for deterministic smooth optimization problems, and achieves the best-known ${\cal SFO}$ complexity results. In the RS method, the smoothing parameter is flexible, and can either be diminishing or fixed. However, the smoothing techniques used by the RS method are dimension-dependent, and the complexity bounds in ${\cal{SFO}}$ ({Theorems 2.1, 2.2, and Corollaries 2.3-2.5} of \cite{duchi2012randomized}) are dimension-dependent, which may cause the RS method time-consuming for high-dimensional problems. Wang et al. recently proposed a stochastic Nesterov's smoothing accelerated (SNSA) method in \cite{wang2022stochastic} that can achieve the best-known order of $\cal{SFO}$ complexity. The Nesterov's smoothing technique in \cite{nesterov2005smooth} is employed, which requires that for each fixed $\xi\in \Xi$, ${\mathbf{H}}(\cdot,\xi)$ has the linear max structure.
	In contrast to the RS method, the Nesterov's smoothing technique and the complexity results of the SNSA method are dimension-independent, which is attractive for high-dimensional problems. A predetermined and fixed smoothing parameter is adopted in the SNSA method. Because of the fixed smoothing parameter, however, the SNSA method can provide approximate solutions at best but not asymptotically exact solutions. {Jalilzadeh et al. \cite{jalilzadeh2022smoothed} employed the concept of smoothable function in \cite{beck2012smoothing,beck2017first} that allows general smoothing techniques, and introduced the sVS-APM algorithm. The sVS-APM algorithm adopts variable sample-size, i.e., increasing batch sizes along iterations, which significantly improves the iteration complexity.}
	
	As far as we know, all existing SA-type methods concentrate on (\ref{orip})-(\ref{orip-2}), and when the existing SA-type methods solve the composite optimization problem (\ref{orip}) with the nonsmooth term $h$ being in (\ref{max}), or (\ref{Emax}) as \cite{duchi2012randomized,jalilzadeh2022smoothed}, they all consider the maximum of finite convex functions as a deterministic function. In this paper, we develop an SA-type method, SSAG, that can tackle (\ref{orip}) where the nonsmooth term $h$ is flexible that can be in the forms of (\ref{orip-2}), (\ref{max}) and (\ref{Emax}). {It is worth emphasizing that it is the first time to consider the index $i$ of $h_i(x)$ in \eqref{max} or $h_i(x,\xi)$ in \eqref{Emax} as a random variable $i\in \mathbb{I}_q$ for $q$ very large.} The SSAG method combines a general smoothing technique and a modified version of the accelerated gradient (AG) method {(Algorithm 1 of \cite{ghadimi2016accelerated})}, so that it can solve general constrained nonsmooth convex composite optimization problems,  instead of the nonsmooth composite optimization problems with a simple nonsmooth term that has easily obtainable proximal operator in \cite{ghadimi2016accelerated}. The main scheme of the modified AG scheme follows that in \cite{ghadimi2016accelerated}, but we adopt different stepsizes in the two gradient descent steps and use different parameters for the convex combination step.
	
	The contributions of this paper are summarized as follows, comparing with existing works.
	\begin{itemize}
		\item We give the definition of smoothing function in Definition \ref{smoothingdefinition}, and the basic assumption for stochastic gradients in Assumption \ref{assumption 1} of {Sect.} \ref{sec:techniques} that will be used throughout the paper. We show that under mild conditions, the smoothing approximations of $h$ in the form of \eqref{orip-2}, or \eqref{max}, or \eqref{Emax} constructed by various smoothing techniques are smoothing functions and satisfy the basic assumption.   
		
		\item
		We propose an SSAG method for solving optimization problems \eqref{orip} with a nonsmooth convex component $h$. We do not assume that $h$ has the linear max structure in \eqref{h} or that its proximal operator is easily obtained.  On the contrary, most SA methods for \eqref{orip} require the linear max structure developed from the Nesterov's smoothing \cite{nesterov2005smooth,wang2022stochastic}, or easily computable proximal operator of $h$ that roots from proximal-type methods, see, e.g., \cite{ghadimi2016mini,wang2017inexact,xiao2014proximal}. {More importantly, we adopt for the first time the new point of view to considering the index $i$ of $h_i(x)$ in \eqref{max} or $h_i(x,\xi)$ in \eqref{Emax} as a random variable $i\in \mathbb{I}_q$ for $q$ very large.}
		
		\item
		Any smoothing technique in SSAG can be adopted, as long as the smoothing approximation 
		$\tilde{h}_{\mu}$ for $h$ satisfies Definition \ref{smoothingdefinition} in {Sect.} \ref{sec:techniques} and Assumption \ref{assumption 1} in {Sect.} \ref{sec:techniques}. %We use a decreasing setting for the smoothing parameter, similar to that in \cite{duchi2012randomized}. 
		The sVS-APM developed in \cite{jalilzadeh2022smoothed} is also promising to allow flexible smoothing techniques that generate smoothable functions defined in \cite{beck2012smoothing}. Smoothing function in Definition \ref{smoothingdefinition} of this paper is different to the smoothable function of \cite{beck2012smoothing}.  Moreover, there is no detailed analysis on what kind of smoothing approximations satisfy the requirements in \cite{jalilzadeh2022smoothed} as we provide in this paper. 
		
		\item
		The SSAG method is flexible, which allows variants using fixed or diminishing smoothing parameters, as well as fixed or variable mini-batch sizes of samples per iteration. {We show that all variants of SSAG achieve the best-known $\cal{SFO}$ complexity of finding an $\epsilon$-approximate solution to problem \eqref{orip}. Moreover, the order of iteration complexity is $\mathcal{O}(\frac{1}{\epsilon})$ that is optimal, when the variable mini-batch size of samples $k$ is adopted at the $k$-th {iteration}.} The complexity results of iterations and $\cal{SFO}$ are better than that of the sVS-APM in \cite{jalilzadeh2022smoothed}. The complexity results can be dimension-independent with proper smoothing techniques, better than those dimension-dependent results for the RS method \cite{duchi2012randomized}. 
		
		\item
		We do numerical experiments on an application arising from %the Wasserstein DRSVM \cite{li2020fast} that is a distributionally robust optimization (DRO) with the ambiguity set defined by Wasserstein distance; 
		the DRO-moment problem that is a general DRO with matrix moment constraints (extension of \cite{wang2023distributionally}).
		%; and the stochastic utility problem \cite{jalilzadeh2022smoothed}. 
		We demonstrate the efficiency of the SSAG method by comparing it with several state-of-the-art SA-type methods and deterministic methods. For the composite optimization problem (\ref{orip}) with the nonsmooth term $h$ defined by (\ref{max}) or (\ref{Emax}), 
		{the significant efficiency of SSAG can be attributed to treating the maximum of finite convex functions as a stochastic function of $i\in\mathbb{I}_q$.}
	\end{itemize}

	The remainder of this paper is organized as follows. In {Sect.} \ref{sec:techniques}, we outline some important smoothing techniques and show that the corresponding smoothing approximations satisfy Definition \ref{smoothingdefinition} and Assumption \ref{assumption 1} required in this paper. In {Sect.} \ref{sec:method}, we develop the SSAG method. We show that various smoothing techniques can be enrolled in the method, and show the complexity results with respect to the number of calls to {the} ${\cal SFO}$ and the number of iterations. The numerical experiments are performed in an application of {Sect.} \ref{sec:numerical} to demonstrate the effectiveness and efficiency of our SSAG method.
	
	%\subsection{Notation and terminology}
	{\bf{Notation and terminology.}}
	We denote by $\|\cdot\|$ the Euclidean norm, by $\|\cdot\|_1$ the $\ell_{1}$ norm, and by $\langle \cdot,\cdot \rangle$ the inner product. Given a convex function $\hat{h}$ {defined on a nonempty convex set} $X$, for any $x \in X$, we use $\partial \hat{h}(x)$ to denote the subdifferential of $\hat{h}$ at $x$. We adopt the shorthand notation $\|\partial \hat{h}(x)\| = \sup\left\{\|g\|\ :\  g \in \partial \hat{h}(x)\right\}$. For any real number $a$, we denote by $\lceil a\rceil$ and $\lfloor a \rfloor$ the nearest integer to $a$ from above and below, respectively. We denote by $B(x,r):=\left\lbrace y \in \mathbb{R}^{d} :\|x-y\|\leq r\right\rbrace $ the closed ball of radius $r>0$ centered at $x\in \mathbb{R}^{d}$. The plus function $[x]_+ = \max\{x,0\}$, where the max operator is performed entry-wise. A random vector $u \sim \mathcal{N}\left(0, I_{d}\right)$ means that $u$ follows {the normal distribution} with mean vector being the zero vector $0$ and covariance matrix being the $d$-dimensional identity matrix $I_{d}$. The set $\Delta_{d}:=\left\{ x \in \mathbb{R}^{d}: \sum_{i=1}^d x_i=1,~x\geq 0\right\}$ refers to the $d$-dimensional simplex. Given a square matrix $C$, let $\lambda_{\max}(C)$ denote the largest eigenvalue of $C$. The set of all $d \times d$ positive semidefinite matrices is denoted by $\mathbb{S}_{+}^d :=\left\lbrace C \in \mathbb{R}^{d \times d}: C \succeq 0\right\rbrace$.
	
	A function $\hat{f}$ is $\hat{L}_{0}$-Lipschitz over $X$ if
	\begin{eqnarray*}
		|\hat{f}(x)-\hat{f}(y)| \leq \hat{L}_{0}\|x-y\| \quad \text { for any } x, y \in X.
	\end{eqnarray*}
	We say $\hat{f}$ is $L_{\hat{f}}$-smooth over $X$, i.e., $\nabla \hat{f}$ is  Lipschitz with constant $L_{\hat{f}}$, if
	\begin{eqnarray}\label{L-smooth}
		\|\nabla \hat{f}(x) - \nabla \hat{f}(y)\| \leq L_{\hat{f}} \|x-y\|\quad \mbox{for all}\ x, y \in X.
	\end{eqnarray}
	Since $\hat{f}$ is convex, by Theorem 5.8 of \cite{beck2017first}, $\hat{f}$ is $L_{\hat{f}}$-smooth over $X$ is equivalent to
	\begin{eqnarray}\label{psiLsmoo}
		\hat{f}(y)\leq \hat{f}(x)+\left\langle \nabla \hat{f}(x),y-x\right\rangle+\frac{L_{\hat{f}}}{2}\|y-x\|^2 \quad \text { for any } x, y \in X.
	\end{eqnarray}
	{We say a proper function $\hat{f}$ is $\sigma_{\hat{f}}$-strongly convex,  if for any $x, y \in\mathbb{R}^d$ and $\lambda\in (0,1)$,
		\begin{eqnarray}\label{strongconv}
			\hat{f}(\lambda x+(1-\lambda)y)\leq \lambda\hat{f}(x)+(1-\lambda)\hat{f}(y)-\frac{\sigma_{\hat{f}}}{2}\lambda(1-\lambda)\|x-y\|^2 .
		\end{eqnarray}
	}
	For a function $\hat{g}: \mathbb{R}^{d} \to (-\infty,\infty]$, its convex conjugate $\hat{g}^*: \mathbb{R}^{d} \to (-\infty,\infty]$ is defined by
	\begin{eqnarray}\label{conjugate}
		\hat{g}^*(y) = \sup_{x\in \mathbb{R}^{d}} \left\{\langle x, y\rangle - \hat{g}(x)\ :\ x\in {\operatorname{dom} \hat{g}}\right\},
	\end{eqnarray}
	where $\operatorname{dom} \hat{g}=\left\{x\in\mathbb{R}^{d}\ :\  \hat{g}(x)<+\infty\right\}$ is the effective domain of $\hat{g}$.
	
	\section{Smoothing functions}\label{sec:techniques}
	
	There exist several smoothing techniques to construct smoothing approximations of the original nonsmooth function. Throughout the paper, a smoothing approximation is called a smoothing function if it satisfies the following definition.
	\begin{definition}\label{smoothingdefinition}
		Given a convex function $c: X  \rightarrow\mathbb{R}$, we call $\tilde{c}: X\times (0,\bar \mu] \rightarrow\mathbb{R}$ a smoothing function of a convex function $c$ with parameters $(\kappa,K,{\cal {L}}_c)$, if for any $\mu\in(0,\bar{\mu}]$, the function $\tilde{c}(\cdot,\mu)$ is continuously differentiable in $X$ and satisfies the following conditions:
		\newline
		(a) $\lim\limits_{z \rightarrow x,~ \mu \downarrow 0} \tilde{c}(z,\mu) = c(x),~~\forall x \in X$;\newline
		(b) (convexity) $\tilde{c}(\cdot,{\mu})$ is convex on $X$;\newline
		%(iii) (gradient consistency) $\left\lbrace \lim\limits_{z \rightarrow x, \mu \downarrow 0} \nabla \tilde{h}_{\mu}(z)\right\rbrace \subseteq \partial h(x), %\forall x \in X$;\newline
		(c) (Lipschitz continuity with respect to $\mu$) there exists a constant $\kappa>0$ such that
		\begin{eqnarray}\label{hmurelation}
			\left\lvert \tilde{c}(x,{\mu_2})-\tilde{c}(x,{\mu_1})\right\rvert \leq \kappa\left\lvert\mu_{1}-\mu_{2}\right\rvert,~~\forall x \in X,~~~\forall \mu_{1},\mu_{2} \in(0,\bar{\mu}];
		\end{eqnarray}
		(d) (Lipschitz smoothness with respect to $x$) there exist constants $K$ and ${\cal L}_{c}>0$ irrelevant to $\mu$ such that $\tilde{c}(\cdot,{\mu})$ is $L_{\tilde{c}_{\mu}}$-smooth on $X$ with gradient Lipschitz constant $L_{\tilde{c}_{\mu}}= K+\frac{{\cal L}_{c}}{\mu}$.
	\end{definition}
	
	Later, we also use $\tilde c_{\mu}$ to represent $\tilde c(\cdot,\mu)$ and say $\tilde c_{\mu}$ is a smoothing function  of $c$ with parameters  $(\kappa, K, {\cal{L}}_c)$ for brevity.  When we do not need to specify the parameters $(\kappa, K, {\cal{L}}_c)$, we just say $\tilde{c}_{\mu} $ is a smoothing function of $c$ for simplicity.  Various smoothing techniques can generate smoothing approximations of the original nonsmooth function that satisfy Definition \ref{smoothingdefinition} (a). More conditions are needed for a smoothing approximation to be a smoothing function in Definition \ref{smoothingdefinition}. 
	
	\vskip 2mm
	\noindent{\bf{Remark 1.}}\quad	%\begin{remark}\label{splus}
	%
	%(a) 
	%By the definition of $L$-smoothness \eqref{L-smooth}, it is clear that if a function is $L_{1}$-smooth, then it is also $L_{2}$-smooth for any $L_{2}\geq L_{1}$.
	%It is clear that $(\kappa,K,{\cal{L}}_c)$ in Definition \ref{smoothingdefinition} can be substituted by an arbitrary upper bound $(\hat{\kappa}, \hat{K}, \hat{\cal{L}}_c)$ of $(\kappa,K,{\cal{L}}_c)$. This is helpful to find out the parameters that can be easily computed. 
	%
	%(b)
	%
	Let $\tilde{c}_{1,\mu}$ and $\tilde{c}_{2,\mu}$ be smoothing functions of two convex functions $c_1$ and $c_2$ with parameters $\left( \kappa_1, K_1, {\cal L}_{c_1}\right)$ and $\left( \kappa_2, K_2, {\cal L}_{c_2}\right)$, respectively. Then it is easy to show that  $\gamma_1\tilde{c}_{1,\mu}+\gamma_2\tilde{c}_{2,\mu}$ is a smoothing function of $\gamma_1 c_1+ \gamma_2 c_2$ with parameters $\left(\gamma_1 \kappa_1+\gamma_2\kappa_2,~\gamma_1K_1+\gamma_2K_2,~\gamma_1{\cal L}_{c_1}+\gamma_2{\cal L}_{c_2}\right)$, for any constants $\gamma_1,~ \gamma_2>0$.
	%\end{remark}

	Definition \ref{smoothingdefinition} is closely related to the smoothing function given in Definition 3.1 of \cite{bian2020smoothing} and the smoothable function given in  Definition 2.1 of \cite{beck2012smoothing}.	 %{Definition 3.1 of \cite{bian2020smoothing} requires all the above conditions in Definition \ref{smoothingdefinition} hold for any $\mu >0$ where $\bar \mu$ is a positive constant. 
		%Here Definition \ref{smoothingdefinition} requires the conditions hold for any $\mu>0$ because the smoothing approximations outlined later satisfy this easily. } 
	Besides, Definition 3.1 of \cite{bian2020smoothing} requires the gradient consistency property that
	\begin{eqnarray}
		\label{gradient-cons}
		G_{\tilde{c}}(x):=\left\{ \lim_{z \rightarrow x,~ \mu \downarrow 0} \nabla\tilde{c}_{\mu}(z)\right\} \subseteq \partial c(x)\quad \text{for~any}~x \in X.
	\end{eqnarray}
	%Since \eqref{orip} is convex,
	The $\epsilon$-approximate solution of this paper is measured by \eqref{epsilon-approximate}, using only objective values.
	%	the difference between the expectation of objective value at the computed %solution and the optimal value,
	% $$\mathbb{E}\left[\psi(y_{N})-\psi^*\right],$$
	%	where the expectation is taken to the history of randomness up to the last iteration.
	Consequently, here we do not require a smoothing function to satisfy the gradient consistency property. Moreover, Definition \ref{smoothingdefinition} (d) is the same as Definition 2.1 (ii) of \cite{beck2012smoothing}, while the parameter $K=0$ in Definition 3.1 (iv) of \cite{bian2020smoothing}. The smoothable function in Definition 2.1 of \cite{beck2012smoothing} does not require Definition \ref{smoothingdefinition} (c) of this paper. In fact, Definition \ref{smoothingdefinition} (c) will be useful to develop the complexity results of our SSAG method.

	\subsection{Smoothing approximations}\label{sec2.1}
	Below we outline several smoothing techniques, and 
	make use of them to construct smoothing approximations for the nonsmooth terms %$h(x):= \mathbb{E}[{\mathbf H}(x,\xi)]$
	$h$ in \eqref{orip-2},
	% and $h(x):=\max\{h_{\xi}(x),~\xi\in \Xi=\{1,\ldots,q\}\}$
	\eqref{max} and \eqref{Emax}, respectively.
	We will show in {Sect.} \ref{Smoothing properties} that under mild conditions the smoothing approximations constructed by those techniques satisfy Definition \ref{smoothingdefinition}. %We
	\vskip 2mm

	\noindent {\textbf{--Nesterov's smoothing for \eqref{orip-2}}}
	
	The Nesterov's smoothing technique for a deterministic nonsmooth function was introduced in \cite{nesterov2005smooth}. Now we extend it to a stochastic nonsmooth function. Assume that $\mathbf{H}(\cdot,\xi)$ has the max linear structure 
	%in \eqref{h} 
	for  $\xi \in \Xi$ almost everywhere (a.e.), i.e., 
	\begin{eqnarray}\label{h}
		{\mathbf{H}}(x ,\xi) =\max\limits_{u\in U} \left\{\left\langle A_{\xi}x, u\right\rangle - Q_{\xi}(u) \right\},
	\end{eqnarray}
	where $U$ is a bounded closed convex set, $A_{\xi}$ is a linear operator, and $Q_{\xi}$ is a continuous convex function.
	
	Then the Nesterov's smoothing technique \cite{nesterov2005smooth} can be used to construct a smoothing approximation
	$\tilde{h}_{\mu}(x):= {\mathbb{E}}[\tilde{\mathbf{H}}_{\mu}(x,\xi)]$ of $h$ in \eqref{orip-2} as follows.
	By inserting a nonnegative, continuous and $\sigma_d$-strongly convex function $d(u)$ in \eqref{h}, we obtain
	\begin{eqnarray}\label{Nessmooth}
		\tilde{\mathbf{H}}_{\mu}(x,\xi):=\max\limits_{u\in U}\left\lbrace \langle A_{\xi}x,u\rangle-Q_{\xi}(u)-\mu d(u)\right\rbrace.
	\end{eqnarray}
	According to Theorem 1 of \cite{nesterov2005smooth}, the gradient of $\tilde{\mathbf{H}}_{\mu}(x,\xi)$ has the formula
	\begin{eqnarray}\label{Nesterov-grad}
		\nabla \tilde{\mathbf{H}}_{\mu}(x,\xi) = A_{\xi}^{T} \hat{u}_{\mu}(x,\xi),
	\end{eqnarray}
	where $\hat{u}_{\mu}(x,\xi)$ is the unique optimal solution of the maximization problem in \eqref{Nessmooth}.
	% A smoothing function $\tilde{h}_{\mu}(x)$ is then defined as
	% \begin{eqnarray}
		% \tilde h_{\mu}(x):= \mathbb{E}_{\xi}[\tilde{\mathbf{H}}_{\mu}(x,\xi)].
		% \end{eqnarray}
	The linear max structure of $\mathbf{H}(\cdot,\xi)$ in \eqref{h} and easily obtainable $\hat{u}_{\mu}(x,\xi)$ for a.e. $\xi\in \Xi$ are essential for the applicability of Nesterov's smoothing technique.

	\vskip 2mm
	
	\noindent{
		\textbf
		{--Randomized smoothing for \eqref{orip-2}, \eqref{max}, and \eqref{Emax}}}

	The RS technique has attracted much attention and {has been well employed} in \cite{lakshmanan2008decentralized,yousefian2012stochastic, duchi2012randomized}. For $h$ in \eqref{orip-2} involving expectation, Duchi et al. in \cite{duchi2012randomized} employ the RS technique that combines the convolution smoothing technique and the random sampling technique to get a smoothing approximation. By introducing an auxiliary random vector $v\in \mathbb{R}^{d}$ with density function $\rho$, a smoothing function $\tilde{h}_{\mu}(x)$ is defined as
	\begin{eqnarray}\label{RSsmooth}
		\tilde{h}_{\mu}(x):= \mathbb{E}_{v,\xi}\left[ \mathbf{H}(x+\mu v,\xi)\right] .
		%= \int_{\mathbb{R}^{d}} \int h(x+\mu v) \rho(v) d v.
	\end{eqnarray}
	Here, the expectation $\mathbb{E}$ is taken with respect to both the original random vector $\xi\in \Xi$ and the auxiliary random vector $v\in\mathbb{R}^{d}$. {The auxiliary random vector $v$ is often chosen to follow the uniform distribution on $B(0,1)$ or the normal distribution ${\cal N}(0,I_d)$.}  At iteration $k$, the RS technique queries $\xi_{i,k}$ and  the oracle  at $m$ points $x_k + \mu_k v_{j,k}$ drawn randomly from some neighborhood of $x_k$, where $\xi_{j,k}$ and $v_{j,k}$, $j=1,\ldots,m$ are independent and identically distributed (i.i.d.) samples drawn according to the distributions for $\xi$ and $v$, respectively. The vector $g_k$ is used as the approximation of stochastic gradient
	\begin{eqnarray}\label{radomizedgrad}
		g_{k}=\frac{1}{m} \sum\limits_{j=1}^{m} g_{j,k}, \quad\mbox{where}\quad g_{j,k}\in \partial {\mathbf{H}}\left(x_{k}+ \mu_{k}v_{j, k},\xi_{j, k}\right).
	\end{eqnarray}
	Here $\partial {\mathbf{H}}\left(x_{k}+ \mu_{k}v_{j, k},\xi_{j, k}\right)$ is the subdifferential of the convex function ${\mathbf{H}}(\cdot,\xi_{j,k})$ at the point $x_{k}+ \mu_{k}v_{j, k}$.
	%For $h$ , %it does not meet the expectation form required in \cite{duchi2012randomized}. 
	{The RS technique is very easy to implement, and has very little restriction on the original nonsmooth functions.} {One shortcoming is that the function in \eqref{RSsmooth} involves high-dimensional integration if the problem is itself high-dimensional, which may cause computational inefficiency. }

	It is worth mentioning that as in \cite{yousefian2012stochastic}, when the maximum of finite convex functions in \eqref{max} or \eqref{Emax} is considered as a deterministic function,   the above RS technique can be adopted  to obtain the smoothing functions for \eqref{max} and \eqref{Emax}, respectively, as follows:
	\begin{eqnarray}\label{RSmax}
		\tilde{h}_{\mu}(x)=\mathbb{E}_{v}\left[ \max \left\lbrace h_{i}(x+\mu v),~i\in \mathbb{I}_q\right\rbrace\right],
	\end{eqnarray}
	and 
	\begin{eqnarray}\label{RSEmax}
		\tilde{h}_{\mu}(x)=\mathbb{E}_{v,\xi}\left[ \max \left\lbrace h_{i}(x+\mu v,\xi),~i\in \mathbb{I}_q\right\rbrace\right]. 
	\end{eqnarray}
	In this way of smoothing for \eqref{max} or \eqref{Emax}, however, the index $i$ is not treated as a random variable. Fortunately, the following inf-conv smoothing technique will provide smoothing approximations of \eqref{max} and \eqref{Emax} where $i$ can be considered as a random variable as we desired in this paper.

\vskip 2mm

\noindent{\textbf{--Inf-conv smoothing for \eqref{orip-2}, \eqref{max}, and \eqref{Emax}}}

The inf-conv smoothing technique for a deterministic nonsmooth function has been intensively studied in section 4 of \cite{beck2012smoothing}. For $h$ in \eqref{orip-2} involving expectation, we define
its inf-conv smoothing approximation as $\tilde{h}_{\mu}(x):={\mathbb{E}}[{\mathbf{\tilde H}}_{\mu}(x,\xi)]$, in which ${\mathbf{\tilde H}}_{\mu}(x,\xi)$ is constructed according to Definition 4.2 of \cite{beck2012smoothing} by
\begin{eqnarray}\label{inf-conv-H}
\tilde{\mathbf{H}}_{\mu}(x,\xi)=\inf\limits_{y \in \mathbb{R}^{d}}\left\{{\mathbf{H}(y,\xi)}+\mu \omega \left(\frac{x-y}{\mu}\right)\right\}, 
%\label{inf-conv}
\end{eqnarray}
where $\omega: \mathbb{R}^{d} \rightarrow \mathbb{R}$ 
is convex and 
$\frac{1}{\sigma_{\omega}}$-smooth. When the inf-conv smoothing technique is used, we always assume that for any 
$\mu\in(0,\bar{\mu}]$, a.e. $\xi\in \Xi$ and any $x\in \mathbb{R}^{d}$, ${\tilde{\mathbf{H}}}_{\mu}(x,\xi)$ is finite 
%{and
% the infimum in \eqref{inf-conv-H} can be obtained at a point
%$\hat y(x) \in %\mathbb{R}^{d}$}.
and
$\inf_{x\in \mathbb{R}^{d}} {\bf{H}}(x,\xi) 
> -\infty$.
%for a.e. $\xi\in \Xi$. 
{This assumption is satisfied, when $\omega$ has bounded level sets and $\mathbf{H}(\cdot,\xi)$ is proper. When the quadratic function $\omega(x)=\frac{1}{2}\|x\|^2$ is chosen, the inf-conv smoothing technique reduces to the well-known Moreau smoothing technique (see {Sect.} 6.7 of \cite{beck2017first})}. By Theorem 4.1 of \cite{beck2012smoothing}, $\tilde{\mathbf{H}}_{\mu}$ has a ``dual" formulation
\begin{eqnarray}\label{inf-conv-reform}
\tilde{\mathbf{H}}_{\mu}(x,\xi) = \max\limits_{y\in \mathbb{R}^{d}}\left\lbrace \langle y,x\rangle - {\mathbf{H}^*(y,\xi) - \mu \omega^*(y)} \right\rbrace,
\end{eqnarray}
where $\mathbf{H}^*(\cdot,\xi)$ and $\omega^*(\cdot)$ are the convex conjugates of $\mathbf{H}(\cdot,\xi)$ and $\omega(\cdot)$, respectively.
Moreover, $\tilde{\mathbf{H}}_{\mu}(\cdot,\xi)$ is differentiable with gradient
$\nabla \tilde{\mathbf{H}}_{\mu}(\cdot,\xi)$ of the form
\begin{eqnarray}\label{grad-inf-conv}
\nabla \tilde{\mathbf{H}}_{\mu}(x,\xi) = \nabla \omega\left(\frac{x-\hat{v}_{\mu}(x,\xi)}{\mu} \right),
\end{eqnarray}
where $\hat{v}_{\mu}(x,\xi)$ is a minimizer of the right-hand side of the inf problem in \eqref{inf-conv-H}.
For $h$ in \eqref{max} that involves the maximization of finite convex functions, we can obtain the smoothing approximation $\tilde h_{\mu}$ of $h$ by Examples 4.4 and 4.9 of \cite{beck2012smoothing}, with the aid of inf-conv smoothing technique. To be specific, let $b: \mathbb{R}^q \to \mathbb{R}$ and $\omega: \mathbb{R}^q \to \mathbb{R}$ be of the form
\begin{eqnarray}\label{b-omega}
b(z) = \max\left\{z_1,\ldots,z_q\right\},\quad \omega(z)= \ln\left(\sum_{i=1}^q e^{z_i}\right),
\end{eqnarray}
where $z=\left(z_1,\ldots,z_q\right)^T$. Then the gradient $\nabla \omega$ of $\omega$ is
\begin{eqnarray}\label{gradient-omega}
\nabla \omega(z) = \frac{1}{\sum_{i=1}^q e^{z_i}} (e^{z_1},\ldots,e^{z_q})^{T} \quad\text{with}\quad \|\nabla \omega(z)\|\leq 1,
\end{eqnarray}
and the convex conjugate $\omega^*$ of $\omega$ is
\begin{eqnarray}\label{2.11}
\omega^*(s) = \sum_{i=1}^q s_i \ln s_i \quad \mbox{with}\quad \operatorname{dom} \omega^* = \Delta_q.
\end{eqnarray}
The inf-conv smoothing approximation $\tilde{b}_{\mu}(z)$ can be expressed as
\begin{eqnarray}\label{bt34}
\tilde b_{\mu}(z) &=& \inf_{s\in \mathbb{R}^q}
\left\{b(s) + \mu \omega\left(\frac{z-s}{\mu}\right) \right\}
= \max\limits_{s\in \mathbb{R}^q}
\left\{\langle s, z \rangle - b^*(s) - \mu \omega^*(s)\right\}\nonumber\\
&=& \mu \omega\left(\frac{z}{\mu}\right)
= \mu \ln \left(\sum_{i=1}^q e^{z_i/\mu}\right). 
\end{eqnarray}
Interestingly, $\tilde b_{\mu}(z)$ in \eqref{bt34} is the same as the so-called Neural Networks smoothing function for $b(z)$ in \eqref{b-omega} constructed by convolution and mathematical induction in \cite{chen2012smoothing}.

%
%\begin{remark}
%	It is worth noting that Example 4.9 of \cite{beck2012smoothing} needs the feasible set $X$ compact, in order to
%	{guarantee the gradients $\{\nabla \tilde{h}_{\xi,\mu}(x):x\in X\}$ are bounded for $\xi\in\Xi=\{1,\ldots,q\}$.}
%	 In this paper, we do not require $X$ to be compact {but directly require the boundedness of $\{\nabla \tilde{h}_{\xi,\mu}(x):x\in X\}$ for $\xi\in\Xi$.}

With the aid of  \eqref{bt34}, we provide a smoothing approximation of (\ref{max}) as follows. This approximation is not directly deduced from the inf-conv smoothing technique, but the inf-conv smoothing technique plays an important role to construct it. Let {$\tilde h_{i,\mu}$} be a smoothing approximation of $h_{i}$ for each $i\in \mathbb{I}_q$ constructed by a certain smoothing technique. Then the smoothing approximation $\tilde h_{\mu}$ of $h$ defined in \eqref{max} can be constructed by 
\begin{eqnarray}\label{inf-conv-max}
\tilde h_{\mu}(x) = \mu \ln \left( \sum_{i=1}^q e^{\frac{\tilde{h}_{i,\mu}(x)}{\mu}} \right).
\end{eqnarray}
By direct computation, the gradient
$\nabla \tilde{h}_{\mu}(x)$ is
\begin{eqnarray}\label{maxgrad}
\nabla \tilde{h}_{\mu}(x) = \sum_{i=1}^q
p_{x,\mu}(i) \nabla \tilde{h}_{i,\mu}(x),
\quad \mbox{with}\quad
p_{x,\mu}(i)=\frac{e^{{\tilde{h}_{i,\mu}(x)}/{\mu}}}{\sum_{j=1}^q e^{{\tilde{h}_{j,\mu}(x)}/{\mu}}}.
\end{eqnarray}
%%{Here $\xi$ is considered as a random index whose %%probability is $p_{x,\mu}(\xi)$ over the support set $\Xi=\{1,\ldots,q\}$.}
{In Proposition 4.1 of 
\cite{beck2012smoothing}, they considered the similar smoothing approximation of \eqref{max} where $h_i$ is smooth for each $i\in \mathbb{I}_q$. Here we do not require $h_i$ to be smooth and this will be useful in DRO with ambiguity set defined by moment constraints outlined in {Sect.} \ref{subsec:DROportfolio}, when the objectives have nonsmooth terms reflecting sparse or risk averse requirements. }

For $h$ in (\ref{Emax}), let $\tilde h_{i,\mu}(\cdot,\xi)$ be a smoothing approximation of $h_{i}(\cdot,\xi)$ for each $i\in \mathbb{I}_q$ constructed by a certain smoothing technique. Then the smoothing approximation $\tilde h_{\mu}$ of $h$ defined in \eqref{Emax} can be constructed by 
\begin{eqnarray}\label{inf-conv-Emax}
\tilde h_{\mu}(x) = \mathbb{E}_{\xi}\left[ \mu \ln \left( \sum_{i=1}^q e^{\frac{\tilde h_{i,\mu}(x,\xi)}{\mu}} \right)\right].
\end{eqnarray}
%where the expectation $\mathbb{E}$ is taken with respect to the random vector $\xi\in \Xi$
%%and random variable $i\in \mathbb{I}_q$ 
%with given probability density.

For a fixed $\xi\in\Xi$, we view 
\begin{eqnarray}
\label{pxmu_xii}
p_{x,\mu}^{\xi}(i) := \frac{e^{{\tilde h_{i,\mu}(x,\xi)}/{\mu}}}{\sum_{j=1}^q e^{{\tilde h_{j,\mu}(x,\xi)}/{\mu}}}
\end{eqnarray}
as a probability of $i$ over the support set $\mathbb{I}_q$. Then we have
\begin{eqnarray}\label{maxgrad1}
\nabla \tilde{h}_{\mu}(x) = \mathbb{E}_{\xi}\left[ 
\sum_{i=1}^q p_{x,\mu}^{\xi}(i) \nabla \tilde h_{i,\mu}(x,\xi)\right]
%\quad \text{with}\quad
%p_{x,\mu}^{\xi}(i)=, \nonumber\\ 
%& =&
= \mathbb{E}_{\xi,i}\left[ \nabla \tilde h_{i,\mu}(x,\xi)\right].
\end{eqnarray}
%where the expectation $\mathbb{E}$ is taken with respect to the random vector $(\xi^T,i)^T \in\Xi \times \mathbb{I}_q$. 

%	 The examples in {Sect.} \ref{sec:numerical} have two cases. The first case is that the feasible set is compact, and the result of \cite{beck2012smoothing} can be used directly. The second case is that $h_{i}(\cdot)$ is a linear function, which makes the gradient a finite value. In this case, ${\cal L}_{h}$ can be computed without the compactness of the feasible set $X$.

%\end{remark}
%%%
%%%
%%%
%We will show in Appendix \ref{appendixa} that the smoothing approximations constructed by the three smoothing techniques outlined above satisfy Definition \ref{smoothingdefinition} under mild conditions.
%%%
%%%
%%%
% see Appendix \ref{appendixa} for details.
%%{Lemma \ref{lemma2.1,lemma2.1-2,lemma2.2,lemma2.21,lemma2.3} compute the parameters $(\kappa,K,{\cal L}_{h})$ for which the randomized smoothing for \eqref{orip-2}, the Nesterov's smoothing for \eqref{orip-2}, the inf-conv smoothing for both \eqref{orip-2} and \eqref{max} will satisfy the premises of Definition \ref{smoothingdefinition}. We prove the above lemmas at Appendix \ref{appendixa}.}

\subsection{Stochastic gradients and fundamental assumptions}
Let $\tilde{h}_{\mu}$ be a smoothing approximation of $h$. %satisfying Definition \ref{smoothingdefinition}.
The smooth counterpart of \eqref{orip} is
\begin{eqnarray}\label{smoothp}
\min\limits_{x\in X}\left\lbrace \tilde{\psi}_{\mu}(x): = f(x) + \tilde{h}_{\mu}(x)\right\rbrace.
\end{eqnarray}
%with $\mu = \mu_{k}$ is $L_{\mu_{k}}$-smooth on $X$, where the Lipschitz constant
%\begin{eqnarray}\label{Lmuk}
%	L_{\mu_{k}}=L_f+L_{\tilde{h}_{\mu_{k}}}=L_f+K+ \frac{{\cal L}_{h}}{\mu_{k}}.
%\end{eqnarray}
%According to Theorem 5.12 of \cite{beck2017first}, we get $L_f\geq\|\nabla^2f(x)\|=\lambda_{\max}(\nabla^2f(x))$ for any $x\in X$, where
%%$\lambda_{\max}(\nabla^2f(x))$ means the maximal eigenvalue of
%$\nabla^2f(x)$ is
%the Hessian matrix of $f$ at $x$. }
%Nonsmooth convex functions can be approximated by various smoothing %techniques in \cite{duchi2012randomized,beck2012smoothing}.

For \eqref{orip}-\eqref{orip-2}, a stochastic gradient for the smoothing function $\tilde{\psi}_{\mu}$ at $x$ is
\begin{eqnarray}\label{gradorip-2}
\nabla\tilde{\mathbf{\Psi}}_{\mu}(x,\xi)=\nabla f(x)+\nabla\tilde{\mathbf{H}}_{\mu}(x,\xi),
\end{eqnarray}
where $\xi$ is the random vector following a certain discrete or continuous probability density $p(\xi)$,
%over the support set $\Xi$, 
and $\tilde{\mathbf{H}}_{\mu}(x,\xi)$ is a smoothing function of $\mathbf{H}(x,\xi)$ in \eqref{orip-2}. 
For \eqref{orip} with $h$ being defined in \eqref{max}, a stochastic gradient for the smoothing function is
\begin{eqnarray}\label{gradmax}
\nabla\tilde{\mathbf{\Psi}}_{\mu}(x,i)=\nabla f(x)+\nabla \tilde{h}_{i,\mu}(x),
\end{eqnarray}
where $i$ is a random vector whose probability is $p_{x,\mu}(i)$ in \eqref{maxgrad} for 
%over the support set 
$i\in \mathbb{I}_q$. 
%that is caused by the inf-conv smoothing technique. 
%
For \eqref{orip} with $h$ being defined in \eqref{Emax}, a stochastic gradient for the smoothing function is
\begin{eqnarray}\label{gradEmax}
\nabla\tilde{\mathbf{\Psi}}_{\mu}(x,\xi,i)=\nabla f(x)+\nabla \tilde{h}_{i,\mu}(x,\xi),
\end{eqnarray}
where $(\xi^T,i)^T \in  \Xi \times \mathbb{I}_q$ is a random vector whose probability is $p_{x,\mu}(\xi,i) := p(\xi)p_{x,\mu}^{\xi}(i)$
with $p_{x,\mu}^{\xi}(i)$ being defined in \eqref{pxmu_xii}.
%for 
%over the support set 
%$(\xi^T,i)^T\in $ 
%that is caused by the inf-conv smoothing technique. 

For the purpose of unification, we use $\zeta$ to represent the random vector involved in the smoothing approximation of $\tilde h_{\mu}$, 
%{and $W$ be the support set of $\zeta$}, 
no matter $h$ is in the form of (\ref{orip-2}), or (\ref{max}), or (\ref{Emax}).
Throughout the paper, we make the following assumptions for the smoothing stochastic gradients of smoothing functions.

\begin{assumption}\label{assumption 1}
For every $x\in X$ and every $\mu\in(0,\bar{\mu}]$, there exists a constant $\sigma>0$ irrelevant to $\mu$ such that 
\begin{eqnarray*}
&(a)&\ {\mathbb{E}}_{\zeta}\left[\nabla\tilde{\mathbf{\Psi}}_{\mu}(x, \zeta)\right]=\nabla\tilde{\psi}_{\mu}(x),\\
&(b)&\ {\mathbb{E}}_{\zeta}\left[\|\nabla\tilde{\mathbf{\Psi}}_{\mu}(x, \zeta)-\nabla\tilde{\psi}_{\mu}(x)\|^2\right]\leq \sigma^2.
\end{eqnarray*}
%where $\sigma>0$ is a constant, and 
%where the expectation $\mathbb{E}$ is taken with respect to the random vector $\zeta\in \Lambda$ with given probability density.
\end{assumption}

For a fixed $\mu\in(0,\bar{\mu}]$, (a) and (b) in Assumption \ref{assumption 1} are common assumptions for stochastic algorithms dealing with smooth objective functions; see e.g., Assumption 1 of \cite{ghadimi2016mini}, and (5.1)--(5.2) of \cite{lan2020communication}. In contrast, Assumption \ref{assumption 1} needs (a) and (b) to be held for every $\mu\in(0,\bar{\mu}]$.
%
%We will show in Appendix \ref{appendixa} that the smoothing functions constructed by the three smoothing techniques outlined in {Sect.} \ref{sec:techniques} satisfy Assumption \ref{assumption 1} under mild conditions.

In fact, for every $x\in X$ {and every $\mu\in(0,\bar{\mu}]$}, Assumption \ref{assumption 1} (a) is easy to fulfill by just using the definitions of stochastic gradients in \eqref{gradorip-2}-\eqref{gradEmax}.

When considering \eqref{orip}-\eqref{orip-2}, {$\zeta = \xi$} if the Nesterov's smoothing and the inf-conv smoothing techniques are used, and $\zeta = (\xi^T,v^T)^T$ if the RS technique is employed, and 
\begin{eqnarray}\label{as1orip2}
\nabla \tilde \psi_{\mu}(x) &=& \nabla f(x) + \nabla \tilde{h}_{\mu}(x) = \nabla f(x) + \nabla \mathbb{E}_{\zeta}\left[\tilde{\mathbf{H}}_{\mu}(x,\zeta)\right] \nonumber\\
&=& \mathbb{E}_{\zeta}\left[\nabla f(x) + \nabla \tilde{\mathbf{H}}_{\mu}(x,\zeta)\right]
= \mathbb{E}_{\zeta}\left[\nabla {\tilde{\mathbf{\Psi}}}_{\mu}(x,\zeta)\right],
\end{eqnarray}
where the position of the gradient operator and the expectation operator can be changed according to Proposition 4 of \cite{ruszczynski2003stochastic} and the facts that $\tilde{\mathbf{\Psi}}_{\mu}(\cdot,\zeta)$ is well defined and finite, and $\tilde{\mathbf{\Psi}}_{\mu}(\cdot,\zeta)$ is convex and differentiable at every $x\in X$.
When considering \eqref{orip} with $h$ being defined in \eqref{max},
we have {$\zeta =i$} and  by \eqref{maxgrad}, 
\begin{eqnarray}\label{as1max}
\nabla\tilde{\psi}_{\mu}(x)
&=& \nabla f(x)+\nabla\tilde{h}_{\mu}(x)=\nabla f(x)+\sum_{\zeta=1}^q p_{x,\mu}(\zeta)\nabla \tilde{h}_{\zeta,\mu}(x)\nonumber\\
&=& \sum_{\zeta=1}^q p_{x,\mu}(\zeta) \left(\nabla f(x) + \nabla \tilde{h}_{\zeta,\mu}(x) \right)
= \mathbb{E}_{\zeta}\left[\nabla \tilde{\mathbf{\Psi}}_{\mu}(x,\zeta)\right].
\end{eqnarray}
When considering \eqref{orip} with $h$ being defined in \eqref{Emax},  we have $\zeta =(\xi^{T},i)^{T}$, and by \eqref{maxgrad1}
\begin{eqnarray*}
\nabla\tilde{\psi}_{\mu}(x)
=\mathbb{E}_{\xi} \left[\sum\limits_{i=1}^q p_{x,\mu}^{\xi}(i)\left(\nabla f(x) +\nabla \tilde h_{i,\mu}(x,\xi)\right)\right]
=\mathbb{E}_{\zeta}\left[\nabla \tilde{\mathbf{\Psi}}_{\mu}(x,\zeta)\right].
\end{eqnarray*}

We will show in {Sect.} \ref{Smoothing properties} that Assumption \ref{assumption 1} (b) holds for every $x\in X$ and every $\mu\in(0,\bar{\mu}]$ under mild assumptions for the stochastic gradients of smoothing approximations. When considering \eqref{orip}-\eqref{orip-2}, we will repeatedly use the following fact with respect to different $\tilde h_{\mu}$ constructed using different smoothing techniques.
\begin{eqnarray}\label{As-b-12}
&&\mathbb{E}_{\zeta}\left[\left\|\nabla \tilde{\mathbf{\Psi}}_{\mu}(x,\zeta)-\nabla\tilde{\psi}_{\mu}(x)\right\|^2\right]\nonumber\\
&&\quad=\mathbb{E}_{\zeta}\left[\left\|\nabla \tilde{\mathbf{H}}_{\mu}(x,\zeta)-\nabla\tilde{h}_{\mu}(x)\right\|^2\right]\nonumber\\
&&\quad=\mathbb{E}_{\zeta}\left[\left\|\nabla \tilde{\mathbf{H}}_{\mu}(x,\zeta)\right\|^2+\left\|\nabla \tilde{h}_{\mu}(x)\right\|^2-2\left\langle \nabla \tilde{\mathbf{H}}_{\mu}(x,\zeta), \nabla \tilde{h}_{\mu}(x)\right\rangle \right]\nonumber\\
&&\quad\leq\mathbb{E}_{\zeta}\left[\left\|\nabla \tilde{\mathbf{H}}_{\mu}(x,\zeta)\right\|^2\right].
\end{eqnarray}

%We will show in the following {Sect.} that under mild assumptions  Assumption \ref{assumption 1} (b) holds for the smoothing gradients based on smoothing approximations $\tilde h_{\mu}$ for  $h$ being defined in \eqref{orip-2}, \eqref{max}, and \eqref{Emax}, using the above mentioned techniques. 

\subsection{Smoothing properties}\label{Smoothing properties}
Now we show that the smoothing approximations constructed by the Nesterov's smoothing, the RS and the inf-conv smoothing techniques satisfy Definition \ref{smoothingdefinition} as well as Assumption \ref{assumption 1} under mild conditions. {Without specification, for those smoothing approximations,  $\bar \mu >0$ in Definition \ref{smoothingdefinition} can be arbitrarily chosen.}  The proofs of Lemmas are given in Appendix \ref{appendixa}.

%\subsection{Nesterov's smoothing}

\begin{lemma}\label{lemma2.3}
%	Let $\omega: U \rightarrow \mathbb{R}$ be a $\sigma_{\omega}$-strongly convex function, and
For $h$ in \eqref{orip-2}, assume that ${\mathbf{H}(x,\xi)}$ 
%in \eqref{orip-2}
has a linear max structure in \eqref{h} for a.e. $\xi \in \Xi$, and there exists a constant $c_1>0$ such that $\|A_{\xi}\| \le c_1$ for a.e. $\xi\in \Xi$. 
%{For an arbitrary $\bar \mu >0$ and any $\mu\in(0,\bar{\mu}]$,} 
Let $\tilde{h}_{\mu}(x) = \mathbb{E}_{\xi}\left[{\mathbf{\tilde H}_{\mu}(x,\xi)}\right]$ with $\mathbf{\tilde H}_{\mu}(x,\xi)$  constructed by the Nesterov's smoothing technique in \eqref{Nessmooth}. Then the following statements hold.
\begin{itemize}
\item[(i)]
$\tilde{h}_{\mu}$ is a smoothing function of $h$ satisfying Definition \ref{smoothingdefinition}, with {parameters $
	(\kappa, K, {\cal{L}}_h) =
	\left(\max_{u\in U}\left\{d(u)\right\}, 0 , {\frac{\left\|\mathbb{E}_{\xi}\left[ A_{\xi}\right]\right\|^2}{\sigma_{d}}}\right)$}. 

\item[(ii)] Assumption \ref{assumption 1} holds with $\zeta=\xi$ and $\sigma^2 = c_1^2 \max_{u\in U} \left\{\|u\|^2\right\}$.
\end{itemize}
\end{lemma}

Now we turn to the smoothing approximations constructed by the RS technique.

\begin{lemma}\label{lemma2.2}%uniform distribution
For $h$ in \eqref{orip-2}, %{and for an arbitrary $\bar \mu >0$ and any $\mu\in(0,\bar{\mu}]$,} 
let $\tilde{h}_{\mu}$ be defined as in \eqref{RSsmooth} that is constructed by the RS technique, and let the random vector $v$ follow the uniform density function $\rho(v)$ on $B(0,1)$.
Assume that 
\begin{eqnarray}
\label{L0uniform}
\mathbb{E}_{\xi}\left[\|\partial \mathbf{H}(x,\xi)\|^{2}\right] \leq L_{0}^{2}~
%\text{in~\eqref{orip-2}~and~\eqref{Emax}
	\text{for~any}~x \in
	\operatorname{int}
	\left(X+B(0,\mu)\right).
\end{eqnarray}
%or
%\begin{eqnarray*}
%\|\partial h(x)\|^{2} \leq %L_{0}^{2}~\text{in~\eqref{max}~for~any~}x\in\left(X+B(0,\mu)\right).
%\end{eqnarray*}
Then the following statements hold.
\begin{itemize}
	\item[(i)] 
	%For any $\bar \mu >0$ and any $\mu\in(0,\bar{\mu}]$, 
	$\tilde{h}_{\mu}$ is a smoothing function of $h$ satisfying Definition \ref{smoothingdefinition}, with {parameters 
		$
		(\kappa, K, {\cal{L}}_h) =
		\left(
		\frac{L_{0}^2+1}{2}, 0,   L_{0} \sqrt{d}\right)$}.
	
	\item[(ii)] Assumption \ref{assumption 1} holds with $\zeta=(\xi^{T},v^{T})^{T}$
	%for \eqref{orip-2} and \eqref{Emax}, or $\zeta=v$ for \eqref{max}, 
	and $\sigma^2 = L_0^2$.
\end{itemize}
\end{lemma}

\begin{lemma}\label{lemma2.21}%standard normal distribution
For $h$ in \eqref{orip-2}, let $\tilde{h}_{\mu}$ be defined as in \eqref{RSsmooth} that is constructed by the RS technique, and let the random vector $v\sim \mathcal{N}\left(0, I_{d}\right)$. Assume that% for any $x\in X$,
\begin{eqnarray}
	\label{L0Gaussian}
	\mathbf{H}(\cdot,\xi)
	~\text{is}~L_{0}\text{-Lipschitz~ for~ a.e.~}\xi\in\Xi.
	%~\text{in~\eqref{orip-2}~and~\eqref{Emax}},
\end{eqnarray}
%or
%\begin{eqnarray*}
%	h(x) ~\text{is}~L_{0}\text{-Lipschitz~in~}\eqref{max}.
%\end{eqnarray*}
Then the following statements hold.
\begin{itemize}
	\item[(i)] 
	%For any $\bar \mu >0$ and any $\mu\in(0,\bar{\mu}]$,
	$\tilde{h}_{\mu}$ is a smoothing function of $h$ satisfying Definition \ref{smoothingdefinition}, with {parameters 
		$
		(\kappa,K,{\cal{L}}_h) =
		\left(L_{0}\sqrt{d},0, L_{0}\right)$}.
	
	\item[(ii)] Assumption \ref{assumption 1} holds with $\zeta=(\xi^{T},v^{T})^{T}$ and $\sigma^2 = L_0^2$.
\end{itemize}
\end{lemma}

We consider the approximations constructed by the inf-conv smoothing technique. Recall that the function  $\omega$ in (\ref{inf-conv-H}) is $\frac{1}{\sigma_{\omega}}$-smooth.
%\subsection{Inf-conv smoothing}
\begin{lemma}\label{lemma2.1}
%	Let $\varphi: X \rightarrow \mathbb{R}$ be a $\frac{1}{\sigma_{\varphi}}$-%smooth and convex function, and
%	let $h$ {in \eqref{max}} be a convex function over $X$.
For $h$ in \eqref{orip-2}, let 
%the function 
$\tilde{h}_{\mu}(x) = \mathbb{E}_{\xi}\left[{\mathbf{\tilde H}}_{\mu}(x,\xi)\right]$ with
${\mathbf{\tilde H}}_{\mu}(x,\xi)$ 
%suggested in \cite{beck2012smoothing}
constructed by the inf-conv smoothing technique in \eqref{inf-conv-H}.
The following statements hold.

\begin{itemize}
	\item[(i)]
	Assume that
	%either $\omega(x)\ge 0$ for all $x\in \operatorname{dom}{\omega}$, or
	$\omega^*(y)\le 0$ for all $y\in \operatorname{dom} \omega^*$.
	% or $\omega^*(y) \ge 0$ for all $y\in \operatorname{dom} \omega^*$.
	Then 
	%for any $\bar \mu >0$ and any $\mu\in(0,\bar{\mu}]$, 
	$\tilde{h}_{\mu}$ is a smoothing function of $h$
	satisfying Definition \ref{smoothingdefinition}, with {parameters $
		(\kappa,K,{\cal L}_h) = 
		\left(\omega({0}),0, \frac{1}{\sigma_{\omega}}\right)$.}
	
	\item[(ii)]	Assume that $\omega(0)\le 0$, $\omega(\cdot)$ is level bounded, the function ${\bf{H}}(\cdot,\xi)$ is lower semicontinuous, {and $\inf_{x\in \mathbb{R}^d} {\bf H}(x,\xi)>-\infty$} for a.e. $\xi\in\Xi$, and $D[{\bf{H}},{\omega^*}] = \mathbb{E}_{\xi}[\sup_{x\in X} \inf_{d\in {\partial {\bf H}(x,\xi)}} \omega^*(d)]<\infty$. Then 
	$\tilde{h}_{\mu}$ is a smoothing function of $h$ on $X$ satisfying Definition \ref{smoothingdefinition}, with 
	{parameters 
		$
		(\kappa,K,{\cal L}_h) =
		\left(D[{\bf{H}},{\omega^*}],0, \frac{1}{\sigma_{\omega}}\right)$.}
	
	\item[(iii)] Assume that $\max_{x\in X}\left\{\left\|\nabla \omega(x)\right\|\right\} < \infty$. Then Assumption \ref{assumption 1} holds with $\zeta=\xi$ and $\sigma^2=\max_{x\in X}\left\{ \left\|\nabla \omega(x)\right\|^2\right\}.$
\end{itemize}
\end{lemma}

\noindent{\bf{Remark 2.}}\quad	%\begin{remark}\label{remark2.6}
For $\omega$ defined in \eqref{b-omega}, its convex conjugate $\omega^*$ in \eqref{2.11} satisfies the assumption in Lemma \ref{lemma2.1} (i) that $\omega^*(y)\le 0$ for all $y\in \operatorname{dom} \omega^*$, and its gradient $\nabla \omega$ in \eqref{gradient-omega} satisfies the assumption in Lemma \ref{lemma2.1} (iii) that $\max_{x\in \mathbb{R}^d} \{\|\nabla \omega(x)\|\}\le 1 < \infty$. %The feasible region $X$ is not required to be bounded.

In Example 4.8 of \cite{beck2012smoothing}, for the function   $t(x)=\|x\|$, the smoothing approximation $t_{\mu}(x) = \sqrt{\|x\|^2 + \mu}$ can be obtained using the inf-conv smoothing technique 
with 
$
\omega(x) =\sqrt{1+\|x\|^2}
$ over $\mathbb{R}^d$. The convex conjugate   
$\omega^*(y) = -\sqrt{1-\|y\|^2}$ with $\operatorname{dom}  \omega^* = \{y\ :\ \|y\|\le 1\}$ clearly satisfies the assumptions in Lemma \ref{lemma2.1} (i)
and consequently $t_{\mu}(x)$ is a smoothing function of $t(x)$ satisfying Definition \ref{smoothingdefinition}. {Its gradient $\nabla \omega(x) = \frac{x}{\sqrt{1+\|x\|^2}}$ satisfies the assumptions in Lemma \ref{lemma2.1}  (iii), because $\max_{x\in \mathbb{R}^d} \|\nabla \omega(x)\|<1$.}

For $\omega(x) = \frac{1}{2}\|x\|^2$, the inf-conv smoothing technique reduces to the well-known Moreau smoothing technique, and the assumptions on $\omega$ in Lemma \ref{lemma2.1}  (ii) are satisfied. If in addition, $D[\mathbf{H},\omega^*]<\infty$, then the Moreau smoothing satisfies Lemma \ref{lemma2.1} (ii).
{If $X$ is bounded and ${\bf{H}}(\cdot,\xi)$ is locally Lipschitz, then $D[\mathbf{H},\omega^*]<\infty$. If ${\bf{H}}(\cdot,\xi)$ is globally Lipschitz, then without the boundedness of $X$, we still have $D[\mathbf{H},\omega^*]<\infty$.}  The feasible region $X$ is bounded is necessary to show Lemma \ref{lemma2.1} (iii) holds.
%\end{remark}

\vskip 2mm
\begin{lemma}\label{lem3.21}
For $h$ in (\ref{max}), assume that for each $i\in \mathbb{I}_q$, 
%any $\bar \mu >0$ and any $\mu\in(0,\bar{\mu}]$, 
$\tilde h_{i,\mu}$ is a smoothing function of $h_i$ %constructed by a certain smoothing technique,
{with parameters $(\kappa_{i}, K_i,{\cal{L}}_{h_i})$}. Let $\tilde h_{\mu}$ be the smoothing approximation defined in (\ref{inf-conv-max}) based on $\tilde h_{i,\mu}$ and the inf-conv smoothing technique. {Assume also that for each $i\in\mathbb{I}_q$, there exists a constant $M_i \ge 0$ irrelevant to $x\in X$ and $\mu\in (0,\bar{\mu}]$ such that 
	\begin{eqnarray}\label{Mi}
		\max\limits_{x\in X,~\mu\in (0,\bar{\mu}]}\left\{\left\|\nabla \tilde{h}_{i,\mu}(x) \right\|\right\} \leq M_i < \infty.
\end{eqnarray}} Then the following statements hold. 

\begin{itemize}
	\item[(i)]
	%i.e., 
	%\begin{eqnarray}
	%	\label{gradient-cons-hi}
	%	\tilde{G}_{h_i}(x):=\left\{ \lim_{z \rightarrow x,~ \mu \downarrow 0} \nabla\tilde{h}_{i,\mu}(z)\right\} \subseteq \partial h_i(x)\quad \text{for~any}~x \in X.
	%\end{eqnarray}
	The function ${\tilde h}_{\mu}$ is a smoothing function of $h$ satisfying Definition \ref{smoothingdefinition}, with
	{parameters 
		$
		(\kappa, K, {\cal L}_h)$  where 
		$\kappa =  \ln q + \sqrt{\sum_{i=1}^q \kappa_i^2}$, $K=\sum_{i=1}^q  K_{i}$, ${\cal L}_h =\left( \sum_{i=1}^{q}M_i\right)^2
		+\sum_{i=1}^{q}{\cal{L}}_{h_i}. $}
	
	\item[(ii)] Assumption \ref{assumption 1} holds with $\zeta=i
	%\in\mathbb{I}_q
	$ and $\sigma^2 = \mathbb{E}_{i}{\left[\left(M_i\right)^2\right]}$.	
\end{itemize}
\end{lemma}

\noindent {\bf{Remark 3.}}\quad	%\begin{remark}\label{remarkb}
The assumption in (\ref{Mi}) is a mild assumption that can be satisfied by $\tilde{h}_{i,\mu}$ constructed from various smoothing techniques, as illustrated below.

Let $\tilde h_{i,\mu}$ be constructed by the Nesterov's smoothing technique as in (\ref{Nessmooth}),
%with $\tilde{\bf{H}}_{\mu}(x,\xi)$, $A_{\xi}$, $Q_{\xi}$ being replaced by $\tilde{h}_{i,\mu}(x)$, $A_i$, $Q_i$.
i.e., 
\begin{eqnarray*}
%\label{hNessmooth}
\tilde{h}_{i,\mu}(x):=\max\limits_{u\in U_i}\left\lbrace \langle A_{i}x,u\rangle-Q_{i}(u)-\mu d(u)\right\rbrace,
\end{eqnarray*}
where $U_i$ is a bounded closed set, $A_i$ is a matrix, and $Q_i$ is a continuous convex function, $d$ is a nonnegative $\sigma$-strongly convex function. Let $\hat{u}_{\mu}(x,i)$ be the unique optimal solution of the maximization problem above.
Then we get
\begin{eqnarray*}
%\label{hiNe}
\left\|\nabla{\tilde h}_{i,\mu}(x) \right\|
=\left\|A_{i}^{T} \hat{u}_{\mu}(x,i)\right\|
%\leq \left\|A_{i} \right\| \left\|\hat{u}_{\mu}(x,i)\right\|
\leq   \left\|A_{i} \right\|\max_{u\in U_i} \|u\|:= M_i<\infty.
\end{eqnarray*}
%where $\hat{u}_{i,\mu}(x)$ is the unique optimal solution of the maximization problem in \eqref{hNessmooth}, the last inequality follows from the compactness of $U$.

Now we consider $\tilde h_{i,\mu}$ being constructed by the RS technique as in (\ref{RSsmooth}), i.e., 
%\begin{eqnarray}
%\label{hRSsmooth}
$\tilde{h}_{i,\mu}(x):= {\mathbb{E}}_{v}[ h_i(x+\mu v)].
$
Letting $g_{i,v}\in \partial h_i\left(x+ \mu v\right)$, we have
\begin{eqnarray*}
%\label{hradomizedgrad}
\left\|\nabla{\tilde h}_{i,\mu}(x) \right\|
=\left\|\mathbb{E}_v \left[g_{i,v}\right]\right\|
\leq \mathbb{E}_v \left[\left\|g_{i,v}\right\|\right]
\leq \mathbb{E}_v \left[\left\|\partial h_i\left(x+ \mu v\right)\right\|\right]
\leq L_{0,i}:= M_i,
\end{eqnarray*}
provided that the assumptions  (\ref{L0uniform}) and (\ref{L0Gaussian}) hold, with
$\bf{H}(\cdot,\xi)$ being replaced by $h_i(\cdot)$ and  $L_0$ being replaced by  $L_{0,i}$, respectively.
%\begin{itemize}
%    \item[1)] If the random vector $v$ follows the uniform density function $\rho(v)$ on $B(0,1)$, and assuming $\left\|\partial h_i\left(z\right)\right\|\leq L_{0,i}$ for any $z\in\mathrm{int}\left(X+B(0,\mu)\right)$ as in Lemma \ref{lemma2.2}, we obtain $M_i=L_{0,i}$.
%   \item[2)] If the random vector $v\sim \mathcal{N}\left(0, I_{d}\right)$ distribution, and assuming $h_i(\cdot)$ is $L_{0,i}$-Lipschitz as in Lemma \ref{lemma2.21}, by Corollary 1.81 of \cite{mordukhovich2006variational} we get $\|\partial h_i(\cdot)\|\leq L_{0,i}$, i.e., $M_i=L_{0,i}$.
%\end{itemize}

%$\circ$ For each $i\in \mathbb{I}_q$, 
Let $\tilde h_{i,\mu}$ be  constructed by the inf-conv smoothing technique as in (\ref{inf-conv-H}) by replacing ${\bf H}(\cdot,\xi)$ by $h_i(\cdot)$ and $\tilde{\bf{H}}_{\mu}(\cdot,\xi)$ by $\tilde{h}_{i,\mu}(\cdot)$, i.e.,
\begin{eqnarray*}
%\label{hinf}
\tilde{h}_{i,\mu}(x):=\inf\limits_{y \in \mathbb{R}^{d}}\left\{h_i(y)+\mu \omega \left(\frac{x-y}{\mu}\right)\right\},
\end{eqnarray*}
and $\hat{v}_{\mu}(x,i)$ be a minimizer of the right-hand side of the above inf problem. As in (\ref{grad-inf-conv}), we find 
\begin{eqnarray*}
%\label{hiin}
\left\|\nabla{\tilde h}_{i,\mu}(x) \right\|
=\left\|\nabla \omega\left(\frac{x-\hat{v}_{\mu}(x,i)}{\mu} \right)\right\|.
\end{eqnarray*}
The existence of an upper bound $M_i$ depends on the choice of $\omega$. 
\begin{itemize}
\item[-] It is clear that if $\|\nabla \omega(\cdot)\|$ is bounded on $\mathbb{R}^d$, then its upper bound can be specified as a constant $M_i$. For example, for $\omega(x)=\ln\left(\sum_{j=1}^d e^{x_j}\right)$ considered in (\ref{b-omega}), $\max_{x\in \mathbb{R}^d}\|\nabla \omega(x)\|\le 1$  as given in (\ref{gradient-omega}).
\item[-] If $\omega(x)=\frac{1}{2}\|x\|^2$ is adopted,
%the inf-conv smoothing technique reduces to the Moreau smoothing technique. Assume that 
and $h_i$ is $L_{0,i}$-Lipschitz continuous on $\mathbb{R}^d$, we can choose $M_i:=L_{0,i}$  according to Lemma 3.3 of \cite{bohm2021variable}.
\end{itemize}

%\end{remark}

\begin{lemma}\label{lem3.211}
For $h$ in (\ref{Emax}), assume that for each $i\in \mathbb{I}_q$ and a.e. $\xi\in\Xi$, 
%any $\bar \mu >0$ and any $\mu\in(0,\bar{\mu}]$, 
$\tilde h_{i,\mu}(\cdot,\xi)$ is a smoothing function of $h_i(\cdot,\xi)$ 
{with parameters $(\kappa_{i,\xi}, K_{i,\xi}, {\cal L}_{h_{i,\xi}})$.}
Let $\tilde h_{\mu}$ be the smoothing approximation defined in \eqref{inf-conv-Emax} based on $\tilde h_{i,\mu}$ and the inf-conv smoothing technique. 
%\begin{eqnarray*}
%	\tilde{G}_{h_i}(x,\xi):=\left\{ \lim_{z \rightarrow x,~ \mu \downarrow 0} \nabla\tilde{h}_{i,\mu}(z,\xi)\right\} \subseteq \partial h_i(x,\xi)\quad \text{for~any}~x \in X.
%\end{eqnarray*}
{Assume also that for each $i\in\mathbb{I}_q$ and {a.e.} $\xi\in\Xi$, there exists a constant $M_{i,\xi} \ge 0$ irrelevant to $x\in X$ and $\mu\in (0,\bar{\mu}]$ such that 
	\begin{eqnarray}\label{Mixi}
		\max\limits_{x\in X,~\mu\in (0,\bar{\mu}]}\left\{\left\|\nabla \tilde{h}_{i,\mu}(x,\xi) \right\|\right\} \leq M_{i,\xi} < \infty.
\end{eqnarray}}
Then the following statements hold.
\begin{itemize}
	\item[(i)]
	The function ${\tilde h}_{\mu}$ is a smoothing function of $h$ satisfying Definition \ref{smoothingdefinition} with parameters $(\kappa,K,{\cal{L}}_h)$, where $\kappa = 
	\ln q + \mathbb{E}_{\xi}
	\left[
	\sqrt{\sum_{i=1}^q {\kappa_{i,\xi}^2}}
	\right]$,  
	$K=\sum_{i=1}^{q}\mathbb{E}_{\xi}
	\left[
	K_{i,\xi}
	\right]$, and  
	${\cal{L}}_{h} = \mathbb{E}_{\xi}
	\left[
	\left( \sum_{i=1}^{q}
	M_{i,\xi}\right)^2
	+\sum_{i=1}^{q}
	{\cal{L}}_{h_{i,\xi}}
	\right]   
	$.

	\item[(ii)] Assumption \ref{assumption 1} holds with $\zeta=(\xi^{T},i)^{T}
	%\in\Xi\times\mathbb{I}_q
	$ and 
	$\sigma^2= \mathbb{E}_{\xi,i}
	\left[ {\left(M_{i,\xi}\right)^2}\right].$
\end{itemize}
\end{lemma}

\section{SSAG method}\label{sec:method}

In this section, we develop the SSAG method for \eqref{orip}. We will show the SSAG method achieves the best-known complexity bounds in terms of the $\mathcal{SFO}$ and iteration number among the SA methods using the first-order information.

At the $k$-th iterate, we randomly choose a mini-batch samples $$\zeta_{[m_k]}=\left\lbrace \zeta_{1,k},\ldots,\zeta_{m_k,k}\right\rbrace $$ of the random vector $\zeta\in \Xi$ following the given probability density, where $m_k$ is the batch size. We use $\mathcal{F}_k$ to denote all randomness occurred before the $(k+1)$-th iterate, i.e., $\mathcal{F}_k=\left\lbrace \zeta_{[m_1]},\ldots,\zeta_{[m_{k}]}\right\rbrace $.
For any $k\geq 1$, we denote the mini-batch stochastic gradients $\nabla \tilde{\mathbf{\Psi}}_{\mu_{k}}^k$ by $$\nabla \tilde{\mathbf{\Psi}}_{\mu_{k}}^k=\frac{1}{m_{k}}\sum_{i=1}^{m_{k}}\nabla \tilde{\mathbf{\Psi}}_{\mu_{k}}(x_k, \zeta_{i,k}).$$

The following lemma addresses the relations of $\nabla \tilde{\mathbf{\Psi}}_{\mu_{k}}^k$ to $\nabla \tilde{\psi}_{\mu_{k}}(x_k)$, which can be shown without difficulty using the arguments similar as those for Lemma 2 in \cite{ghadimi2016mini}.
\vskip 2mm
\begin{lemma}\label{lemma3.1} Under Assumption \ref{assumption 1}, we have for any $k\geq 1$ and $\mu_{k}>0$,
	\begin{eqnarray*}
		&(a)&~ \mathbb{E}_{\mathcal{F}_k}\left[\nabla \tilde{\mathbf{\Psi}}_{\mu_{k}}^k\right] = \mathbb{E}_{\mathcal{F}_k}\left[\nabla \tilde{\psi}_{\mu_{k}}(x_k)\right],\\
		&(b)&~ \mathbb{E}_{\mathcal{F}_k}\left[\left\| \nabla \tilde{\mathbf{\Psi}}_{\mu_{k}}^k - \nabla\tilde{\psi}_{\mu_{k}}(x_k) \right\|^2\right] \leq \frac{\sigma^2}{m_{k}}.
	\end{eqnarray*}
	%where the expectation $\mathbb{E}$ is taken with respect to the history $\mathcal{F}_k$.
	%of samples drawn according to the given probability density function.
\end{lemma}
%Lemma \ref{lemma3.1} can be shown without difficulty using arguments similar as Lemma 2 in \cite{wang2022stochastic}.

%\subsection{Smoothing stochastic accelerated gradient method}
\vskip 1mm

Now we are ready to give the basic scheme of the SSAG method in Algorithm \ref{alg:SSAG}.
\begin{algorithm}
	\caption{Stochastic smoothing accelerated gradient (SSAG) method}
	\label{alg:SSAG}
	\begin{algorithmic}[1]
		\State{Given the iteration limit $N>0$, initial points $z_{0}=y_{0} \in X$, the batch sizes $\left\{m_k\right\}_{k\geq 1}$ with $m_k \geq 1$, the smoothing parameters $\left\{\mu_{k}\right\}_{k\geq 0}$ with $\mu_{k}>0$, and $\left\{\alpha_{k}\right\}_{k\geq 0}$ with $\alpha_{k}\in (0, 1]$, and positive sequences $\left\{\beta_k\right\}_{k\geq 1}$ and $\left\{\theta_k\right\}_{k\geq 1}$.
			% $\alpha_{0}=1$.
		}
		\For{$k=1,\ldots,N$}
		\State{Step 1. Set $x_{k}=\alpha_{k-1}z_{k-1}+(1-\alpha_{k-1})y_{k-1}.$}
		\State{Step 2. Call the $\mathcal{SFO}~~m_k$ times to obtain $\nabla \tilde{\mathbf{\Psi}}_{\mu_{k}}(x_{k}, \zeta_{i,k})$, $i=1,\ldots, m_{k}$, and compute $\nabla \tilde{\mathbf{\Psi}}_{\mu_{k}}^{k}=\frac{1}{m_{k}}\sum\limits_{i=1}^{m_{k}}\nabla \tilde{\mathbf{\Psi}}_{\mu_{k}}(x_{k}, \zeta_{i,k})$.}
		\State{Step 3. Find
			$y_{k}={\rm{argmin}}_{y\in X}\left\lbrace \left\langle \nabla \tilde{\mathbf{\Psi}}_{\mu_{k}}^{k},y-x_{k}\right\rangle+\frac{\beta_k}{2}\|y-x_{k}\|^2\right\rbrace .$}
		\State{Step 4. Find
			$z_{k}={\rm{argmin}}_{x\in X}\left\lbrace \left\langle\nabla\tilde{\mathbf{\Psi}}_{\mu_{k}}^{k},x-x_{k}\right\rangle+\frac{\theta_{k}}{2}\|x-z_{k-1}\|^2\right\rbrace .$}
		%	\State{Step 5. Set $k:= k+1$.}
		\EndFor\\
		
		\Return $y_N$
	\end{algorithmic}
\end{algorithm}

We now specify the algorithmic parameters of the SSAG method. {For simplicity, we denote $$L_{\mu_{k}}:=L_{\tilde{\psi}_{\mu_{k}}} = K + \frac{\cal{L}_{\psi}}{\mu_k}$$ for}  the smoothness parameter of $\tilde{\psi}_{\mu_{k}}(x)$. For given ${\hat{\mu}}>0$ and $k\ge 1$,
\begin{eqnarray}
	& &	\alpha_0=1,~~~~~\frac{1-\alpha_{k}}{\alpha_{k}^2}=\frac{1}{\alpha_{k-1}^2},~~~~~
	{\mu_{k}= \frac{\hat{\mu}}{k}}
	%\alpha_{k-1}
	,\label{alphak}\\
	& &	
	%\beta_{1}= L_{\mu_{1}}+\frac{1}{\sqrt{m_{1}}},~~	
	{\beta_{k} = L_{\mu_{k}}}+ {\frac{1}{\alpha_{k-1}}}, \label{betak}\\
	%\max\left\lbrace \beta_{k},L_{\mu_{k+1}}+\frac{1}{\sqrt{(k+1)m_{k+1}}\alpha_{k}^2}
	%\right\rbrace ,\label{betak}\\
	& &	{{\theta_{k}=2\alpha_{k-1}\beta_{k}}. }\label{thetak}
\end{eqnarray} 
%defined in \eqref{Lmuk}. %For a predetermined accuracy $\epsilon>0$ and the batch size $m$, the iteration limit $N$ is calculated by \eqref{N}. That is, for any batch size $m$, the corresponding iteration limit $N$ is obtained.

The sequence $\{\alpha_k\}$ is defined in \eqref{alphak}, the same as that in Duchi's work \cite{duchi2012randomized}. 
{It is easy to find that $\{\alpha_k\}$ is monotonically decreasing, since  
	\begin{eqnarray*}
		\alpha_k = \frac{-\alpha_{k-1}^2 + \alpha_{k-1}^2 \sqrt{1 + \frac{4}{\alpha_{k-1}^2}}}{2} < \frac{-\alpha_{k-1}^2 + \alpha_{k-1}^2 (1+ \frac{2}{\alpha_{k-1}})}{2} = \alpha_{k-1}.
\end{eqnarray*}}
%
%\begin{lemma}\label{lemma3.42} (
According to {Lemma 4.3 of \cite{duchi2012randomized}} and by simple mathematical induction,
%Let the sequence $\{\alpha_{k}\}$ satisfy \eqref{alphak}, then
\begin{eqnarray}\label{alphak_p1}
	{\frac{1}{2(k+2)} \le } \alpha_{k} \leq \frac{2}{k+2},\quad \mbox{for}\  k=0,1,2,\ldots.
\end{eqnarray}
%\end{lemma}
%Similar as (2.9) of \cite{ghadimi2016accelerated}, the sequence $\{\beta_k\}$ is set in \eqref{betak}. 
The sequence $\{\beta_k\}_{k\ge 1}$ with $\beta_k$ defined in \eqref{betak} is a nondecreasing sequence that satisfies
%\begin{eqnarray}
$\beta_{k}-L_{\mu_{k}}>0.$
%\end{eqnarray}
For $k\ge 1$, the parameter $\theta_k$ is set in
\eqref{thetak}. They are designed for the later theoretical analysis. 

Below we provide several technique lemmas for developing complexity results. Recall that $x^{\rm{opt}}$ refers to an arbitrary optimal solution of \eqref{orip} 

{that is bounded}. For ease of notation, we denote by 
\begin{eqnarray}\label{deltaV_muk}
	\delta_{\mu_{k}}(x_k):=\nabla\tilde{\mathbf{\Psi}}_{\mu_{k}}^k-\nabla\tilde{\psi}_{\mu_{k}}(x_k),\ 
	V_{\mu_k}:=\tilde{\psi}_{\mu_{k}}(y_{k})-{\tilde{\psi}_{\mu_k}(x^{\rm opt})},
\end{eqnarray}
and
\begin{eqnarray}\label{Pk}
	P_{k}:=2\kappa(1-\alpha_{k-1})\left(\mu_{k-1}-\mu_{k}\right). 
\end{eqnarray}
%in TABLE \ref{tablesymbol}.
%
By Definition \ref{smoothingdefinition}, we can easily obtain that for any $k\geq 1$ and any $x,y_k\in X$,
\begin{eqnarray}
	\psi(y_k)-\psi(x)&\leq& \tilde{\psi}_{\mu_{k}}(y_k)-\tilde{\psi}_{\mu_{k}}(x)+2\kappa\mu_{k},\label{psibound}\\
	\tilde{\psi}_{\mu_{k+1}}(y_k)-\tilde{\psi}_{\mu_{k+1}}(x)&\leq&\tilde{\psi}_{\mu_{k}}(y_k)-\tilde{\psi}_{\mu_{k}}(x)+2\kappa\left(\mu_{k}-\mu_{k+1}\right).\label{psimurelation}
\end{eqnarray}

Lemma \ref{lemma3.41} below is a special case of Lemma 2 in \cite{ghadimi2012optimal}.

\vskip 2mm
\begin{lemma}
	%(Lemma 2 of %\cite{ghadimi2012optimal})
	\label{lemma3.41}
	Let the convex function $q: X \rightarrow \mathbb{R}$, the point $\tilde{x}\in X$, and the scalar $\nu\geq 0$ be given. %Let $\omega: X \rightarrow \mathbb{R}$ be a differentiable convex function and $V(x, z)=\frac{\|x-z\|^2}{2}$.
	If $x^{\sharp} \in \mathop{\rm{argmin}}_{x \in X}\left\{ q(x)+ \frac{\nu}{2} \left\|\tilde{x}-x\right\|^2 \right\}$, then for any $x\in X$, we have
	\begin{eqnarray*}
		q(x^{\sharp})+ \frac{\nu}{2} \left\|\tilde{x}-  x^{\sharp}\right\|^2 \leq q(x)+ \frac{\nu}{2} \left\|\tilde{x}- x\right\|^2 - \frac{\nu}{2} \left\|x^{\sharp}-x\right\|^2.
	\end{eqnarray*}
\end{lemma}

\vskip 2mm
\begin{lemma}\label{lemma3.4}
	%and {$\theta_k\geq\alpha_{k-1}\beta_k$?}.
	For any $k\geq 1$,
	{\begin{eqnarray}\label{pu}
			& & \quad\quad \left\langle\nabla\tilde{\mathbf{\Psi}}_{\mu_{k}}^{k},y_{k}-\alpha_{k-1}x^{\rm{opt}}-(1-\alpha_{k-1})y_{k-1}\right\rangle
			%&&~~~\leq
			%\frac{\alpha_{k-1}(\alpha_{k-1}\beta_{k}-\theta_k)}{2}\|z_{k-1}-z_{k}\|^2
			%-\frac{\beta_{k}}{2}\|y_{k}-x_{k}\|^2
			%+\alpha_{k-1}\theta_{k}\left(D_{k-1}^2-D_{k}^2\right)\\
			%&&~~~\leq 2\alpha_{k-1}(\alpha_{k-1}\beta_{k}-\theta_k)\left(D_{k-1}^2+D_{k}^2\right)-\frac{\beta_{k}}{2}\|y_{k}-x_{k}\|^2	+\alpha_{k-1}\theta_{k}\left(D_{k-1}^2-D_{k}^2\right)\\
			%&&~~~=(2\alpha_{k-1}^2\beta_{k}-\alpha_{k-1}\theta_k)D_{k-1}^2+(2\alpha_{k-1}^2\beta_{k}-3\alpha_{k-1}\theta_k)D_{k}^2
			%-\frac{\beta_{k}}{2}\|y_{k}-x_{k}\|^2\\
			\leq -\frac{\beta_{k}}{2}\|y_{k}-x_{k}\|^2 \\ \nonumber
			& & \quad\quad\quad  + \frac{\theta_k \alpha_{k-1}}{2} (\|z_{k-1} - x^{\rm opt}\|^2 -\|z_k - x^{\rm opt}\|^2 ) - \frac{\alpha_{k-1}^2 \beta_k
				%(\theta_k - \alpha_{k-1} \beta_k)
			}{2}\|z_k - z_{k-1}\|^2.
	\end{eqnarray}}
\end{lemma}
\begin{proof}
	By the definitions of $z_{k}$ and $y_{k}$ in Algorithm \ref{alg:SSAG}, and Lemma \ref{lemma3.41} with $q(x):= \langle \nabla\tilde{\mathbf{\Psi}}_{\mu_{k}}^{k}, x-x_k \rangle$, we have
	\begin{eqnarray}\label{z}
		\quad\quad\quad \left\langle\nabla\tilde{\mathbf{\Psi}}_{\mu_{k}}^{k},z_{k}-x^{\rm opt}\right\rangle\leq \frac{\theta_{k}}{2}\left(\|z_{k-1} - x^{\rm opt}\|^2
		-\|z_k-x^{\rm opt}\|^2 - \|z_k - z_{k-1}\|^2\right),
	\end{eqnarray}
	and
	\begin{eqnarray*}
		\left\langle\nabla\tilde{\mathbf{\Psi}}_{\mu_{k}}^{k},y_{k}-x\right\rangle
		\leq\frac{\beta_{k}}{2}\left(\|x_{k}-x\|^2-\|y_{k}-x\|^2-\|y_{k}-x_{k}\|^2\right).
	\end{eqnarray*}
	By choosing $x=\alpha_{k-1}z_{k}+(1-\alpha_{k-1})y_{k-1}\in X$ in the above inequality, we obtain
	\begin{eqnarray*}
		&&\left\langle\nabla\tilde{\mathbf{\Psi}}_{\mu_{k}}^{k},y_{k}-\alpha_{k-1}z_{k}-(1-\alpha_{k-1})y_{k-1}\right\rangle\\
		&\leq& \frac{\beta_{k}}{2}\left(\|x_{k}-\alpha_{k-1}z_{k}-(1-\alpha_{k-1})y_{k-1}\|^2-\|y_{k}-x_{k}\|^2\right)\\
		&=& \frac{\beta_{k}}{2}\left(\alpha_{k-1}^2\|z_{k-1}-z_{k}\|^2-\|y_{k}-x_{k}\|^2\right),
	\end{eqnarray*}
	where the last equality follows from  $x_k=\alpha_{k-1}z_{k-1}+(1-\alpha_{k-1})y_{k-1}$ according to  Step 1 of Algorithm \ref{alg:SSAG}.
	Multiplying both sides of \eqref{z} by $\alpha_{k-1}$, then summing it to the above inequality, and using $\theta_k = 2 \alpha_{k-1} \beta_k$ in (\ref{thetak}), we have (\ref{pu}) as desired.
\end{proof}

\vskip 2mm

\begin{lemma}\label{lemma3.5}
	%	Let $x^{\rm opt}$ be an arbitrary optimal solution of \eqref{orip}. Assume that the sequence $\{\beta_k\}_{k\geq1}$ is nondecreasing, and $\beta_{k}-L_{\mu_{k}}>0$. Then,
	For any $k\geq 1$, we have
	{
		\begin{eqnarray*}
			%\label{theorem3.1eq}
			%V_{\mu_k}} 
		V_{\mu_k} 
		%&= & \tilde \psi_{\mu_k}(y_k) 
		% -  {\tilde{\psi}_{\mu_k}(x^{\rm opt})}\\
		&\le&  (1-\alpha_{k-1}) 
		\left( 
		\tilde \psi_{\mu_k}(y_{k-1}) 
		- {\tilde{\psi}_{\mu_k}(x^{\rm opt})} 
		\right)
		+\frac{\theta_k \alpha_{k-1}}{2} (\|z_{k-1} - x^{\rm opt} \|^2 -\|z_k - x^{\rm opt}\|^2) \nonumber\\
		& & \nonumber \quad  +\frac{\|\delta_{\mu_k}(x_{k})\|^2}{2(\beta_{k}-L_{\mu_{k}})}
		+\alpha_{k-1}\left\langle \delta_{\mu_{k}}(x_{k}),x^{\rm opt}-z_{k-1}\right\rangle.
\end{eqnarray*}}
%and if Assumption \ref{assumption 1} holds, then
%\begin{eqnarray}\label{prop3.1eq}
%\mathbb{E}_{\mathcal{F}_{k}}\left[\Gamma_{k}\right]
%\leq \mathbb{E}_{\mathcal{F}_{k}}\left[R_{k}\right]+Q_{k}.
%\end{eqnarray}
%where
%%$R_k$, $Q_k$ and $\Gamma_{k}$ are defined in TABLE \ref{tablesymbol}, and
%the expectation $\mathbb{E}$ is taken with respect to $\mathcal{F}_k$.
\end{lemma}
\vskip 1mm
%\subsection{Proof of Lemma \ref{lemma3.5}}\label{ap3.5}
\begin{proof}
%In view of the definitions of $V_{\mu_k}$ and $P_k$ in TABLE \ref{tablesymbol}, it follows that
%\begin{eqnarray}\label{psimuykminusx}
%&&(1-\alpha_{k-1})\tilde{\psi}_{\mu_{k}}(y_{k-1})-\tilde{\psi}_{\mu_{k}}(x^{\rm opt})\nonumber\\
%&&~~~=~~(1-\alpha_{k-1})\left[\tilde{\psi}_{\mu_{k}}(y_{k-1})-\tilde{\psi}_{\mu_{k}}(x^{\rm opt})\right]-\alpha_{k-1}\tilde{\psi}_{\mu_{k}}(x^{\rm opt})\nonumber\\
%&&~~\overset{\eqref{psimurelation}}{\leq}(1-\alpha_{k-1})\left[\tilde{\psi}_{\mu_{k-1}}(y_{k-1})-\tilde{\psi}_{\mu_{k-1}}(x^{\rm opt})\right]
%+P_{k}-\alpha_{k-1}\tilde{\psi}_{\mu_{k}}(x^{\rm opt})\nonumber\\
%&&~~~\leq~~(1-\alpha_{k-1})V_{\mu_{k-1}}+P_{k}-\alpha_{k-1}\tilde{\psi}_{\mu_{k}}(x^{\rm opt}).
%\end{eqnarray}
By the convexity of $\tilde{\psi}_{\mu_{k}}$, we have for any $x\in X$,
\begin{eqnarray}\label{convex1}
	&&\tilde{\psi}_{\mu_{k}}(x_{k})\leq \tilde{\psi}_{\mu_{k}}(x)-\left\langle\nabla\tilde{\psi}_{\mu_{k}}(x_{k}),x-x_{k}\right\rangle.
\end{eqnarray}
Since $\tilde{\psi}_{\mu_{k}}$ is an $L_{\mu_{k}}$-smooth function over $X$, by setting $x=y_{k-1}$ {and $x=x^{\rm{opt}}$} respectively in \eqref{convex1}, we have
\begin{eqnarray*}
	&&\tilde{\psi}_{\mu_{k}}(y_{k})\nonumber\\
	&
	\overset{\eqref{psiLsmoo}}{\leq}&\tilde{\psi}_{\mu_{k}}(x_{k})+\left\langle\nabla\tilde{\psi}_{\mu_{k}}(x_{k}),y_{k}-x_{k}\right\rangle+\frac{L_{\mu_{k}}}{2}\|y_{k}-x_{k}\|^2\nonumber\\
	&=&(1-\alpha_{k-1})\tilde{\psi}_{\mu_{k}}(x_{k})+\alpha_{k-1}\tilde{\psi}_{\mu_{k}}(x_{k})+\left\langle\nabla\tilde{\psi}_{\mu_{k}}(x_{k}),y_{k}-x_{k}\right\rangle+\frac{L_{\mu_{k}}}{2}\|y_{k}-x_{k}\|^2\nonumber\\
	&\overset{\eqref{convex1}}{\leq}&(1-\alpha_{k-1})\tilde{\psi}_{\mu_{k}}(y_{k-1})
	-(1-\alpha_{k-1})\left\langle\nabla\tilde{\psi}_{\mu_{k}}(x_{k}),y_{k-1}-x_{k}\right\rangle+\alpha_{k-1}\tilde{\psi}_{\mu_{k}}(x^{\rm opt})\nonumber\\
	&&~~~~~~~~~		-\alpha_{k-1}\left\langle\nabla\tilde{\psi}_{\mu_{k}}(x_{k}),x^{\rm opt}-x_{k}\right\rangle+\left\langle\nabla\tilde{\psi}_{\mu_{k}}(x_{k}),y_{k}-x_{k}\right\rangle+\frac{L_{\mu_{k}}}{2}\|y_{k}-x_{k}\|^2.
	%&&~~= (1-\alpha_{k-1})\tilde{\psi}_{\mu_{k}}(y_{k-1})  +\left\langle\nabla\tilde{\psi}_{\mu_{k}}(x_{k}),y_{k}-\alpha_{k-1}x^{\rm opt}-(1-\alpha_{k-1})y_{k-1}\right\rangle +\frac{L_{\mu_{k}}}{2}\|y_{k}-x_{k}\|^2. 
\end{eqnarray*}

Subtracting 
{$\tilde{\psi}_{\mu_{k}}(x^{\rm opt}) + (1-\alpha_{k-1}) \left( \tilde \psi_{\mu_k}(y_{k-1}) - \tilde \psi_{\mu_k}(x^{\rm opt}) \right) $} from both sides of the above inequality, %using the definitions of $\Gamma_k$, $V_{\mu_k}$, and $\delta_{\mu_k}(x_k)$ in TABLE \ref{tablesymbol},
%and employing \eqref{psimuykminusx},
we get
\begin{eqnarray}\label{11-1}
	& &\tilde \psi_{\mu_k}(y_k) - \tilde \psi_{\mu_k}(x^{\rm opt}) 
	-  (1-\alpha_{k-1}) \left( \tilde \psi_{\mu_k}(y_{k-1}) - \tilde \psi_{\mu_k}(x^{\rm opt}) \right)  \\
	& \leq &
	\left\langle\nabla\tilde{\psi}_{\mu_{k}}(x_{k}),y_{k}-\alpha_{k-1}x^{\rm opt}-(1-\alpha_{k-1})y_{k-1}\right\rangle+\frac{L_{\mu_{k}}}{2}\|y_{k}-x_{k}\|^2\nonumber\\
	&=&\left\langle\nabla\tilde{\mathbf{\Psi}}_{\mu_{k}}^{k}-\delta_{\mu_k}(x_{k}),y_{k}-\alpha_{k-1}x^{\rm opt}-(1-\alpha_{k-1})y_{k-1}\right\rangle
	+\frac{L_{\mu_{k}}}{2}\|y_{k}-x_{k}\|^2\nonumber\\
	&\overset{\eqref{pu}}{\leq}& { \frac{\theta_k \alpha_{k-1}}{2} (\|z_{k-1} - x^{\rm opt} \|^2 -\|z_k - x^{\rm opt}\|^2)} -\frac{\beta_{k}-L_{\mu_{k}}}{2}\|y_{k}-x_{k}\|^2 
	\nonumber\\
	& & \quad\quad 
	+\left\langle \delta_{\mu_k}(x_{k}),x_{k}-y_{k}\right\rangle -\alpha_{k-1}\left\langle \delta_{\mu_k}(x_{k}),z_{k-1}-x^{\rm opt}\right\rangle,
\end{eqnarray}
where the last inequality is based on $x_{k}=\alpha_{k-1}z_{k-1}+(1-\alpha_{k-1})y_{k-1}$ in Step 1 of Algorithm \ref{alg:SSAG}. It is clear that $\left\langle \delta_{\mu_k}(x_{k}),x_{k}-y_{k}\right\rangle \le \|\delta_{\mu_k}(x_{k})\| \|x_{k}-y_{k}\|$, and $-\frac{b_2b_3^2}{2}+b_1b_3\leq\frac{b_1^2}{2b_2}$ for any nonnegative real numbers $b_1$, $b_3$, and positive real number  $b_2$. Then by setting $b_1=\|\delta_{\mu_k}(x_{k})\|$, $b_3=\|x_{k}-y_{k}\|$, and $b_2 =\beta_{k}-L_{\mu_{k}}$ that is positive according to \eqref{betak}, we have
\begin{eqnarray}\label{11-2}
	& & -\frac{\beta_{k}-L_{\mu_{k}}}{2}\|y_{k}-x_{k}\|^2
	+\left\langle \delta_{\mu_k}(x_{k}),x_{k}-y_{k}\right\rangle
	\leq\frac{\|\delta_{\mu_k}(x_{k})\|^2}{2(\beta_{k}-L_{\mu_{k}})}.
\end{eqnarray}
In view of \eqref{11-1} and \eqref{11-2}, it follows that Lemma \ref{lemma3.5} %\eqref{theorem3.1eq} 
holds.
\end{proof}
\vskip 3mm

\vskip 1mm
Theorem \ref{prop3.2} below considers the case that the SSAG adopts the diminishing smoothing parameter, %(SSAG-d for brevity), 
and provides an upper bound of the expectation of difference between the objective value at computed solutions and the optimal objective value.
\vskip 1mm
\begin{theorem}\label{prop3.2}
{Let $\{\mu_{k}\}_{k\geq 0}$, $\{\alpha_{k}\}_{k\geq 0}$, $\{\beta_{k}\}_{k\geq 1}$ and $\{\theta_{k}\}_{k\geq 1}$ be defined in \eqref{alphak}--\eqref{thetak} for Algorithm \ref{alg:SSAG}. Let $m_k=k$ for all $k\ge 1$. Suppose Assumption \ref{assumption 1} holds. }
Then for any $N\geq1$,  
\begin{eqnarray}\label{ineq-new-1}
	\mathbb{E}_{{\cal F}_N}[\psi(y_N) - \psi(x^{\rm opt})] \le \frac{d_1}{N^2} + \frac{d_2}{N},
\end{eqnarray}
where $d_1$ and $d_2$ are constants defined by 
\begin{eqnarray*}
	d_1&:=&{\frac{4 V_{\mu_1}}{\alpha_0^2} + 4(K+2)\|z_1 - x^{\rm opt}\|^2}\\
	d_2 &:=& { (32 \kappa \hat \mu + 4 \sigma^2)(1+\frac{1}{e}) + 2\kappa \hat \mu + 4\left(\frac{{\cal{L}}_{\psi}}{\hat \mu} + 2\right) \|z_1 - x^{\rm opt}\|^2}.
\end{eqnarray*}
\end{theorem}
\begin{proof}
{Using the definition of smoothing function, we know that 
	\begin{eqnarray*}
		\tilde \psi_{\mu_k}(y_{k-1}) - \tilde \psi_{\mu_k}(x^{\rm opt}) \le \tilde \psi_{\mu_{k-1}}(y_{k-1}) - \tilde \psi_{\mu_{k-1}}(x^{\rm opt}) + 2 \kappa (\mu_{k-1} - \mu_k).
	\end{eqnarray*}
	This, combined with Lemma \ref{lemma3.5}, yields that 
	\begin{eqnarray}\label{star}
		V_{\mu_k} &\le& (1- \alpha_{k-1}) V_{\mu_{k-1}} + P_k + \frac{\theta_k \alpha_{k-1}}{2}(\|z_{k-1} - x^{\rm opt}\|^2 - \|z_k - x^{\rm opt}\|^2) \nonumber\\
		& & \quad + \frac{\|\delta_{\mu_k}(x_k)\|^2}{2(\beta_k - L_{\mu_k})} + \alpha_{k-1} \langle \delta_{\mu_k}(x_k), x^{\rm opt} - z_{k-1} \rangle.
	\end{eqnarray}
	Dividing both sides of the above inequality by $\alpha_{k-1}^2$ and using $\frac{1-\alpha_{k-1}}{\alpha_{k-1}^2} = \frac{1}{\alpha_{k-2}^2}$, we get }
\begin{eqnarray}
	\frac{V_{\mu_k}}{\alpha_{k-1}^2} &\le& \frac{V_{\mu_{k-1}}}{\alpha_{k-2}^2} + \frac{P_k}{\alpha_{k-1}^2} + \frac{\theta_k}{2 \alpha_{k-1}} (\|z_{k-1} - x^{\rm opt}\|^2 - \|z_k - x^{\rm opt}\|^2) \nonumber\\
	& &\quad \frac{\|\delta_{\mu_k}(x_k)\|^2}{2(\beta_k - L_{\mu_k})\alpha_{k-1}^2} + \frac{1}{\alpha_{k-1}} \langle \delta_{\mu_k}(x_k), x^{\rm opt} - z_{k-1} \rangle \nonumber\\
	& \le &  \cdots \nonumber\\
	& \le & \frac{V_{\mu_1}}{\alpha_0^2} + \sum_{i=2}^k \frac{P_i}{\alpha_{i-1}^2} + \sum_{i=2}^k \frac{\theta_i}{ 2 \alpha_{i-1}}(\|z_{i-1} - x^{\rm opt}\|^2 - \|z_i - x^{\rm opt}\|^2) \nonumber \\
	& & \quad + \sum_{i=2}^k \frac{\|\delta_{\mu_i}(x_i)\|^2}{2 (\beta_i - L_{\mu_i}) \alpha_{i-1}^2} + \sum_{i=2}^k \frac{1}{\alpha_{i-1}} \langle \delta_{\mu_i}(x_i), x^{\rm opt} - z_{i-1} \rangle.  \label{Vmukk-1}
\end{eqnarray}

Now we derive upper bounds for $\alpha_{k-1}^2$ times each of the first three terms in the right-hand side of (\ref{Vmukk-1}).  In view of $\alpha_{k-1} \le \frac{2}{k+1}$, it follows that 
\begin{eqnarray}\label{item1}
	\alpha_{k-1}^2\frac{V_{\mu_1}}{\alpha_0^2} \le \frac{4 V_{\mu_1}}{(k+1)^2 \alpha_0^2}.
\end{eqnarray}
By using $\alpha_{i-2} \ge \frac{1}{2i}$ according to (\ref{alphak_p1}), we have 
\begin{eqnarray*}
	\sum_{i=2}^k \frac{P_i}{\alpha_{i-1}^2} &=& \sum_{i=2}^k \frac{2 \kappa(1-\alpha_{i-1})}{\alpha_{i-1}^2}  (\mu_{i-1}-\mu_i)\\
	%&=& 2 \kappa \sum_{i=2}^k \frac{1}{\alpha_{i-2}^2}(\mu_{i-1} - \mu_i)\\
	&=& 2 \kappa \hat \mu \sum_{i=2}^k \frac{1}{\alpha_{i-2}^2}\left(\frac{1}{i-1} - \frac{1}{i} \right)\\
	&\le& 8 \kappa \hat \mu \sum_{i=2}^k i^2 \left(\frac{1}{i-1} - \frac{1}{i} \right)\\ 
	&=& 8 \kappa \hat \mu \sum_{i=2}^k
	\left(1+\frac{1}{i-1}\right)\\
	&\le& 8 \kappa \hat \mu [k-1 + \ln(k-1)],
\end{eqnarray*}
where the last inequality is obtained since 
\begin{eqnarray*}
	\sum_{i=2}^k \frac{1}{i-1} \le \int_{1}^k \frac{1}{x-1} dx = \ln(k-1).
\end{eqnarray*}
By direct computation, $\left(\frac{\ln t}{t}\right)' = \frac{1-\ln t}{t^2}$. Consequently $\left(\frac{\ln t}{t}\right)' \ge 0$ for $t\in (0,e]$ and $\left(\frac{\ln t}{t}\right)' <0$ for $t\in (e,\infty)$,  which indicates that $\frac{\ln t}{t}$ achieves the global maximum at $t=e$ on the interval $(0,\infty)$, i.e.,
\begin{eqnarray*}
	\frac{\ln t}{t}  \le \frac{1}{e}\quad \mbox{for any}\ t\in (0,\infty).
\end{eqnarray*}
This, together with $\alpha_{k-1} \le \frac{2}{k+1}$, yields that 
\begin{eqnarray}\label{1}
	\alpha_{k-1}^2 \sum_{i=2}^k \frac{P_i}{\alpha_{i-1}^2} &\le& \left( \frac{2}{k+1}\right)^2 8 \kappa \hat \mu [k-1 + \ln(k-1)]  \nonumber \\
	&\le & 32 \kappa \hat \mu \left[\frac{1}{k+1} + \frac{\ln(k+1)}{k+1} \frac{1}{k+1} \right] \nonumber\\
	&\le& \frac{32 \kappa \hat \mu \left(1+ \frac{1}{e}\right)}{k+1}. 
\end{eqnarray}
Furthermore, by using $\alpha_{i-1} \ge \frac{1}{2(i+1)}$ in (\ref{alphak_p1}), $\theta_i = 2 \alpha_{i-1} \beta_i$ in (\ref{thetak}),  we have 
\begin{eqnarray*}
	\frac{\theta_i}{2 \alpha_{i-1}} &=& \frac{2 \alpha_{i-1}\beta_i}{2 \alpha_{i-1}} = \beta_i = L_{\mu_i} + \frac{1}{\alpha_{i-1}} = K+ \frac{{\cal L}_{\psi}}{\mu_i} + \frac{1}{\alpha_{i-1}}\\
	%&=& K + \frac{L_{\psi}}{\hat \mu} + \frac{1}{\alpha_{i-1}} 
	&\le& K + \frac{{\cal L}_{\psi}}{\hat \mu} i + 2(i+1)\\
	&=& K+2 + \left(\frac{{\cal L}_{\psi}}{\hat \mu} +2  \right) i \\
	&\le& K+2 + \left(\frac{{\cal L}_{\psi}}{\hat \mu} +2  \right) k\quad \mbox{for all}\ i\le k.
\end{eqnarray*}
It follows that 
\begin{eqnarray*}
	& & \sum_{i=2}^k \frac{\theta_i}{2 \alpha_{i-1}} (\|z_{i-1}-x^{\rm opt}\|^2 - \|z_i - x^{\rm opt}\|^2)\\
	&\le& \left[K+2 + \left(\frac{{\cal L}_\psi}{\hat \mu} + 2\right) k \right] \sum_{i=2}^k (\|z_{i-1}-x^{\rm opt}\|^2 - \|z_i - x^{\rm opt}\|^2)\\
	&\le & \left[K+2 + \left(\frac{{\cal L}_\psi}{\hat \mu} + 2\right) k \right] \|z_1 - x^{\rm opt}\|^2.  
\end{eqnarray*}
Therefore, we find
\begin{eqnarray} \label{2}
	& &\alpha_{k-1}^2 \sum_{i=2}^k \frac{\theta_i}{2 \alpha_{i-1}} (\|z_{i-1}-x^{\rm opt}\|^2 - \|z_i - x^{\rm opt}\|^2) \nonumber\\ 
	&\le& \left(\frac{2}{k+1} \right)^2 \left[K+2 + \left(\frac{{\cal L}_\psi}{\hat \mu} + 2\right) k \right] \|z_1 - x^{\rm opt}\|^2 \nonumber\\
	&\le& \quad \frac{4(K+2)\|z_1 - x^{\rm opt}\|^2}{(k+1)^2} + \frac{4 \left(\frac{{\cal L}_{\psi}}{\hat \mu} +2 \right) \|z_1 - x^{\rm opt}\|^2}{k+1}.
\end{eqnarray}

Let us denote the sum of the right-hand sides of (\ref{item1}), (\ref{1}) and (\ref{2}) by $A_k$.
Multiplying both sides of \eqref{Vmukk-1} by $\alpha_{k-1}^2$ and taking the expectation on both sides with respect to ${\cal F}_k$, we get 
\begin{eqnarray}\label{main-ineq}
	{\mathbb{E}_{{\cal{F}}_k}[V_{\mu_k}]} 
	&\le&   
	%{\frac{4 V_{\mu_1}}{(k+1)^2 \alpha_0^2}  +  \frac{32 \kappa \hat \mu \left(1+ \frac{1}{e}\right)} {k+1} +  \frac{4(K+2)\|z_1 - x^{\rm opt}\|^2}{(k+1)^2} + \frac{4 \left(\frac{{\cal L}_{\psi}}{\hat \mu} +2 \right) \|z_1 - x^{\rm opt}\|^2}{k+1} }
	A_k + {\alpha_{k-1}^2} \mathbb{E}_{{\cal F}_k} \left[\sum_{i=2}^k 
	\frac{
		\|\delta_{\mu_i}(x_i)\|^2}{2 (\beta_i - L_{\mu_i}) \alpha_{i-1}^2}\right] \nonumber \\
	& & \quad   +
	{\alpha_{k-1}^2}\mathbb{E}_{{\cal F}_k} \left[\sum_{i=2}^k \frac{1}{\alpha_{i-1}} \langle \delta_{\mu_i}(x_i), x^{\rm opt} - z_{i-1} \rangle \right].
\end{eqnarray}

{
	By Lemma \ref{lemma3.1} (b), 
	$\beta_i - L_{\mu_i} = \frac{1}{\alpha_{i-1}}$,
	%the definition of $Q_k$ in Table \ref{tablesymbol}, 
	{$m_i = i$} and $\alpha_{i-1} \ge \frac{1}{2(i+1)}$, we have 
	\begin{eqnarray*}
		\mathbb{E}_{{\cal{F}}_k} \left[ \sum_{i=2}^k \frac{\|\delta_{\mu_i}(x_i)\|^2}{2(\beta_i - L_{\mu_i}) \alpha_{i-1}^2} \right] 
		&=& \sum_{i=2}^k \frac{\mathbb{E}_{{\cal F}_k}[\|\delta_{\mu_i} (x_i)\|^2]}{2(\beta_i - L_{\mu_i}) \alpha_{i-1}^2} \\
		&\le&   \frac{\sigma^2}{2} \sum_{i=2}^k \frac{1}{\alpha_{i-1} m_i} \\ 
		&\le& \sigma^2 \sum_{i=2}^k \frac{i+1}{i} \\
		&\le& \sigma^2 \left[k-1+ \sum_{i=2}^k \frac{1}{i} \right] \\
		&\le & \sigma^2(k-1+ \ln k).  
	\end{eqnarray*}
}
It is then easy to derive that 
\begin{eqnarray}\label{sun}
	\alpha_{k-1}^2 \mathbb{E}_{{\cal{F}}_k} \left[ \sum_{i=2}^k \frac{\|\delta_{\mu_i}(x_i)\|^2}{2(\beta_i - L_{\mu_i}) \alpha_{i-1}^2} \right] &\le& \left(\frac{2}{k+1} \right)^2 \sigma^2[(k-1) + \ln k] \nonumber\\
	&\le& 4 \sigma^2 \left(\frac{1}{k+1} + \frac{\ln k}{k+1} \frac{1}{k+1} \right) \nonumber\\
	&\le &  \frac{4 \sigma^2 (1+\frac{1}{e})}{k+1}.
\end{eqnarray}

Since $x^{\rm opt}$, $z_{i-1}=z_{i-1}(\mathcal{F}_{i-1})$ and $x_{i}=x_{i}(\mathcal{F}_{i-1})$ for any $i\le k$ are deterministic if $\mathcal{F}_{k}$ is given, we have by the definition of $\delta_{\mu_k}(x_k)$ in (\ref{deltaV_muk}), Lemma \ref{lemma3.1} (a), and by the similar arguments in the proof of Lemma 3 of \cite{wang2022stochastic},
\begin{eqnarray}\label{delta0}
	\mathbb{E}_{\mathcal{F}_{k}}\left[\left\langle \delta_{\mu_i}(x_{i}),x^{\rm opt}-z_{i-1}\right\rangle\right]
	=0,\quad \forall \ i\le k.
\end{eqnarray}
%where the first equality is based on Theorem 3.24 of %\cite{wasserman2004all}, and the last second equality follows from Lemma \ref{lemma3.1} (a).

%Multiplying both sides of (\ref{main-ineq}) by $\alpha_{i-1}^2$, 
Setting $k=N$ in (\ref{main-ineq}), employing \eqref{deltaV_muk}, (\ref{sun}), (\ref{delta0}), and the definition of $A_k$, and using the fact \eqref{psibound} that 
\begin{eqnarray*}
	\psi(y_N) - \psi(x^{\rm opt}) \le V_{\mu_N} + 2 \kappa \mu_N,
\end{eqnarray*}
we find 
\begin{eqnarray}
	& & \mathbb{E}_{{\cal F}_N}[\psi(y_N) - \psi(x^{\rm opt})] 
	\le \mathbb{E}_{{\cal F}_N} [V_{\mu_N}] + 2 \kappa \mu_N \nonumber\\
	&\le& \left(\frac{2}{N+1}\right)^2 \left(\frac{V_{\mu_1}}{\alpha_0^2}  \right) + \frac{32\kappa \hat \mu(1+\frac{1}{e})}{N+1}  + \frac{4(K+2)\|z_1 - x^{\rm opt}\|^2}{(N+1)^2} \nonumber\\
	& & \quad + \frac{4(\frac{L_{\psi}}{\hat \mu} + 2)\|z_1 -x^{\rm opt}\|^2}{N+1} + \frac{4\sigma^2(1+\frac{1}{e})}{N+1} + \frac{2\kappa \hat \mu}{N}\label{pp}\\
	&\le & \frac{d_1}{N^2} + \frac{d_2}{N},\nonumber
\end{eqnarray}
as we desired.

\end{proof}

{ Theorem \ref{prop3.2} guarantees asymptotic convergence  of the SSAG method that adopts the diminishing smoothing parameter and variable mini-batch size of samples.
{ Corollary \ref{corbatchk}  below shows that our SSAG method can simultaneously own the iteration complexity   ${\cal O}(\frac{1}{\epsilon})$, and the $\cal{SFO}$ complexity  ${\cal O}(\frac{1}{\epsilon^2})$ to get an $\epsilon$-approximate solution of the original problem \eqref{orip}.}}
\vskip 1mm

%%%
%%%
%%%
%%%

\vskip 1mm 

\vspace{2mm}
\begin{corollary}\label{corbatchk}
Suppose that all the conditions in Theorem \ref{prop3.2} are satisfied and the constants $d_1$ and $d_2$ are defined as in Theorem \ref{prop3.2}. %Let the batch size at the $k$-th iterate is 
%\begin{eqnarray*}
%m_k= k.
%\end{eqnarray*}
Then the iteration complexity of finding an $\epsilon$-approximate solution of the original problem \eqref{orip} is
\begin{eqnarray}\label{N1}
N=\left\lceil	
\frac{d_2 + \sqrt{d_1}\sqrt{\epsilon}} {\epsilon} \right\rceil,
\end{eqnarray}
which is of order $\mathcal{O}\left(\frac{1}{\epsilon}\right)$. Moreover, the $\cal SFO$ complexity of finding an $\epsilon$-approximate solution of the original problem \eqref{orip} is $\mathcal{O}\left(\frac{1}{\epsilon^{2}}\right)$.
\end{corollary}

\begin{proof}
By Theorem \ref{prop3.2},  in order to get an $\epsilon$-approximate solution, we need
\begin{eqnarray*}
\frac{d_1}{N^2}+\frac{d_2}{N}\leq \epsilon.
\end{eqnarray*}
Then, {we can choose $N$ in  \eqref{N1} to satisfy the above inequality, because}
\begin{eqnarray*}
N = \left\lceil	
\frac{d_2 + \sqrt{d_1}\sqrt{\epsilon}} {\epsilon} \right\rceil \ge \frac{d_2 + \sqrt{d_2^2+ 4 \epsilon d_1}}{2\epsilon}:= b
\end{eqnarray*}
and 
\begin{eqnarray*}
\epsilon b^2 - d_2 b - d_1 =0, 
\end{eqnarray*}
implies 
\begin{eqnarray*}
\epsilon N^2 - d_2 N -d_1 >0, \quad  \mbox{i.e.}, \ \frac{d_1}{N^2} + \frac{d_2}{N} < \epsilon.
\end{eqnarray*}

In this case, the order of the $\mathcal{SFO}$ complexity for finding an $\epsilon$-approximate solution is 
\begin{eqnarray*}
\overline{N}=\sum_{k=1}^Nm_k= \sum_{k=1}^N k = \frac{N(N+1)}{2} = \mathcal{O}\left( \frac{1}{\epsilon^2}\right) .
\end{eqnarray*}
\end{proof}

{Now we give remarks about the merits of our SSAG method, compared with other state-of-the art SA-type methods.}
\begin{remark}
To the best of our knowledge, the SSAG method is the first SA-type method, that simultaneously achieves the best-known order ${\cal{O}}(\frac{1}{\epsilon})$ of iteration complexity, and the optimal order ${\cal{O}}(\frac{1}{\epsilon^2})$ of $\cal{SFO}$ complexity for (\ref{orip}), as shown in Corollary \ref{corbatchk}. 
Here the diminishing smoothing parameter and the variable mini-batch size $k$ of samples at iterate $k$ are adopted.  

The sVS-APM method in \cite{jalilzadeh2022smoothed} can achieve the optimal order of iteration complexity, but it uses larger variable mini-batch size $\lfloor k^{1+\delta} \rfloor$ of samples for an arbitrary $\delta>0$ at iterate $k$. According to Theorem 4 of \cite{jalilzadeh2022smoothed}, if sVS-APM uses the variable mini-batch size $k$ at the $k$-th iterate as our SSAG method, the order of iteration complexity is worse than ${\cal O}(\frac{1}{\epsilon})$, because it requires the number of iterations $N$ to satisfy
$${\cal{O}} \left(\frac{\log N}{N} \right)\le \epsilon.$$
Moreover,  the  order of ${\cal SFO}$ complexity of  sVS-APM method in \cite{jalilzadeh2022smoothed} is worse than that of our SSAG method, because it  can only be arbitrarily near optimal order ${\cal O}(\frac{1}{\epsilon^2})$ of $\cal{SFO}$ complexity.

The SSAG method with the diminishing smoothing parameter   achieves the optimal order ${\cal{O}}(\frac{1}{\epsilon^2})$ of $\cal{SFO}$ complexity in order to  find an $\epsilon$-approximate solution, no matter it adopts the variable mini-batch size of samples (Corollary \ref{corbatchk}). 
The order is the same as
%$\mathcal{O}(\frac{1}{\epsilon^2})$
that obtained by the state-of-the-art SA algorithms \cite{nemirovski2009robust,lan2012optimal,ghadimi2012optimal,ghadimi2016accelerated,  wang2022stochastic}, and the above algorithms do not provide the best-known order of iterate  complexity.  It is worth mentioning that our SSAG method does not require the nonsmooth components to have easily obtainable proximal operators, as required in \cite{nemirovski2009robust,lan2012optimal,ghadimi2012optimal,ghadimi2016accelerated}. For a predetermined accuracy $\epsilon>0$ and the choice of the batch size $m_k$, the iteration limit $N$ is calculated by \eqref{N1}. Hence the SSAG method with the diminishing smoothing parameter does not require a prior knowledge of the number of iterations $N$ to be performed, thereby
rendering it suitable to online and streaming applications. 
\end{remark}

\begin{remark}
Unlike the RS method given by Duchi et al. in \cite{duchi2012randomized} and the SNSA method introduced by Wang et al. in \cite{wang2022stochastic} that restricted to a certain smoothing technique, our SSAG method is capable of enrolling general smoothing techniques that can be dimension-independent as Nesterov's smoothing and inf-conv smoothing.   The SSAG method allows the diminishing smoothing parameter and consequently has asymptotic convergence, while the SNSA method only allows constant smoothing parameter and can only guarantee an $\epsilon$-approximate solution.

Unlike the bounds in the ${\cal{SFO}}$ complexity results of the RS method (Corollaries 2.3-2.6 of \cite{duchi2012randomized}) that are dimension-dependent, the bounds of our complexity results (Theorem \ref{prop3.2}, Corollary \ref{corbatchk} can be dimension-independent if proper smoothing techniques are employed, which is beneficial for high-dimensional problems.

In fact, whether the bounds for our complexity results are dimension-dependent or dimension-independent is determined by the constant $\kappa$ and $\sigma$. For instance, if the inf-conv smoothing is adopted with $\omega$ defined in (\ref{b-omega}), then according to Lemma \ref{lemma2.1} and (\ref{gradient-omega}) we know that $\kappa = \omega(0)= \ln q$ and $\sigma \le 1$. Consequently the bounds are dimension-independent. In contrast, if the RS technique is employed with auxiliary random vector $v\sim {\cal{N}}(0,I_{d})$, then according to Lemma  \ref{lemma2.21}, $\kappa = L_0 \sqrt{d}$ and hence the bounds are dimension-dependent. 
%It is interesting to find that if the RS technique is employed with $v$ follows the uniform density function, according to Lemma \ref  , the bounds are also dimension-dependent.

The above observations provide affirmative answers to the questions in Section 5 of \cite{duchi2012randomized} that dimension-dependent smoothing techniques and dimension-dependent bounds are possible.

\end{remark}

\section{Applications and numerical results}\label{sec:numerical}
We do numerical experiments on {the} DRO-moment problem \cite{xu2018distributionally} 
in {Sect.} \ref{subsec:DROportfolio}. Among the application, the transformation of the DRO-moment problem to the problem (\ref{orip}) where $h$ is defined in (\ref{max}) is given in detail, which can be seen as an extension of \cite{wang2023distributionally}. We compare our SSAG method with the following state-of-the-art methods. 
\begin{itemize}
\item[(1)]SSAG: The batch sizes $m_k= k$, and the iteration limit $N$ is computed by using \eqref{N1}. Summing up the batch sizes $m_k$ ($k=1,\ldots,N$) chosen for the SSAG method results in $\overline{N}$, i.e., the budget for total number of calls to {the} $\mathcal{SFO}$. We set $\overline{N}$ to be the same for different methods.
%\item[(2)]SNSA \cite{wang2022stochastic}: The smoothing parameter, the batch size, and the stepsize choices in the SNSA method follow from Theorem 2 and Corollary 2 of \cite{wang2022stochastic}.
%% The smoothing parameter uses a fixed setting. By the column of ``Nonsmoothness" of Table \ref{tableconvex1}, the SNSA method is used to solve the Wasserstein DRSVM problem in {Sect.} \ref{subsec:drsvm}.
\item[(2)]SA: The batch sizes $m_k= k$, and the stepsize rule follows $\mathcal{O}(\frac{1}{k})$ in \cite{nemirovski2009robust}.
%By the column of ``Nonsmoothness" of Table \ref{tableconvex1}, the S-Subgrad method
\item[(3)]Subgrad: The Armijo's stepsize rule is used instead of $\mathcal{O}(\frac{1}{k})$ in \cite{nemirovski2009robust} because the computational performance is better than that using %batch size $m=1$ and
stepsize $\mathcal{O}(\frac{1}{k})$ in \cite{nemirovski2009robust}. The batch size $m=N_{\rm{tr}}$ in each iteration, where $N_{\rm{tr}}$ is the number of the training samples.
%\item[(5)]sVS-APM \cite{jalilzadeh2022smoothed}: %The sVS-APM method combines smoothing techniques with the variable sample-size accelerated proximal gradient scheme (VS-APM). %This method considers the smoothing function for the nonsmooth function $\max\{z_1,\ldots,z_q\}$ as $\mu \ln \left(\sum\limits_{i=1}^q e^{\frac{z_i}{\mu}}\right)$. 
%The formula of diminishing sequence of smoothing parameters, the choice of stepsizes, and the parameters in sVS-APM are derived from Theorem 4 of \cite{jalilzadeh2022smoothed}.

\item[(4)]RS \cite{duchi2012randomized}: 
%The RS method combines the convolution technique and the random sampling technique to get a smooth approximation. It requires enrolling an 
The auxiliary random vector $v\in \mathbb{R}^{d}$ follows the  probability density function $\rho$ that is uniform on $B(0,1)$. The initial smoothing parameter, the formula for the diminishing sequence of smoothing parameters, the stepsize choices, and the parameters in the RS method follow from Corollary 2.3 of \cite{duchi2012randomized}. %{At iteration $k$, a stochastic gradient of the RS method besides draws i.i.d. $m$ samples from the distribution for $\zeta$. It still needs to be approximated by the averaging of subgradients at $m$ samples of $x_k+ \mu_k v_{i,k}$, where $v_{i,k}$, $i=1,\ldots,m$ are i.i.d. samples drawn according to the uniform distribution on $B(0,1)$.}
%The smoothing parameter is diminishing.
%By the column of ``Nonsmoothness" of Table \ref{tableconvex1}, the RS method
%\item[(1)]GraFuS: The gradient and function sampling method for Problem \eqref{CB1} was proposed by Helou et al. \cite{Helou}.
%\item[(7)]IPPA \cite{li2020fast}: It is a deterministic method.
%%The incremental proximal point algorithm for \eqref{DRSVMorip} was proposed by Li et al..
%We use the code of IPPA that is available at https://github.com/ gerrili1996/Incremental\_DRSVM. Each iteration of IPPA requires solving $N_{\rm{tr}}$ single-sample proximal point subproblems. %The stepsize of each iteration $\alpha_{k}=\rho\alpha_{k-1}$, where $\rho=0.965$ and $\alpha_{1}=1\text{e}$-2 is setted in the above code.
\item[(5)]{SPG \cite{zhang2009smoothing,zhang2020smoothing}: The smoothing projected gradient method (SPG) is a deterministic method, which was first proposed in \cite{zhang2009smoothing}, and extended in \cite{zhang2020smoothing}.} 
%The setting of SPG follows from (24) of \cite{wang2023distributionally}. 
% \item[(9)]CP \cite{xu2018distributionally}: %The cutting plane method for \eqref{DRO} was employed by Xu et al. \cite{xu2018distributionally}. %The algorithmic procedures follow the classical cutting plane method by Kelley \cite{kelley1960cutting} that involves a minimization subproblem and a maximization subproblem. The only minor difference is that the minimization subproblem involves some matrix variables, and the maximization subproblem has to be solved by a semidefinite programming (SDP) solver.
% %The SDP subproblems are solved by MATLAB solver. % ``SDPT3-4.0" \cite{toh1999sdpt3}.
% %We add \eqref{stop} as an additional stopping criterion based on the original stopping criterion in \cite{xu2018distributionally}.
% It is a cutting plane method {that is deterministic.} %with matrix moments
%in {Sect.} \ref{subsec:DROportfolio}.
\end{itemize}

%\begin{table}[htp]\caption{Applications and algorithms compared}\label{tablealg}
%\begin{tabular}{ll}\toprule
%Applications & Algorithms compared\\\midrule
%{Sect.} \ref{subsec:drsvm}\quad Wasserstein DRSVM problem \quad\quad & (1), (2), (3), (5), (6), (7)	\\ \midrule
%{Sect.} \ref{subsec:DROportfolio} \quad DRO-moment problem \quad\quad  & (1), (4),  (6), (8) \\\midrule
%{Sect.} \ref{subsec:sup} \quad Stochastic utility problem \quad\quad &(1), (3), (5), (6) \\\bottomrule
%\end{tabular}
%\end{table}

{The batch size $m$ in RS is determined in the following way.} For different batch sizes, we make 5 runs of the method under a CPU time limit of 50 seconds (50s) for each run, and select the optimal batch size so that the average objective value decreases the fastest with this $m$. {For RS, $m$ is tried in the set $\{10^{l},~l=0,1,2,3,4 \}$, respectively.} 
%The CPU time limit for each run is 50s. The optimal batch size is finally selected because it leads to the fastest decrease in the objective value. 
%Here larger batch sizes are tried in RS, because RS employs the additional random vector $v\in \mathbb{R}^{d}$. 
Then the iteration limit $N$ of {RS} can be calculated as $N=\left\lfloor \frac{\overline{N}}{m}\right\rfloor$. Since Subgrad and SPG are deterministic methods, their batch size $m$ equals the number of training samples {of} the problem. Additionally, their iteration limits $N$ are computed in the same way as RS. Denote by $N_{\rm{tol}}$ the {total} sample size, and by $N_{\rm{tr}}$ the number of training samples.
%We use the $N_{\rm{tr}}$ training samples to estimate the parameter, $\sigma^2$.
We follow the way of estimating the parameter $\sigma^2$ as in \cite{ghadimi2016mini}. To be specific, using the training samples, we compute the stochastic gradients of the objective function $\lceil N_{\rm{tr}}/100\rceil$ times at 100 randomly selected points and then take the average of the variances of the stochastic gradients for each point as an estimation of $\sigma^2$.

%For the Wasserstein DRSVM and {the} stochastic utility problem, we estimate an approximate solution 
%$\breve{y}$ 
%by running S-Subgrad 10000 iterations as in \cite{bai2022inexact}. %{(
For the DRO-moment problem, we estimate an approximate solution
$\breve{y}$ 
by running Subgrad 10000 iterations.
%since we use all training set in every iteration.
%)}
%Let us denote $\psi(\breve{y})$ be the  approximate optimal value. %
We follow the way of giving a CPU time budget as in \cite{bai2021convergence,bai2022inexact}. For each problem, we make 20 runs of each method under the CPU time budget 200s for each run in order to have statistical relevance of the results. %For each of the deterministic methods (IPPA and CP), only one run is performed under the CPU time budget 200s. 
Let $y_k$ be the $k$-th {iterate} point of a certain method. We also use
\begin{eqnarray}\label{stop}
\psi(y_k)-\psi(\breve{y})\leq\epsilon
\end{eqnarray}
to stop the method. If one of the three conditions (CPU time budget, inequality \eqref{stop}, the total number {$\bar N$} of calls to $\cal{SFO}$) is met, then we stop the method. 

All experiments are performed in Windows 11 on an AMD Ryzen 9 7900X 12-Core CPU at 4.70 GHz with 32 GB of RAM, using MATLAB R2024b.

\subsection{DRO-moment problem}\label{subsec:DROportfolio}
Let ${S=(\xi_1,\ldots,\xi_q)}$ be the data matrix composed by historical data, $\hat{\varsigma}$ and $\hat{\Sigma}$ be the mean vector and covariance matrix of ${S}$.  As in \cite{delage2010distributionally}, we assume that $\hat{\Sigma}$ is a symmetric positive definite matrix. Let $\mathcal{M}$ be the convex set of all probability measures $P$ in the measurable space $(\Theta, \mathds{B})$, with $\Theta \subseteq \mathbb{R}^{n+1}$ being a convex compact set known to contain the support of $P$, and $\mathds{B}$ being the Borel $\sigma$-algebra on $\Theta$. 

We consider the DRO-moment model as follows:
\begin{eqnarray}\label{nonDROorip}
%(\mathrm{DRO-moment model})\qquad
& & \min\limits_{\hat{x}\in \hat{X}}\max\limits_{P\in\mathscr{P}}~\left\lbrace  \mathbb{E}_{P}[\varphi(\hat{x},\xi)] + \phi(\hat{x})\right\rbrace,
%+ \tau_1 \phi_1(x)+\tau_2\phi_2(x)
\end{eqnarray}
where $\hat{X}$ is a compact convex set, 
%e.g., $\hat{X}=\left\{x\mid\sum\limits_{i=1}^d x_i=1,~ x\geq 0\right\}$, 
$\varphi(\cdot,\xi)$ is convex and possibly nonsmooth for every random vector $\xi\in \Theta$, $\phi$ is convex and possibly nonsmooth, and the ambiguity set $\mathscr{P}$ is defined as 
\begin{eqnarray}
\label{ambiguityset}
\quad\quad 	\mathscr{P}=
\left\{ \begin{array}{l|l}
	P\in \mathcal{M} & \begin{array}{l}
		\mathbb{E}_{P}\left[\begin{aligned}
			&-\hat{\Sigma}&\hat{\varsigma}-\xi\\
			&(\hat{\varsigma}-\xi)^{T} &-\hat{t}_{1} 
		\end{aligned} \right] \preceq 0\\
		\mathbb{E}_{P}\left[ (\xi-\hat{\varsigma})(\xi-\hat{\varsigma})^{T}\right] \preceq \hat{t}_{2} \hat{\Sigma}\\
		\mathbb{E}_{P}\left[ 1\right] =1
	\end{array}
\end{array} \right\},
%\end{eqnarray}
\end{eqnarray}
with  $\hat{t}_{1},~\hat{t}_{2}>0$. 
%The DRO-moment model has been intensively studied by Delage and Ye in \cite{delage2010distributionally} and by Xu et al. \cite{xu2018distributionally}.

The DRO-moment model is broadly applicable across various domains due to the versatility in selecting $\varphi(\cdot,\xi)$ and $\phi(\cdot)$. For example, 
\begin{eqnarray*}
& &\varphi_1(\hat{x},\xi) = \|A_{\xi} \hat x - b_{\xi}\|^2,\quad \varphi_2(\hat{x},\xi)=%=\hat{c}\max\left\{1-\xi_a(\hat{y}^T\xi_{B}+\hat{b}),0\right\},
\|\max\{A_{\xi} \hat x - b_\xi, 0\}\|;\\ 
& & \phi_1(\hat{x})=\varrho_1\|\hat{x}\|_{1},\quad \phi_2(\hat{x})=\varrho_2\|\hat{x}\|^2,
\end{eqnarray*}
where $ \varrho_1, \varrho_2$ are positive constants, and  $A_{\xi}$, $b_{\xi}$ are a stochastic matrix and a stochastic vector, respectively.
Here $\varphi_1(\hat x,\xi)$  and $\varphi_2(\hat x,\xi)$ are the data fidelity terms using least squares and max operator, respectively.
%and $\phi_2(\hat x, \xi)$ is used in SVM \cite{lee2015distributionally} for encouraging correct classification with $\hat{x}=(\hat{y}^T,\hat{b})^T$being the unknown vector and $\xi=(\xi_B^T,\xi_a)^T$ being the pair of the stochastic feature vector and class label. 
The functions $\phi_1(\hat x)$ and $\phi_2(\hat x)$ are introduced to promote sparsity and enhance out-of-sample performance, respectively.

%$in the distributionally robust index-tracking portfolio optimization problem with the CVaR penalty %\eqref{DROorip}
%\begin{eqnarray*}
%\varphi(\hat{x},\xi)=\left\|\xi_a-\hat{z}^{T}\xi_{B}\right\|^2+\frac{\tau_{2}}{1-\beta}\left[-\hat{z}^{\top}\xi_{B}-\alpha\right]_{+},
%\end{eqnarray*}
%where $\hat{x}:=(\hat{z}^T,\alpha)^T$, $\xi=(\xi_B^T,\xi_a)^T$, $\xi_B$ means the individual return vector of the $n-1$ assets, and $\xi_{a}\in \mathbb{R}$ is the corresponding random market index return. 

%{In the distributionally robust SVM model (model (4) of \cite{lee2015distributionally})
%where $\hat{C}$ is a constant, . The penalty terms $$$$
% Here $\tau_1,~\tau_2\ge 0$.}

Let $\mathbf{x_{\Lambda}}:=(\hat{x},\delta,\Lambda)\in \hat{X}\times\mathbb{R}^{n+1}\times\mathbb{S}_+^{n+1}$. Using 
%the similar arguments for tractable reformulation in 
Lemma 1 of \cite{delage2010distributionally} by Delage and Ye, we equivalently transform the inner max problem of \eqref{nonDROorip} to an equivalent semi-infinite problem by the Lagrange dual and {the strategy that further simplifies this dual problem by solving analytically for the $(n+2)^2$
%the delicate strategy for deleting one vector of 
dual variables corresponding to the first constraint in (\ref{ambiguityset}) using an auxiliary vector of $n+1$ variables, while keeping the other dual variables. In fact, this strategy reduces $n^2+3n+3$  unknown variables, at the expense of adding a nonsmooth term in the objective function. The mathematical model of the equivalent} semi-infinite program of \eqref{nonDROorip} is  
\begin{eqnarray}\label{semiinfi}
\begin{aligned}
	\min\limits_{\mathbf{x_{\Lambda}}}\quad&h_1(\mathbf{x_{\Lambda}})+r\\
	\rm{s.t.}\quad& h_{2,\xi}(\mathbf{x_{\Lambda}})\leq r ,~\forall\xi\in\Xi,\\
	\quad&  \mathbf{x_{\Lambda}}\in \hat{X}\times\mathbb{R}^{n+1}\times\mathbb{S}_+^{n+1}.
\end{aligned}
\end{eqnarray}
Here $h_1(\mathbf{x_{\Lambda}})$ and $h_{2,\xi}(\mathbf{x_{\Lambda}})$ are possibly nonsmooth functions defined as follows
\begin{eqnarray*}
h_1(\mathbf{x_{\Lambda}})&=&\left\langle\hat{t}_{2}\hat{\Sigma},\Lambda\right\rangle+(\Lambda^{T}\hat{\varsigma}+\delta)^{T}\hat{\varsigma}+\sqrt{\hat{t}_{1}}\left\|\hat{\Sigma}^{1/2} (\delta + 2 \Lambda {\hat{\varsigma}}) \right\|+\phi(\hat{x}),\\
h_{2,\xi}(\mathbf{x_{\Lambda}})&=&\varphi(\hat{x},\xi)-(\Lambda^{T}\xi+\delta)^{T}\xi,
\end{eqnarray*}
where $\hat{\Sigma}^{1/2}$ stands for the symmetric positive definite square root of 
%the positive definite matrix 
$\hat{\Sigma}$.
By substituting the semi-infinite constraints
%$ h_{2,\xi}(\mathbf{x_{\Lambda}})\leq r ,~\forall\xi\in\Xi$ 
as $\max\limits_{\xi\in\Xi}\left\{ h_{2,\xi}(\mathbf{x_{\Lambda}})\right\}\leq r$, and enrolling it in the objective, we get an equivalent form of \eqref{semiinfi} as
\begin{eqnarray}\label{innerdual}
\begin{aligned}
	\min\limits_{\mathbf{x_{\Lambda}}}\quad&\left\{\hat{\varphi}(\mathbf{x_{\Lambda}}):=h_{1}(\mathbf{x_{\Lambda}})+\max\limits_{\xi\in\Xi}\quad h_{2,\xi}(\mathbf{x_{\Lambda}})\right\}\\
	\rm{s.t.}\quad&  \mathbf{x_{\Lambda}}\in \hat{X}\times\mathbb{R}^{n+1}\times\mathbb{S}_+^{n+1}.
\end{aligned}
\end{eqnarray}
Employing the discretization scheme as in {Sect.} 3 of \cite{xu2018distributionally} by Xu et al., we have the discretized tractable reformulation for the DRO-moment model 
%with the ambiguity set \eqref{ambiguityset}  can be written as 
%	Similar to {Sect.} 3 of \cite{xu2018distributionally}, we consider the discrete approximation of \eqref{pelconti}. Our first step is to develop a discretized approximation of the continuous support set $\Theta$. Let The discretization scheme of \eqref{pelconti} 
as
\begin{eqnarray}\label{peldis}
\begin{aligned}
	\min\limits_{\mathbf{x_{\Lambda}}}\quad&\left\{\hat{\varphi}^q(\mathbf{x_{\Lambda}}):=h_{1}(\mathbf{x_{\Lambda}})+\max\limits_{i\in\mathbb{I}_{q}}\quad h_{2,\xi_i}(\mathbf{x_{\Lambda}})\right\}\\
	\rm{s.t.}\quad&  \mathbf{x_{\Lambda}}\in \hat{X}\times\mathbb{R}^{n+1}\times\mathbb{S}_+^{n+1},
\end{aligned}
\end{eqnarray}
where $\xi_i, i\in \mathbb{I}_{q}$ are i.i.d. samples of $\xi$ drawn by Monte Carlo sampling from the set $\Theta$.
%{and 
%\begin{eqnarray*}
%h_{2,\xi_i}(\mathbf{x_{\Lambda}})=\varphi(\hat{x},\xi_i)-(\Lambda^{T}\xi_i+\delta)^{T}\xi_i.
%\end{eqnarray*}}
%where $r\in\mathbb{R}$, the vector $\delta\in\mathbb{R}^{n+1}$, and $\Lambda\in\mathbb{R}^{(d+1)\times (d+1)}$ is a symmetric matrix.
In our model, both the function $\varphi(\cdot,\xi)$ and $\phi(\cdot)$ can be nonsmooth. Moreover, 
%since $\hat{\Sigma}$ is positive definite, we know that $\hat{\Sigma} = R^T R$ for a nonsingular upper triangular matrix $R$.  Thus
%$$\sqrt{\hat{t}_{1}(\delta+2\Lambda\hat{\varsigma})^{T}\hat{\Sigma}(\delta+2\Lambda\hat{\varsigma})} = \|\sqrt{\hat{t}_{1}} R (\delta + 2 \Lambda \hat{\varsigma}) \|.$$ It follows that  $\sqrt{\hat{t}_{1}(\delta+2\Lambda\hat{\varsigma})^{T}\hat{\Sigma}(\delta+2\Lambda\hat{\varsigma})}$
$\sqrt{\hat{t}_{1}}\left\|\hat{\Sigma}^{1/2} (\delta + 2 \Lambda \hat{\varsigma}) \right\|$ is also nonsmooth at the point $(\hat x,\delta,\Lambda)$ such that $\delta + 2 \Lambda \hat{\varsigma} =0$. This nonsmooth term is due to the delicate strategy to reduce the number of dual variables
as in Lemma 1 of \cite{delage2010distributionally}. 

The discretized tractable reformulation in (\ref{peldis}) is different to that in  \cite{xu2018distributionally}.  In \cite{xu2018distributionally}, the objective function of the DRO-moment problem is required to be smooth, and the strategy to decrease the number of the dual variables as in Lemma 1 of \cite{delage2010distributionally} is not employed. Consequently the discretized reformulation for the DRO-moment problem in \cite{xu2018distributionally} is a finite min-max problem and all the functions in the max operator are smooth. The cutting plane (CP) method is employed to address this problem \cite{xu2018distributionally} that is also a deterministic algorithm.

As follows, we consider a specific instance in risk management. Risk management in portfolio optimization determines an optimal weight for each asset, by solving specific optimization problems reflecting the risk attitudes of the manager on known data samples \cite{kremer2020sparse,martinez2021experimental}. {Let $\xi_{B}\in \mathbb{R}^{n}$ be the return rate vector of the $n$ assets,} and $\xi_{a}\in \mathbb{R}$ be the corresponding random market index return. We denote by $\xi = (\xi_B^T,\xi_a)^T \in \Theta\subseteq\mathbb{R}^{n+1}$. {Let $\hat{z} = (\hat{z}_1,\ldots,\hat{z}_n)^T \in \Delta_{n}$ be the tracking portfolio, with $\hat{z}_i$ being the investment weight in the $i$-th component stock.} We consider a distributionally robust index-tracking portfolio optimization problem with the CVaR penalty \cite{wang2023distributionally}:
\begin{eqnarray}\label{DROorip}
	\min\limits_{\hat{z}\in \Delta_{n}}\max\limits_{P\in\mathscr{P}}~ \mathbb{E}_{P}\left[ \left\|\xi_a-\hat{z}^{T}\xi_{B}\right\|^2\right] 
	+\tau_1 \|\hat{z}\|^2 + \tau_2 
	%\|x\|_1 + \tau_3 
	\phi_{\beta}(\hat{z}),
\end{eqnarray}
where the regularization parameters $\tau_{1},~\tau_{2}> 0$, 
and $\phi_{\beta}$ is the CVaR penalty defined as
%in \eqref{cvar}
\begin{eqnarray}\label{cvar}
	\phi_\beta(\hat{z})= \min\limits_{\alpha\in\mathbb{R}} \left\{\alpha+(1-\beta)^{-1}\mathbb{E}_{P}\left[-\hat{z}^{\top}\xi_{B}-\alpha\right]_{+}\right\}.
\end{eqnarray}
%Here $\ell(\hat{z},\xi_{B})$ is a convex loss function in $\hat{z}\in \Delta_{n}$. In this paper, the convex loss function
%\begin{eqnarray*}
%\ell(\hat{z},\xi_{B}):=-\hat{z}^{\top}\xi_{B}
%\end{eqnarray*}
%is the same as \cite{rockafellar2000optimization}. 
Using the similar arguments in Example 3 of \cite{delage2010distributionally}, we interchange $\max\limits_{P\in\mathscr{P}}$ and $\min\limits_{\alpha\in\mathbb{R}}$, then transfer \eqref{DROorip} equivalently in the form of  \eqref{nonDROorip} with $\hat{x}:=(\hat{z}^T,\alpha)^T$ and $\hat{X}=\Delta_{n}\times\mathbb{R}$, 
\begin{eqnarray*}
	\varphi(\hat{x},\xi)=\left\|\xi_a-\hat{z}^{T}\xi_{B}\right\|^2+\tfrac{\tau_{2}}{1-\beta}\left[-\hat{z}^{\top}\xi_{B}-\alpha\right]_{+},
	\quad\text{and}\quad
	\phi(\hat{x})=\tau_{1}\|\hat{z}\|^2+\tau_{2}\alpha.
\end{eqnarray*} 
%Using the similar arguments as \eqref{peldis},
The discretized tractable reformulation for \eqref{DROorip} is the same as \eqref{peldis}.

We employ the smoothing function \eqref{inf-conv-Emax} and obtain the smoothing problem of \eqref{peldis} as
\begin{equation}\label{smooDRO}
	\begin{aligned}
		%	\min\limits_{x,\Lambda_1,\Lambda_2,\Lambda_3} & \mu\ln q+\mu\ln\mathbb{E}_{\zeta}\left[ e^{\frac{\mathbf{\Phi}(x, \zeta)-\sum\limits_{i=1}^3 \left\langle \Lambda_i, \hat{\mathbf{\phi}}_i\left(\zeta^{\zeta}\right)\right\rangle}{\mu}}\right] \\
		\min\limits_{\mathbf{x_{\Lambda}}} \quad&\left\lbrace \tilde{\varphi}_{\mu}^q(\mathbf{x_{\Lambda}}):=~ \tilde{h}_{1,\mu}(\mathbf{x_{\Lambda}})+\mu\ln\left(\sum\limits_{i=1}^{q} e^{\frac{\tilde{h}_{2,\xi_{i},\mu}(\mathbf{x_{\Lambda}})}{\mu}}\right) \right\rbrace \\
		\text { s.t. } \quad & \mathbf{x_{\Lambda}}\in \hat{X}\times\mathbb{R}^{n+1}\times\mathbb{S}_+^{n+1},
	\end{aligned}
\end{equation}
where
\begin{eqnarray*}
	&&\tilde{h}_{1,\mu}(\mathbf{x_{\Lambda}})=\left\langle\hat{t}_{2}\hat{\Sigma},\Lambda\right\rangle+(\Lambda^{T}\hat{\varsigma}+\delta)^{T}\hat{\varsigma}+\sqrt{\hat{t}_{1}}\sqrt{(\delta+2\Lambda\hat{\varsigma})^{T}\hat{\Sigma}(\delta+2\Lambda\hat{\varsigma})+\mu}+\tau_{1}\|\hat{z}\|^2+\tau_{2}\alpha,\\
	&&	\tilde{h}_{2,\xi_{i},\mu}(\mathbf{x_{\Lambda}})=
	\left\|\xi_{a,i}-\hat{z}^{T}\xi_{B,i}\right\|^2 +\frac{\tau_{2}\mu}{1-\beta}\ln\left( 1+e^{\frac{-\hat{z}^{\top}\xi_{B,i}-\alpha}{\mu}}\right)
	-(\Lambda^{T}\xi_{i}+\delta)^{T}\xi_{i}.
\end{eqnarray*}

Here we claim that  
$\sqrt{\hat{t}_{1}(\delta+2\Lambda\hat{\varsigma})^{T}\hat{\Sigma}(\delta+2\Lambda\hat{\varsigma}) + \mu}$ is a smoothing function of the term $\left\|\sqrt{\hat{t}_{1}}\hat{\Sigma}^{1/2} (\delta + 2 \Lambda \hat{\varsigma}) \right\|$,
% $\sqrt{\hat{t}_{1}(\delta+2\Lambda\hat{\varsigma})^{T}\hat{\Sigma}(\delta+2\Lambda\hat{\varsigma})}$
%= \|\sqrt{\hat{t}_{1}} {\hat{\Sigma}}^{\frac{1}{2}} (\delta + 2 \Lambda \hat{\varsigma}) \|$$ 
{which is} essentially makes use 
of the inf-conv smoothing technique for $\|x\|$ that is mentioned in Remark 2.  We then get $\tilde{h}_{1,\mu}(\mathbf{x_{\Lambda}})$ is a smoothing function of $h_{1}(\mathbf{x_{\Lambda}})$ with parameters 
\begin{eqnarray*}\label{normpa}
	\left(\kappa_{r}, K_{r}, \mathcal{L}_{h,{r}}\right) = \left(\sqrt{\hat{t}_{1}}, 2\tau_{1}, \sqrt{\hat{t}_{1}}
	\left\|\hat{\Sigma}^{1/2}\right\|\left(1+2\left\|\hat{\varsigma} \right\| \right) \right).
\end{eqnarray*}
Meanwhile,  
%we know that $\frac{\tau_{2}\mu}{1-\beta}\ln\left( 1+e^{\frac{-\hat{z}^{\top}\xi_{B,i}-\alpha}{\mu}}\right)$ is a {smoothing function} of the term $\frac{\tau_{2}}{1-\beta}\left[-\hat{z}^{\top}\xi_{B,i}-\alpha \right]_+$. Then 
we know 
%{by Lemma \ref{lem3.21}} 
that $\tilde{h}_{2,\xi_{i},\mu}(\mathbf{x_{\Lambda}})$ is a smoothing function of $h_{2,\xi_{i}}(\mathbf{x_{\Lambda}})$ with parameters 
\begin{eqnarray*}
	\left(\kappa_{i}, K_{i}, \mathcal{L}_{h,i} \right)
	= \left( \tfrac{\tau_{2}\ln 2}{1-\beta}, 2\left\|\xi_{B,i}\right\|^2, \tfrac{\tau_{2}\left\|\xi_{B,i}\right\|^2}{1-\beta}
	\right).
\end{eqnarray*}

By simple computation, 
\begin{eqnarray*}
	& & \max\limits_{\mathbf{x_{\Lambda}}\in \hat{X}\times\mathbb{R}^{n+1}\times\mathbb{S}_+^{n+1},~\mu> 0}\left\{\left\|\nabla \tilde{h}_{2,\xi_i,\mu}(\mathbf{x_{\Lambda}}) \right\|\right\} 
	\\
	&\le& \max\limits_{\hat{z}\in \Delta_{n},~\mu>0}
	\left\lbrace
	\sqrt{
		\left[2\left(-\hat{z}^T ,1\right)\xi_{i}+\tfrac{\tau_{2}}{1-\beta}\right]^2\left\|\xi_{B,i} \right\|^2
		+
		\left( \tfrac{\tau_{2}}{1-\beta}\right)^2
		+\left\|\xi_{i} \right\|^2
		+\left\|\xi_{i}\right\|^4  
	}
	\right\rbrace\\
	&\leq&\sqrt{ \left[ \left( \tfrac{\tau_{2}}{1-\beta}\right)^2+
		8\left\|\xi_{i}\right\|^2
		+\frac{8\tau_{2}}{1-\beta}\left\|\xi_{i}\right\|\right] 
		\left\|\xi_{B,i} \right\|^2+\left( \tfrac{\tau_{2}}{1-\beta}\right)^2+\left\|\xi_{i} \right\|^2+\left\|\xi_{i}\right\|^4 
	} := M_i.
\end{eqnarray*}
Then combining Lemma \ref{lem3.21} and Remark 1 %\ref{splus}
yields that $\tilde{\varphi}_{\mu}^q(\mathbf{x_{\Lambda}})$ is a smoothing function of $\hat{\varphi}^q(\mathbf{x_{\Lambda}})$, with parameters $\left( \kappa, K, \hat{\mathcal{L}}_{h} \right)$, where 
\begin{eqnarray*}
	\kappa = \kappa_r+\ln q+\sqrt{\sum_{i=1}^q\kappa_i}, \quad K = K_r+\sum_{i=1}^{q}K_i, \quad \hat{\mathcal{L}}_{h} 
	=\mathcal{L}_{h,r} + 
	\left( \sum_{i=1}^q  M_i\right) ^2+ \sum_{i=1}^q \mathcal{L}_{h,i}.
\end{eqnarray*}
We use a historical daily return rate of $n=40,~80$ stocks between January 2005 and July 2023 from National Association of Securities Deal Automated Quotations (NASDAQ) index\footnotemark[2]\footnotetext[2]{https://cn.investing.com}, which contains $N_{\rm{tol}}=4675$ samples, that is, $q=4675$ in \eqref{peldis}. We denote the daily return rates of $n$ assets on the $i$-th day as $\xi_{B,i} = (\xi_{B_1,i}, \xi_{B_2,i}, \ldots, \xi_{B_n,i})^{T}$, where $\xi_{B_j,i}$ represents the natural logarithm of the closing price divided by the opening price of the $j$-th asset on the $i$-th day for $i=1,\ldots, N_{\rm{tol}}$.
%{All portfolios are re-balanced monthly, discarding the oldest and including the most recent observations.}

%Four methods are employed to solve \eqref{peldis}, and each method runs 20 times.
We list the mean of objective value (Obj) in the training set and the CPU time in Table \ref{tab1.1}. {Since Subgrad and SPG are deterministic methods, the objective values of 20 times are the same, which leads ``Obj Var"  to be 0.} It is clear that SSAG provides the computed solutions with the smallest ``Obj Mean" for $\epsilon=0.01,~0.001,~0.0001$ among all the methods. The ``Obj Var" is also very small. Moreover, we can see from the last column of Table \ref{tab1.1} that the proposed SSAG {method} can obtain the desired objective value within 20 seconds, and the CPU time does not grow dramatically as $d$ increases and $\epsilon$ decreases. In contrast, the CPU time of the Subgrad method grows quickly as $\epsilon$ decreases. {From Table \ref{tab1.1}, we can see that the SSAG method is around $5\sim 10$ times faster than the second best method - SPG}. 

\begin{table}[htp]
\caption{Obj and CPU time with parameters $(d,\epsilon,\overline{N}, \mu, m)$ }\label{tab1.1}
\begin{tabular}{ccccccccc}
	\toprule
	\multirow{2}{*}{$d$}&			\multirow{2}{*}{$\epsilon$}&\multirow{2}{*}{$\overline{N}$}&\multirow{2}{*}{ALG.} &\multirow{2}{*}{$\mu$}	&\multirow{2}{*}{$m$}&Obj &Obj &\multirow{2}{*}{CPU}  \\ 
	& & & & & &Mean &Var &	\\\midrule
	\multirow{15}{*}{$40$}&				\multirow{5}{*}{0.01}&\multirow{5}{*}{1.12e+09}
	&SSAG
	&$\downarrow$ &$\uparrow$ &\textbf{0.1152} &6.68e-10&\textbf{3.25} \\
	& & 	&RS
	&$\downarrow$ %&$\uparrow$
	&100 &\textbf{0.1152} &3.07e-09&39.01 \\
	&		& &Subgrad
	&- %&-
	&4675 &0.1153&\textbf{0}&143.26 \\
				&		& &SA
				&- %&-
				&$\uparrow$ &0.2245&1.28e-05&200.00 \\
	&		& &SPG
	&$\downarrow$ %&-
	&4675 &0.1195&\textbf{0} &7.45 \\	\cmidrule{2-9}
	&			\multirow{5}{*}{0.001}&\multirow{5}{*}{1.12e+11}
	&SSAG
	&$\downarrow$&$\uparrow$ &\textbf{0.1062}  &1.27e-10 &\textbf{4.21} \\
	& & 	&RS
	&$\downarrow$ %&$\uparrow$
	&100 &\textbf{0.1062} &3.35e-9&71.32 \\
	&			& &Subgrad
	&- %&-
	&4675 &0.1105&\textbf{0} &200.00 \\
				&			& &SA
				&- %&-
				&$\uparrow$ &0.2237&1.19e-05 &200.00 \\
	&			& &SPG
	&$\downarrow$ %&-
	&4675 &0.1072&\textbf{0} &9.80 \\	\cmidrule{2-9}
	&			\multirow{5}{*}{0.0001}&\multirow{5}{*}{1.12e+13}
	&SSAG
	&$\downarrow$&$\uparrow$&\textbf{0.1053}  &3.78e-11 &\textbf{4.75}\\
	& & 	&RS
	&$\downarrow$ %&$\uparrow$
	&100 &\textbf{0.1053} &1.11e-09 &73.68 \\
	&			& &Subgrad
	&- %&-
	&4675 &0.1104 &\textbf{0} &200.00 \\
				&			& &SA
				&- %&-
				&$\uparrow$ &0.2234 &1.40e-05 &200.00 \\
	&			& &SPG
	&$\downarrow$ %&-
	&4675 &0.1056&\textbf{0} &200.00 \\
	\midrule
	\multirow{15}{*}{$80$}&				\multirow{5}{*}{0.01}&\multirow{5}{*}{1.12e+09}
	&SSAG
	&$\downarrow$&$\uparrow$&\textbf{0.3570} &2.20e-09&\textbf{6.72} \\
	& & 	&RS
	&$\downarrow$ %&$\uparrow$
	&100 &\textbf{0.3570} &4.31e-08&72.14 \\
	&		& &Subgrad
	&- %&-
	&4675 &0.3790&\textbf{0}&200.00\\
				&		& &SA
				&- %&-
				&$\uparrow$ &0.5764&4.78e-05&200.00 \\
	&		& &SPG
	&$\downarrow$ %&-
	&4675 &0.3610&\textbf{0} &200.00 \\	\cmidrule{2-9}
	&			\multirow{5}{*}{0.001}&\multirow{5}{*}{1.12e+11}
	&SSAG
	&$\downarrow$ &$\uparrow$ &\textbf{0.3480}  &3.08e-09 &\textbf{7.33} \\
	& & 	&RS
	&$\downarrow$ %&$\uparrow$
	&100 &\textbf{0.3480} &1.36e-08&79.97 \\
	&			& &Subgrad
	&- %&-
	&4675 &0.3794&\textbf{0} &200.00 \\
				&			& &SA
				&- %&-
			&$\uparrow$ &0.5749&5.20e-05 &200.00 \\
	&			& &SPG
	&$\downarrow$ %&-
	&4675 &0.3610&\textbf{0} &200.00\\	\cmidrule{2-9}
	&			\multirow{5}{*}{0.0001}&\multirow{5}{*}{1.12e+13}
	&SSAG
	&$\downarrow$ &$\uparrow$&\textbf{0.3471} &2.23e-09 &\textbf{7.54} \\
	& & 	&RS
	&$\downarrow$ %&$\uparrow$
	&100 &\textbf{0.3471}  &5.91e-09&83.25 \\
	&			& &Subgrad
	&- %&-
	&4675 &0.3791&\textbf{0} &200.00 \\
				&			& &SA
				&- %&-
				&$\uparrow$ &0.5737 &3.41e-05 &200.00 \\
	&			& &SPG
	&$\downarrow$ %&-
	&4675 &0.3610&\textbf{0} &200.00 \\
	\bottomrule
\end{tabular}
\end{table}

%Fig. \ref{fig1.2} draws the boxplots of Obj in the training data for all methods in 20 runs, when $\epsilon=0.001$. Due to the similarity, we omit the figures of $\epsilon=0.01$ and $\epsilon=0.0001$. It is clear that our SSAG method performs the best.
%
%
%
%\begin{figure}[htbp]
%\subfigure[$d=40$]
%{
%	\includegraphics[height=4.5cm]{40box.eps}
%}
%\subfigure[$d=80$]
%{
%	\includegraphics[height=4.5cm]{80box.eps}
%}
%\caption{The boxplots of Obj in 20 runs, when $\epsilon=0.001$.}\label{fig1.2}
%\end{figure}

{At the end of Sect. \ref{subsec:DROportfolio}}, we mention that we also try 
the CP method in \cite{xu2018distributionally} to solve the smooth counterpart 
%by  the inf-conv smoothing technique  
with  a fixed smoothing parameter of $\mu=1\text{e-4}$, because in \cite{xu2018distributionally} it requires the smoothness of the functions within the {max} operator. The CP method is very slow to obtain a comparable objective value. Therefore, we do not present the results in Table \ref{tab1.1}.

\section{Conclusions}\label{sec:conclusions}

In this paper, we propose a stochastic smoothing accelerated gradient method for solving nonsmooth convex composite minimization problems. Various smoothing techniques can be employed to construct smoothing functions that satisfy Definition \ref{smoothingdefinition}  and  Assumption \ref{assumption 1}. 
%We do not need the requirement that the nonsmooth component is of linear max structure as in \cite{nesterov2005smooth,wang2022stochastic}. The requirement that the proximal operator of the nonsmooth component does not  need also, compared to \cite{bian2020smoothing}. 
As far as we know, it is the first time to propose an SA-type method to solve the constrained convex composite optimization problem whose nonsmooth term involves the maximization of finite but numerous nonsmooth convex functions. %The adaptive strategy for decreasing the smoothing parameter is also beneficial to faster computational speed. 
Moreover, the complexity results in terms of the number of iterations and the $\mathcal{SFO}$ match the best-known complexity bounds of the state-of-the-art first-order SA methods.
%Moreover, the complexity results can be dimension-independent, unlike the complexity results of the RS method that are dimension-dependent. 
The effectiveness and efficiency of our SSAG method have been  demonstrated by extensive numerical results.
In future, it is very interesting to develop smoothing  SA-type methods, either first-order or second-order methods, that address general nonsmooth nonconvex composite optimization, extending the excellent works  \cite{ghadimi2016mini,jalilzadeh2022variable}.

\backmatter

%\bmhead{Supplementary information}
%
%If your article has accompanying supplementary file/s please state so here. 
%
%Authors reporting data from electrophoretic gels and blots should supply the full unprocessed scans for key as part of their Supplementary information. This may be requested by the editorial team/s if it is missing.
%
%Please refer to Journal-level guidance for any specific requirements.

\bmhead{Acknowledgments}
We thank Prof. Yongchao Liu of Dalian University of Technology, for providing us the Matlab code of the CP method used in \cite{xu2018distributionally}. We are grateful to Prof. Uday V. Shanbhag of Pennsylvania State University for his constructively discussing with us about the sVS-APM method in \cite{jalilzadeh2022smoothed}.
% We express our gratitude to Prof. Yongchao Liu from Dalian University of Technology for his invaluable discussions regarding the CP method as presented in \cite{xu2018distributionally}.

\section*{Declarations}

%Some journals require declarations to be submitted in a standardised format. Please check the Instructions for Authors of the journal to which you are submitting to see if you need to complete this section. If yes, your manuscript must contain the following sections under the heading `Declarations':

\begin{itemize}
\item National Natural Science Foundation of China (No.12171027)
\item The authors declare no conflict of interest.
\item The data used to support this study are included within the article.
\item Conceptualization: Ruyu Wang, Chao Zhang; Methodology: Ruyu Wang, Chao Zhang; Writing - original draft preparation: Ruyu Wang, Chao Zhang; Writing - review and editing: Ruyu Wang, Chao Zhang; Funding acquisition: Chao Zhang; Supervision: Chao Zhang.
\end{itemize}

%\noindent
%If any of the sections are not relevant to your manuscript, please include the heading and write `Not applicable' for that section. 
%
%%%===================================================%%
%%% For presentation purpose, we have included        %%
%%% \bigskip command. please ignore this.             %%
%%%===================================================%%
%\bigskip
%\begin{flushleft}%
%Editorial Policies for:
%
%\bigskip\noindent
%Springer journals and proceedings: \url{https://www.springer.com/gp/editorial-policies}
%
%\bigskip\noindent
%Nature Portfolio journals: \url{https://www.nature.com/nature-research/editorial-policies}
%
%\bigskip\noindent
%\textit{Scientific Reports}: \url{https://www.nature.com/srep/journal-policies/editorial-policies}
%
%\bigskip\noindent
%BMC journals: \url{https://www.biomedcentral.com/getpublished/editorial-policies}
%\end{flushleft}

\begin{appendices}

\section{Proof of the smoothing functions}\label{appendixa}
In this section, we provide the proofs to show various smoothing approximations satisfy Definition \ref{smoothingdefinition} and Assumption \ref{assumption 1}.
% \subsection{Statements of smoothing lemmas}
\subsection{Nesterov's smoothing}

\noindent\textbf{Proof of Lemma \ref{lemma2.3}}

(i) According to (2.7) of \cite{nesterov2005smooth}, we know that for a.e. $\xi \in \Xi$ and any $\mu\in(0,\bar{\mu}]$,
$$ \tilde{\mathbf{H}}_{\mu}(x,\xi) \le \mathbf{H}(x,\xi) \le \tilde{\mathbf{H}}_{\mu}(x,\xi) + \kappa \mu\quad \mbox{with}\ \kappa = \max_{u\in U}\left\{d(u)\right\}.$$
Taking the expectation of the above inequalities with respect to $\xi$ yields 
$$\tilde{h}_{\mu}(x)\leq h(x) \leq \tilde{h}_{\mu}(x)+ \kappa \mu.$$
Hence Definition \ref{smoothingdefinition} (a) holds. By using Theorem 1 of \cite{nesterov2005smooth} on $\tilde{\mathbf{H}}_{\mu}(\cdot,\xi)$ and taking the expectation, we know that the convex function $\tilde{h}_{\mu}$ is $L_{\tilde{h}_{\mu}}$-smooth with $L_{\tilde{h}_{\mu}}=\frac{\left\|\mathbb{E}_{\xi}\left[ A_{\xi}\right]\right\|^{2}}{\sigma_{d} \mu}$. Thus, $\tilde{h}_{\mu}$ satisfies Definition \ref{smoothingdefinition} (b) and (d) with $K=0$ and ${\cal L}_{h}=\frac{\left\|\mathbb{E}_{\xi}\left[ A_{\xi}\right]\right\|^2}{\sigma_{d}}$.

Definition \ref{smoothingdefinition} (c) holds, because  for any $\mu_1$, $\mu_2 \in(0,\bar{\mu}]$,
%%we assume without loss of generality that $\mu_1 %%\ge \mu_2>0$. Consequently $\tilde{h}_{\mu_1}(x) %%\le \tilde{h}_{\mu_2}(x).$
%By the definition of $\tilde{\mathbf{H}}_{\mu}(\cdot,\xi)$ in \eqref{Nessmooth}, using expectation with respect to $\xi\in \Xi$, we have
% and
\begin{eqnarray*}
	& & \left| \tilde{h}_{\mu_2}(x)-\tilde{h}_{\mu_1}(x)\right|
	\\
	&&~=\left|	\mathbb{E}_{\xi}\left[ \max_{u\in U}\left\lbrace \left\langle A_{\xi}x, u\right\rangle - Q_{\xi}(u)-\mu_{2} d(u)\right\rbrace -\max_{u\in U}\left\lbrace \left\langle A_{\xi}x, u\right\rangle - Q_{\xi}(u)-\mu_{1} d(u)\right\rbrace \right]\right| \\
	&&~\leq	\left|\mathbb{E}_{\xi}\left[ \max_{u\in U}\left\lbrace \left\langle A_{\xi}x, u\right\rangle - Q_{\xi}(u)-\mu_{2} d(u)-\left(\left\langle A_{\xi}x, u\right\rangle - Q_{\xi}(u)-\mu_{1} d(u)\right)\right\rbrace \right]\right| \\
	&&~\leq \left\lvert \max_{u\in U}\left\lbrace (\mu_{1}-\mu_{2}) d(u)\right\rbrace \right\rvert= \kappa\left|\mu_{1}-\mu_{2}\right|,
\end{eqnarray*}
where the first inequality holds because for any continuous functions $t_1,~ t_2 : U\to \mathbb{R}$,
\begin{eqnarray*}
	\max_{u\in U}\left\{t_1(u)-t_2(u)\right\}+\max_{u\in U}\left\{t_2(u)\right\}
	\geq\max_{u\in U}\left\{t_1(u)-t_2(u)+t_2(u)\right\}
	= \max_{u\in U}\left\{t_1(u)\right\}.
\end{eqnarray*}
Till now, we have shown that statement (i) holds.

(ii) By \eqref{Nesterov-grad} and the fact that $U$ is bounded as required in \eqref{h}, we have
\begin{eqnarray*}
	\mathbb{E}_{\xi}\left[\left\|\nabla \tilde{\mathbf{H}}_{\mu}(x,\xi)\right\|^2\right] = \mathbb{E}_{\xi}\left[\left\|A_{\xi}^{T} \hat u_{\mu}(x,\xi)\right\|^2\right]
	\le \mathbb{E}_{\xi}\left[\left\|A_{\xi}^{T}\right\|^2 \left\|{\hat u}_{\mu}(x,\xi)\right\|^2\right] \le c_1^2 \max_{u\in U} \left\{\left\|u\right\|^2\right\}.
\end{eqnarray*}
Then the statement (ii) holds according to \eqref{As-b-12}.
\qed

\subsection{Randomized smoothing}

\noindent \textbf{Proof of Lemma \ref{lemma2.2}}

(i) By Lemma 7 of \cite{yousefian2012stochastic}, we have $\tilde{h}_{\mu}$ is convex and 
\begin{equation*}
	h(x) \leq \tilde{h}_{\mu}(x) \leq h(x)+L_{0} \mu. 
\end{equation*}
Thus Definition \ref{smoothingdefinition} (a) and (b) hold. By Lemma E.2 (iii) of \cite{duchi2012randomized}, %and Lemma 8 of \cite{yousefian2012stochastic}, 
the smoothing function $\tilde{h}_{\mu}$ in \eqref{RSsmooth} is $\frac{L_{0} \sqrt{d}}{\mu}$-smooth. Then, Definition \ref{smoothingdefinition} (d) holds with $K=0$ and ${\cal L}_{h}=L_{0} \sqrt{d}$.

%By the definition of $\tilde{h}_{\mu}$ in \eqref{RSsmooth} for $h$ in \eqref{orip-2}, we have
%	\begin{eqnarray*}
	%		\left\lvert \tilde{h}_{\mu_2}(x)-\tilde{h}_{\mu_1}(x)\right\rvert
	%		&=&	\left\lvert\int_{\mathbb{R}^{d}} \left[ h(x+\mu_2 v)-h(x+\mu_1 v)\right] \rho(v) d v\right\rvert \\
	%		&=&	\left\lvert\int_{\mathbb{R}^{d}} \int_{\Xi}\left[ \mathbf{H}(x+\mu_2 v,\xi)-\mathbf{H}(x+\mu_1 v,\xi) \right] dP(\xi) \rho(v) d v\right\rvert \\
	%		&=&\left\lvert\int_{\mathbb{R}^{d}} \int_{\Xi}\left\langle \eta,(\mu_2-\mu_1)v\right\rangle dP(\xi) \rho(v) d v\right\rvert \\
	%		&\leq&\frac{\mathbb{E}_{v,\xi} \left[\left\| \eta\right\|^2+\left\| v\right\|^2\right]}{2}\left\lvert \mu_2-\mu_1\right\rvert \leq \kappa\left\lvert\mu_{2}-\mu_{1}\right\rvert~~\text{with}~~\kappa=\frac{L_{0}^2+\mu^2}{2}.
	%	\end{eqnarray*}
%	Here for any given $\bar \mu\in (0,1)$ and any $\mu_1, \mu_2 \in (0,\bar \mu)$, the last equality is obtained by the Mean-Value Theorem (see, e.g. \cite[Theorem 2.3.7]{clarke1990optimization}) where $\eta\in\partial \mathbf{H}(x+\mu v,\xi)$ and $x+\mu v \in \operatorname{int}\left(X+B(0,\mu)\right)$ for some $\mu$ in the interior of a line segment connecting $\mu_1$ and $\mu_2$. The first inequality follows from Jensen's inequality that $\left|\mathbb{E}\left[ Z\right]\right|\leq\mathbb{E}\left[ \left|Z\right|\right]$ and Cauchy-Schwarz inequality. 

By the definition of $\tilde{h}_{\mu}$ in \eqref{RSsmooth}, we have
\begin{eqnarray*}
	\left\lvert \tilde{h}_{\mu_2}(x)-\tilde{h}_{\mu_1}(x)\right\rvert
	&=&\left\lvert \mathbb{E}_{v,\xi}\left[ \mathbf{H}(x+\mu_2 v,\xi)-\mathbf{H}(x+\mu_1 v,\xi)\right]\right\rvert\\
	&=&\left\lvert \mathbb{E}_{v,\xi}\left[ \left\langle \eta_\xi,(\mu_2-\mu_1)v\right\rangle\right]\right\rvert\\
	%&=&	\left\lvert\int_{\mathbb{R}^{d}} \left[ h(x+\mu_2 v)-h(x+\mu_1 v)\right] \rho(v) d v\right\rvert \\
	%&=&\left\lvert\int_{\mathbb{R}^{d}} \left\langle \eta,(\mu_2-\mu_1)v\right\rangle \rho(v) d v\right\rvert \\
	&\leq& \tfrac{\mathbb{E}_{\xi} \left[\left\| \eta_\xi\right\|^2
		\right]+\mathbb{E}_{v} \left[\left\| v\right\|^2\right]}{2} \left\lvert \mu_2-\mu_1\right\rvert \leq \kappa\left\lvert\mu_{2}-\mu_{1}\right\rvert~\text{with}~\kappa=\tfrac{L_{0}^2+1}{2}.
\end{eqnarray*}
Here for any $\mu_1,~ \mu_2 \in(0,\bar{\mu}]$, the second equality is obtained by the mean value theorem (see, e.g. \cite[Theorem 3.20]{mordukhovich2023easy}), where $\eta_\xi\in\partial \mathbf{H}(x+\mu v,\xi)$ and $x+\mu v \in \operatorname{int}\left(X+B(0,\mu)\right)$ for some $\mu$ in the interior of a line segment connecting $\mu_1$ and $\mu_2$. The first inequality follows from Jensen's inequality that $\left|\mathbb{E}\left[ Z\right]\right|\leq\mathbb{E}\left[ \left|Z\right|\right]$ and Cauchy-Schwarz inequality. Hence $\tilde{h}_{\mu}$ satisfies Definition \ref{smoothingdefinition} (c). Statement (i) holds as desired.

(ii) Assumption \ref{assumption 1} holds by Lemma E.2 (iv) of \cite{duchi2012randomized} with $\sigma^2=L_{0}^2$.
% for \eqref{orip}-\eqref{orip-2} by noting \eqref{As-b-12} and
% \begin{eqnarray*}
	% \mathbb{E}_{\xi}\left[\left\|\nabla \tilde{\mathbf{H}}_{\mu}(x,\xi)\right\|^2\right]
	% &=&\mathbb{E}_{\xi}\left[\left\|\nabla \mathbb{E}_{v}[\mathbf{H}(x+\mu v,\xi)]\right\|^2\right]\\
	% &=&\mathbb{E}_{\xi}\left[\left\| \mathbb{E}_{v}[G(x+ \mu v,\xi)]\right\|^2\right]\\ 
	% &\leq& \mathbb{E}_{v,\xi}\left[\left\|G(x+\mu v,\xi)\right\|^2\right] \le L_0^2,
	% \end{eqnarray*}
% where $G(x+ \mu v,\xi) \in \partial \mathbf{H}(x+\mu v, \xi)$, the first equality follows from \eqref{RSsmooth}, the second equality is obtained by Theorem 7.47 of \cite{shapiro2021lectures}. The first inequality follows from Jensen's inequality.
% that $\left|\mathbb{E}\left[ Z\right]\right|\leq\mathbb{E}\left[ \left|Z\right|\right]$. 
%The proof for the problem \eqref{orip} with the nonsmooth function $h$ defined in \eqref{Emax} is the same as the aforementioned proof.
%For \eqref{orip}-\eqref{max}, within the context of RS, we treat $\max\limits_{\xi\in\Xi}\{h_{\xi}(x)\}$ as a nested expression of two nonsmooth functions, without considering $\xi$ as a random variable. Therefore, there is no need to establish Assumption \ref{assumption 1} for the RS technique.
\qed

\vskip 2mm

\noindent \textbf{Proof of Lemma \ref{lemma2.21}}

(i) By Lemma E.3 of \cite{duchi2012randomized}, Definition \ref{smoothingdefinition} (a), (b) and (d) hold with $K=0$ and ${\cal L}_{h}=L_{0}$. By the definition of $\tilde{h}_{\mu}$ in \eqref{RSsmooth}, we have
\begin{eqnarray*}
	\left\lvert \tilde{h}_{\mu_2}(x)-\tilde{h}_{\mu_1}(x)\right\rvert
	&=&	\left\lvert \mathbb{E}_{v,\xi}\left[ \mathbf{H}(x+\mu_2 v,\xi)-\mathbf{H}(x+\mu_1 v,\xi)\right]\right\rvert\\
	&\leq&\mathbb{E}_{v,\xi}\left[ \left\lvert\mathbf{H}(x+\mu_2 v,\xi)-\mathbf{H}(x+\mu_1 v,\xi)\right\rvert\right]\\
	&\leq&	L_{0}\left\lvert \mu_2-\mu_1\right\rvert \mathbb{E}_v\left[\left\|v\right\| \right] \\
	&\leq&L_{0}\left\lvert \mu_2-\mu_1\right\rvert \sqrt{\mathbb{E}_{v} \left[\left\| v \right\|^2\right]}\\
	&=&L_{0}\sqrt{d}\left\lvert \mu_2-\mu_1\right\rvert,
\end{eqnarray*}
where the first and the last inequalities follow from Jensen's inequality, 
% that $\left|\mathbb{E}\left[ Z\right]\right|\leq\mathbb{E}\left[ \left|Z\right|\right]$,
and the second inequality follows from the $L_{0}$-Lipschitz condition in (\ref{L0Gaussian}). 
%	$\mathbf{H}(x+\mu v,\xi)$ is 
%	$L_{0}$-Lipschitz for a.e. $\xi\in\Xi$
By $v\sim \mathcal{N}\left(0, I_{d}\right)$ and Example 5.21 of \cite{shapiro2021lectures}, 
%(1.44) of \cite{shapiro2021lectures}, 
we have that $\left\|v\right\|^2$ follows the chi-square distribution with mean $d$. So far we have shown that $\tilde{h}_{\mu}$ satisfies property (c) of Definition \ref{smoothingdefinition} with $\kappa=L_{0}\sqrt{d}$. Thus statement (i) holds.

(ii) Assumption \ref{assumption 1} holds by Lemma E.3 (iv) of \cite{duchi2012randomized} with $\sigma^2=L_{0}^2$.
% Assumption \ref{assumption 1} holds by \eqref{As-b-12} and
% $$\mathbb{E}_{\xi}\left[\left\|\nabla \tilde{\mathbf{H}}_{\mu}(x,\xi)\right\|^2\right] = \mathbb{E}_{v,\xi}\left[\left\|G(x+\mu v,\xi)\right\|^2\right] \le L_0^2,$$
% that is, Lemma E.3 (iv) of \cite{duchi2012randomized}. Here, $G(x+ \mu v,\xi) \in \partial \mathbf{H}(x+\mu v, \xi)$. %The proof for the problem \eqref{orip} with the nonsmooth function $h$ defined in \eqref{Emax} is the same as the aforementioned proof. %Similar to the proof of Lemma \ref{lemma2.2}, for the RS technique in \eqref{orip}-\eqref{max}, there is no need to establish Assumption \ref{assumption 1}.	
\qed

\subsection{Inf-conv smoothing}
\textbf{Proof of Lemma \ref{lemma2.1}}

%For any $\bar \mu>0$ and any $\mu>0$, 
(i) By Lemma 4.2 of \cite{beck2012smoothing},
\eqref{inf-conv-reform}, and the condition that
$\omega^*(y)\le 0$ for all $y\in \operatorname{dom} \omega^*$, we know that
for every $\mu\in(0,\bar{\mu}]$ and $\xi\in \Xi$, %$\tilde{\mathbf{H}}_{\mu}$ is convex, continuous, and
%there exist constants $\tau_1>0$ and $\tau_{2}>0$ such that
$$\mathbf{H}(x,\xi) \leq \tilde{\mathbf{H}}_{\mu}(x,\xi) \leq \mathbf{H}(x,\xi)+\omega(0) \mu.$$
By taking expectation on the above inequalities with respect to $\xi$, we know that $\tilde{h}_{\mu}$ satisfies Definition \ref{smoothingdefinition} (a). Moreover, by using Theorem 4.1 of \cite{beck2012smoothing} for $\tilde{\mathbf{H}}_{\mu}$ and taking expectation on $\xi$, we know that $\tilde{h}_{\mu}$ is convex, finite-valued, differentiable, and $L_{\tilde{h}_{\mu}}$-smooth with constant $L_{\tilde{h}_{\mu}}=\frac{1}{\sigma_{\omega} \mu}$. Thus, $\tilde{h}_{\mu}$ satisfies Definition \ref{smoothingdefinition} (b) and (d) with $K=0$ and ${\cal L}_{h}=\frac{1}{\sigma_{\omega}}$.

For any $\mu_1,~ \mu_2 \in(0,\bar{\mu}]$, we assume without loss of generality that $0<\mu_2 \le \mu_1$. Consequently $ \tilde{h}_{\mu_2}(x) \le \tilde{h}_{\mu_1}(x)$, according to \eqref{inf-conv-reform} and the assumption that $\omega^*(y)\le 0$ for all $y\in \operatorname{dom} \omega^*$. By using the definition of $\tilde{h}_{\mu}$ and the similar arguments as for the Nesterov's smoothing approximation in Lemma \ref{lemma2.3}, we have
\begin{eqnarray*}
	0 &\le&	\tilde{h}_{\mu_1}(x)-\tilde{h}_{\mu_2}(x)
	\le \mathbb{E}_{\xi}\left[\max_{y\in \mathbb{R}^{d}} \left\{-\mu_1 \omega^*(y) + \mu_2 \omega^*(y)\right\}\right]\\
	&=& (\mu_1 - \mu_2) \max_{y\in \mathbb{R}^{d}} \left\{-\omega^*(y)\right\} = (\mu_1 - \mu_2) \omega(0).
\end{eqnarray*}
% where the last equality holds because
%$
%\omega(0) = \omega^{**}(0) = \max_{y\in \mathbb{R}^{d}}\{\langle 0, y\rangle - \omega^*(y)\}.
%$
Hence $\tilde{h}_{\mu}$ satisfies Definition \ref{smoothingdefinition} (c) with $\kappa=\omega(0)$. Statement (i) holds as desired.

(ii) By the definition of convex conjugate in \eqref{conjugate}, 
%Theorem 4.8 of \cite{beck2017first}, 
%($\omega*(y)=\omega^{**}(0)$), and $\omega^*(y)\geq 0$ for all $y\in \operatorname{dom} \omega^*$
we have 
\begin{eqnarray*}			
	\omega^*(y)=\sup_{x\in \mathbb{R}^{d}} \left\{\langle x, y\rangle - \omega(x)\ :\ x\in \operatorname{dom} \omega\right\} \ge -\omega(0) 
	%=\sup_{x\in \mathbb{R}^{d}} \{ - \omega^*(x)\}
	\ge 0.
\end{eqnarray*}
Similar to the proof of statement (i), we obtain
%							By Lemma 4.4 of \cite{burke2013epi},
%							\eqref{inf-conv-reform}, and the assumption that
%						$\omega^*(y)\ge 0$ for all $y\in \operatorname{dom} \omega^*$, we know 
%						that for every $\mu>0$ and $\xi\in \Xi$,
%				there exist constants $\tau_1>0$ and $\tau_{2}>0$ such that
%					 $$\mathbf{H}(x,\xi)- \mu\omega^*(\gamma_x^\xi) \leq \tilde{\mathbf{H}}_{\mu}(x,\xi) \leq \mathbf{H}(x,\xi),$$
%					 where $\gamma_x^\xi\in \partial\mathbf{H}(x,\xi)$ is a subgradient of $\mathbf{H}$ at $x$. By taking expectation on the above inequalities with respect to $\xi$ and using Theorem 4.1 of \cite{beck2012smoothing} for $\tilde{\mathbf{H}}_{\mu}$, we know 
that $\tilde{h}_{\mu}$ satisfies Definition \ref{smoothingdefinition} (a), (b), and (d) with $K=0$ and ${\cal L}_{h}=\frac{1}{\sigma_{\omega}}$.

For any $\mu_1,~\mu_2 \in(0,\bar{\mu}]$, we assume without loss of generality that $0<\mu_2 \le \mu_1$. Consequently $\tilde{h}_{\mu_2}(x) \geq \tilde{h}_{\mu_1}(x)$, according to \eqref{inf-conv-reform} and the condition that $\omega^*(y)\geq 0$ for all $y\in \operatorname{dom} \omega^*$. Since $\omega(\cdot)$ is level bounded, {${\bf{H}}(\cdot,\xi)$ is lower semicontinuous, and $\inf_{x\in \mathbb{R}^d} {\bf H}(x,\xi)>-\infty$} for a.e. $\xi\in\Xi$, we know that for any $\mu\in(0,\bar{\mu}]$ and $x\in X$, 
$$\min_{u\in \mathbb{R}^d} \{ {\bf{H}}(u,\xi) + \mu_2 \omega((x-u)/\mu_2) \}$$
%\eqref{inf-conv-H} 
has finite objective value and the infimum can be obtained in a compact set $S_{x}$. Then we have
\begin{eqnarray*}
	0 &\leq&\tilde{h}_{\mu_2}(x)-\tilde{h}_{\mu_1}(x)\\
	%	&= &\mathbb{E}_{\xi}\left[\min_{u\in \mathbb{R}^{d}} \left\{
	%	\mathbf{H}(u,\xi)+\mu_2 \omega\left( (x-u)/\mu_{2}\right)  
	%	-\tilde{\mathbf{H}}_{\mu_1}(x,\xi)\right\}\right]\\
	&= &\mathbb{E}_{\xi}\left[\min_{u\in S_{x} } \left\{
	\mathbf{H}(u,\xi)+\mu_2 \omega\left( 
	(x-u)/{\mu_{2}}\right)  
	-\tilde{\mathbf{H}}_{\mu_1}(x,\xi)\right\}\right]\\
	&=& \mathbb{E}_{\xi}\left[\min_{u\in S_{x}} \left\{
	\mathbf{H}(u,\xi)+\mu_2 \omega\left( (x-u)/\mu_{2}\right)  
	-\min_{y\in \mathbb{R}^{d}} \left\{
	\mathbf{H}(y,\xi)+\mu_1 \omega\left( (x-y)/\mu_{1}\right) \right\}\right\}\right]\\
	&=& \mathbb{E}_{\xi}\left[\min_{u\in S_{x}} \max_{y\in \mathbb{R}^{d}}\left\{\mathbf{H}(u,\xi)+\mu_2 \omega\left( (x-u)/\mu_{2}\right)  
	- \mathbf{H}(y,\xi)-\mu_1 \omega\left( (x-y)/\mu_{1}\right) \right\}\right].
\end{eqnarray*}
Let $\hat u = (1-\tfrac{\mu_2}{\mu_1}) x + \tfrac{\mu_2}{\mu_1} y$. It is clear that $\frac{x-\hat u}{\mu_2} = \frac{x-y}{\mu_1}$  and 
\begin{eqnarray*}
	\min_{u\in \mathbb{R}^d} \left\{ \mathbf{H}(u,\xi) + \mu_2 \omega\left(\frac{x-u}{\mu_2}\right) \right\} \le \mathbf{H}(\hat u,\xi) + \mu_2 \omega \left(\frac{x-\hat u}{\mu_2} \right).
\end{eqnarray*}
This, together with the fact that ${\bf{H}}(\cdot,\xi)$ is lower semicontinuous and the Minimax Theorem (Theorem 4.2' of \cite{sion1958general}), yields
\begin{eqnarray*}
	0 &\leq& \tilde{h}_{\mu_2}(x)-\tilde{h}_{\mu_1}(x)\\
	&=& \mathbb{E}_{\xi}\left[ \max_{y\in \mathbb{R}^{d}}\min_{u\in S_{x}}\left\{
	\mathbf{H}(u,\xi)+\mu_2 \omega\left( (x-u)/\mu_{2}\right)  
	- \mathbf{H}(y,\xi)-\mu_1 \omega\left( (x-y)/\mu_{1}\right) \right\}\right]\\
	&=& \mathbb{E}_{\xi}\left[ \max_{y\in \mathbb{R}^{d}}\min_{u\in \mathbb{R}^{d}}\left\{
	\mathbf{H}(u,\xi)+\mu_2 \omega\left( (x-u)/\mu_{2}\right)  
	- \mathbf{H}(y,\xi)-\mu_1 \omega\left( (x-y)/\mu_{1}\right) \right\}\right]\\
	&\leq& \mathbb{E}_{\xi}\left[ \max_{y\in \mathbb{R}^{d}}\left\{
	\mathbf{H}\left( \hat u,\xi\right) 
	- \mathbf{H}(y,\xi)+\left( \mu_{2}-\mu_1\right)  \omega\left( (x-y)/\mu_{1}\right) \right\}\right].
\end{eqnarray*}
Then, by using the convexity of $\bf{H}(\cdot,\xi)$ and the expression of $\hat u$, we find
\begin{eqnarray*}
	0 &\leq& \tilde{h}_{\mu_2}(x)-\tilde{h}_{\mu_1}(x)\\
	&\leq& \mathbb{E}_{\xi}\left[ \max_{y\in \mathbb{R}^{d}}\left\{
	\left( 1-\mu_{2}/\mu_{1}\right)\left[ \mathbf{H}\left( x,\xi\right) - \mathbf{H}(y,\xi)\right] -\left( \mu_{1}-\mu_2\right)  \omega\left( (x-y)/\mu_{1}\right) \right\}\right]\\
	&\leq&  \mathbb{E}_{\xi}\left[\max_{y\in \mathbb{R}^{d}}\left\{
	\left( 1-\mu_{2}/\mu_{1}\right)\left\langle \gamma_x^\xi,x-y\right\rangle  -\left( \mu_{1}-\mu_2\right)  \omega\left( (x-y)/\mu_{1}\right) \right\}\right] \\
	&=& \left( \mu_{1}-\mu_2\right) \mathbb{E}_{\xi}\left[\max_{y\in \mathbb{R}^{d}}\left\{\left\langle \gamma_x^\xi,(x-y)/\mu_{1}\right\rangle  -  \omega\left( (x-y)/\mu_{1}\right) \right\}\right] \\
	&=& \left( \mu_{1}-\mu_2\right)\mathbb{E}_{\xi}\left[ \max_{z\in \mathbb{R}^{d}}\left\{\left\langle \gamma_x^\xi,z\right\rangle  -  \omega\left( z\right) \right\}\right] \\
	&=&	(\mu_1 - \mu_2)\mathbb{E}_{\xi}\left[ \omega^*(\gamma_x^\xi)\right] .
\end{eqnarray*}
Here,
%the second inequality follows from $u=\left( 1-\mu_{2}/\mu_{1}\right) x+\left(\mu_{2}/\mu_{1}\right)y$, the last second and last inequalities are based on the convexity of $\mathbf{H}$, and
$\gamma_x^\xi\in\partial\mathbf{H}\left( x,\xi\right)$ is an arbitrary subgradient of $\mathbf{H}(\cdot,\xi)$ at $x$. Hence $\tilde{h}_{\mu}$ satisfies Definition \ref{smoothingdefinition} (c) with $\kappa=D[{\bf{H}},{\omega^*}]$. Statement (ii) holds as desired.

(iii) By \eqref{grad-inf-conv}, it is easy to obtain that
\begin{eqnarray*}
	\mathbb{E}_{\xi}\left[\left\|\nabla \tilde{\mathbf{H}}_{\mu}(x,\xi)\right\|^2\right] = \mathbb{E}_{\xi}\left[\left\|\nabla \omega\left(\frac{x-\hat{v}_{\mu}(x,\xi)}{\mu} \right)\right\|^2\right] \le \sigma^2.
\end{eqnarray*}
Hence by \eqref{As-b-12} statement (iii) holds.
%	\begin{remark}
	%		For $\omega$ defined in \eqref{b-omega}, its convex conjugate $\omega^*$ in \eqref{2.11} satisfies the assumption in Lemma \ref{lemma2.1} (i) that $\omega^*(y)\le 0$ for all $y\in \operatorname{dom} \omega^*$, and its gradient $\nabla \omega$ in \eqref{gradient-omega} satisfies the assumption in Lemma \ref{lemma2.1} (ii) that $\max \{\|\nabla \omega(x)\|\}\le \sqrt{q} < \infty$.
	%	\end{remark}
\qed

\vskip 2mm
\noindent \textbf{Proof of Lemma \ref{lem3.21}}

(i) Since $\tilde h_{i,\mu}$ is a smoothing function of $h_i(x)$ for each $i\in \mathbb{I}_q$, %{with the properties in statement (i)}, and 
by slightly modifying the proof of Proposition 4.1 \cite{beck2012smoothing}, we can easily deduce that the smoothing approximation $\tilde{h}_{\mu}$ satisfies the conditions in Definition \ref{smoothingdefinition} (a) and  (b).

Denote ${\tilde{z}}_{\mu}(x)=({\tilde h}_{1,\mu}(x), \ldots,{\tilde h}_{q,\mu}(x))^{T}$ and $\mathbf{J}_{\tilde{z}_{\mu}}(x)=(\nabla{\tilde h}_{1,\mu}(x), \ldots,\nabla{\tilde h}_{q,\mu}(x))$. By the definition of $\tilde b_{\mu}(z)$ in \eqref{bt34}, we have 
\begin{eqnarray*}
	\left|\tilde{h}_{\mu_1}(x) - \tilde{h}_{\mu_2}(x)\right|
	&=&  \left|\tilde b_{\mu_1}\left(\tilde z_{\mu_1}(x)\right) - \tilde b_{\mu_2} \left(\tilde z_{\mu_2}(x)\right)\right|\\
	%	&=& \left|\tilde b_{\mu_1}\left(\tilde z_{\mu_1} (x)\right) - \tilde b_{\mu_1} \left(\tilde z_{\mu_2}(x)\right) + \tilde b_{\mu_1} \left(\tilde z_{\mu_2}(x)\right) - \tilde b_{\mu_2}\left(\tilde z_{\mu_2}(x)\right)\right|\\
	&\leq& \left|\tilde b_{\mu_1}\left(\tilde z_{\mu_1} (x)\right) - \tilde b_{\mu_1} \left(\tilde z_{\mu_2}(x)\right)\right| +\left| \tilde b_{\mu_1} \left(\tilde z_{\mu_2}(x)\right) - \tilde b_{\mu_2}\left(\tilde z_{\mu_2}(x)\right)\right|\\
	&=& \left|\left\langle \nabla\tilde b_{\mu_1}\left(\hat z\right),\tilde z_{\mu_1}(x)-\tilde z_{\mu_2}(x)\right\rangle \right| +\left| \tilde b_{\mu_1} \left(\tilde z_{\mu_2}(x)\right) - \tilde b_{\mu_2}\left(\tilde z_{\mu_2}(x)\right)\right|\\
	&\leq& \left| \left\langle \frac{\left( e^{{\hat z}_1 / \mu},\ldots,e^{{\hat z}_q / \mu}\right)^T}{\sum_{i=1}^q e^{{\hat z}_i /\mu}},
	\tilde z_{\mu_1}(x)-\tilde z_{\mu_2}(x)\right\rangle \right| +\omega(0) \left|\mu_1 - \mu_2\right|\\
	&\leq& 
	\left\| \tilde z_{\mu_1}(x)-\tilde z_{\mu_2}(x)\right\| +\omega(0) \left|\mu_1 - \mu_2\right|\\
	% &\leq& \left\| \nabla\omega\left(\frac{\tilde z_{\mu_1}-u_{\mu_1}(\tilde z_{\mu_1})}{\mu}\right)\right\|
	% \left\| \kappa_{z}\right\|(\mu_1 - \mu_2) +\omega(0) (\mu_1 - \mu_2)\\
	&\leq& 
	%\left(\sqrt{\sum_{i=1}^q \kappa_i^2}+\omega(0)\right) 
	\kappa \left|\mu_1 - \mu_2\right|,
\end{eqnarray*}
where $\hat z$ is a point on the line segment between $\tilde z_{\mu_1}(x)$ and $\tilde z_{\mu_2}(x)$, the second equality is obtained by the mean value theorem, {and the second inequality follows from (\ref{bt34}), Lemma \ref{lemma2.1} (i) and Remark 2.} Hence Definition \ref{smoothingdefinition} (c) is satisfied with $\kappa = \sqrt{\sum_{i=1}^q \kappa_i^2}+\ln q$, by the definition of $\omega$ in \eqref{b-omega}. 

%To calculate $K$ and $\mathcal{L}_h$ in
{Next, we prove that Definition \ref{smoothingdefinition} (d) holds.} We have for $\forall x, y\in X$,
\begin{eqnarray*}
	&&\left\|\nabla\tilde{h}_{\mu}(x) - \nabla\tilde{h}_{\mu}(y)\right\|\\
	&&\quad=  \left\|\mathbf{J}_{\tilde{z}_{\mu}}(x)\nabla\tilde b_{\mu}(\tilde z_{\mu}(x)) - \mathbf{J}_{\tilde{z}_{\mu}}(y)\nabla\tilde b_{\mu} (\tilde z_{\mu}(y))\right\|\\
	&&\quad=  \left\|\mathbf{J}_{\tilde{z}_{\mu}}(x)\left[ \nabla\tilde b_{\mu}(\tilde z_{\mu}(x))- \nabla\tilde b_{\mu} (\tilde z_{\mu}(y))\right]
	+ \left[ \mathbf{J}_{\tilde{z}_{\mu}}(x)-\mathbf{J}_{\tilde{z}_{\mu}}(y)\right] \nabla\tilde b_{\mu} (\tilde z_{\mu}(y))\right\|\\
	&&\quad\leq  \left\| \nabla\tilde b_{\mu}(\tilde z_{\mu}(x))- \nabla\tilde b_{\mu} (\tilde z_{\mu}(y))\right\|
	\left\|\mathbf{J}_{\tilde{z}_{\mu}}(x)\right\|
	+ \left\|\nabla\tilde b_{\mu} (\tilde z_{\mu}(y))\right\|
	\left\|\mathbf{J}_{\tilde{z}_{\mu}}(x)-\mathbf{J}_{\tilde{z}_{\mu}}(y) \right\|\\
	&&\quad\leq  \frac{1}{\mu}\left\| \tilde z_{\mu}(x)-\tilde z_{\mu}(y)\right\|
	\left\|\mathbf{J}_{\tilde{z}_{\mu}}(x)\right\|
	+ \left\|\nabla\tilde b_{\mu} (\tilde z_{\mu}(y))\right\|
	\left\|\mathbf{J}_{\tilde{z}_{\mu}}(x)-\mathbf{J}_{\tilde{z}_{\mu}}(y) \right\|,
\end{eqnarray*}
where the last inequality follows from  {the fact that the Lipschitz constant of $\nabla\tilde b_{\mu}(\cdot)$ is $\frac{1}{\mu}$, according to Example 4.4 of \cite{beck2012smoothing}.} 

{Using} the definition of $M_i$, we can get that for any $x, y \in X$ and $\mu\in (0,\bar \mu]$,
%$\nabla{\tilde h}_{i,\mu}(x)\in\partial h_{i}(x)$ by Proposition 1 (3) of \cite{chen2012smoothing}, we get
\begin{eqnarray*}
	\left\| \tilde z_{\mu}(x)-\tilde z_{\mu}(y)\right\|
	&=&\sqrt{\sum_{i=1}^{q}\left|{\tilde h}_{i,\mu}(x)-{\tilde h}_{i,\mu}(y) \right|^2}
	\leq\sum_{i=1}^{q}\left|{\tilde h}_{i,\mu}(x)-{\tilde h}_{i,\mu}(y) \right|\\
	%&\leq&\sum_{i=1}^{q}\left\|\nabla{\tilde h}_{i,\mu}(x) \right\|\left\|x-y \right\|\\
	&\leq&\sum_{i=1}^{q}\max_{z\in X}\left\{\left\|\nabla{\tilde h}_{i,\mu}(z) \right\|\right\}\left\|x-y \right\|\\
	&\leq&\sum_{i=1}^{q}M_i\left\|x-y \right\|,
\end{eqnarray*}
and
\begin{eqnarray*}
	\left\|\mathbf{J}_{\tilde{z}_{\mu}}(x)\right\|
	=\sqrt{\sum_{i=1}^{q}\left\|\nabla{\tilde h}_{i,\mu}(x)\right\|^2}
	\leq\sum_{i=1}^{q}\left\|\nabla{\tilde h}_{i,\mu}(x)\right\|
	\leq\sum_{i=1}^{q}M_i.
\end{eqnarray*}
In view of \eqref{bt34}, we have 
\begin{eqnarray*}
	\left\|\nabla\tilde b_{\mu} (\tilde z_{\mu}(y))\right\|
	&=&\sqrt{\left[ \sum_{i=1}^q\left( e^{\frac{{\tilde h}_{i,\mu}(y)}{\mu}}\right)^2\right]  \left/ \left(\sum_{j=1}^q e^{\frac{{\tilde h}_{j,\mu}(y)}{\mu}}\right)^2\right.}
	\leq 1,
\end{eqnarray*}
and
\begin{eqnarray*}
	\left\|\mathbf{J}_{\tilde{z}_{\mu}}(x)-\mathbf{J}_{\tilde{z}_{\mu}}(y) \right\|
	&=&\sqrt{\sum_{i=1}^{q}\left\|\nabla{\tilde h}_{i,\mu}(x)-\nabla{\tilde h}_{i,\mu}(y) \right\|^2}
	%\\
	\leq \sum_{i=1}^{q}\left\|\nabla{\tilde h}_{i,\mu}(x)-\nabla{\tilde h}_{i,\mu}(y) \right\|\\
	&\leq&\sum_{i=1}^{q}\left( K_i + \frac{{\cal{L}}_{h_i}}{\mu}\right) \left\|x-y \right\|.
\end{eqnarray*}
By the above inequalities, we have 
\begin{eqnarray*}
	\left\|\nabla\tilde{h}_{\mu}(x) - \nabla\tilde{h}_{\mu}(y)\right\|
	\leq  \left( \frac{\left( \sum_{i=1}^{q}M_i\right)^2
		+\sum_{i=1}^{q}{\cal{L}}_{h_i}}{\mu}
	+
	\sum_{i=1}^{q}K_i\right) \left\|x-y \right\|.
\end{eqnarray*}
Hence Definition \ref{smoothingdefinition} (d) is satisfied with $K = \sum_{i=1}^{q}K_i$, and ${\cal L}_h = \left( \sum_{i=1}^{q}M_i\right)^2
+\sum_{i=1}^{q}{\cal{L}}_{h_i}$. Till now we have shown that statement (i) of this lemma holds.

(ii) Using the similar arguments as \eqref{As-b-12} and $\sigma^2 = \mathbb{E}_{i}\left[\left(M_i\right)^2\right]$, we have statement (ii) holds because
\begin{eqnarray}\label{infii}
	\mathbb{E}_{i}\left[\left\|\nabla \tilde{\mathbf{\Psi}}_{\mu}(x,i) - \nabla \tilde{\psi}_{\mu}(x)\right\|^2\right]
	%\nonumber\\
	%&&\quad	
	&=&\sum_{i=1}^q p_{x,\mu}(i)\left\|\nabla f(x)+\nabla \tilde{h}_{i,\mu}(x)-\nabla\tilde{\psi}_{\mu}(x)\right\|^2\nonumber\\
	&=&\sum_{i=1}^q p_{x,\mu}(i)\left\|\nabla \tilde{h}_{i,\mu}(x)-\nabla\tilde{h}_{\mu}(x)\right\|^2\nonumber\\
	%&&\quad =\sum_{i=1}^q p_{x,\mu}(i) \left\|\nabla \tilde{h}_{i,\mu}(x)\right\|^2+\left\|\nabla\tilde{h}_{\mu}(x)\right\|^2-2\left\langle \sum_{i=1}^q p_{x,\mu}(i)\nabla \tilde{h}_{i,\mu}(x),\nabla\tilde{h}_{\mu}(x)\right\rangle \nonumber\\
	%&&\quad	\leq\sum_{i=1}^q p_{x,\mu}(i)\left\|\nabla \tilde{h}_{i,\mu}(x)\right\|^2 -\left\|\nabla\tilde{h}_{\mu}(x)\right\|^2\nonumber\\
	%&&\quad	=\mathbb{E}_{i}\left[ \left\|\nabla \tilde{h}_{i,\mu}(x)\right\|^2\right]  -\left\|\nabla\tilde{h}_{\mu}(x)\right\|^2\nonumber\\
	&\leq& \mathbb{E}_{i}\left[ \left\|\nabla \tilde{h}_{i,\mu}(x)\right\|^2\right] 
	\leq\sigma^2.
\end{eqnarray}
%where the last inequality follows from $\nabla \tilde{h}_{i,\mu}(x)\in\partial h_{i}(x)$ by Proposition 1 (3) of \cite{chen2012smoothing}.
\qed

\noindent \textbf{Proof of Lemma \ref{lem3.211}}

(i) {Since $\tilde h_{i,\mu}(\cdot,\xi)$ is a smoothing function of $h_i(\cdot,\xi)$} for each $i\in \mathbb{I}_q$, a.e. $\xi\in\Xi$, and any $\mu\in(0,\bar{\mu}]$, by slightly modifying the {proofs} of Lemmas \ref{lemma2.1} and \ref{lem3.21}, we can easily deduce that the smoothing approximations $\tilde{h}_{\mu}$ satisfies  Definition \ref{smoothingdefinition} (a) and (b).
Denote
$${\tilde{z}}_{\mu}(x,\xi) = \left({\tilde h}_{1,\mu}(x,\xi), \ldots,{\tilde h}_{q,\mu}(x,\xi)\right)^{T}, \quad \mathbf{J}_{\tilde{z}_{\mu}}(x,\xi) = \left(\nabla{\tilde h}_{1,\mu}(x,\xi), \ldots,\nabla{\tilde h}_{q,\mu}(x,\xi)\right).$$ 
Using the similar arguments in Lemma \ref{lem3.21}, we can show that  
Definition \ref{smoothingdefinition} (c) and (d) are satisfied with the parameters $\kappa$, $K$, and $\mathcal{L}_h$, as defined in statement (i).

(ii) Using the similar arguments as \eqref{infii} as well as  $\sigma^2= \mathbb{E}_{\xi,i}\left[ {\left(M_{i,\xi}\right)^2}\right]$, we have statement (ii) holds.
%	 because
%	\begin{eqnarray*}
	%		&&	\mathbb{E}_{\xi,i}[\|\nabla \tilde{\mathbf{\Psi}}_{\mu}(x,\xi,i) - \nabla \tilde{\psi}_{\mu}(x)\|^2]\\
	%		&&\quad	=\mathbb{E}_{\xi}\left[ \sum_{i=1}^q p_{x,\mu}^{\xi}(i)\|\nabla f(x)+\nabla \tilde h_{i,\mu}(x,\xi)-\nabla\tilde{\psi}_{\mu}(x)\|^2\right] \\
	%		&&\quad =\mathbb{E}_{\xi}\left[ \sum_{i=1}^q p_{x,\mu}^{\xi}(i)\|\nabla \tilde h_{i,\mu}(x,\xi)-\nabla\tilde{h}_{\mu}(x)\|^2\right] \\
	%		&&\quad =\mathbb{E}_{\xi}\left[ \sum_{i=1}^q p_{x,\mu}^{\xi}(i) \|\nabla \tilde h_{i,\mu}(x,\xi)\|^2\right] +\|\nabla\tilde{h}_{\mu}(x)\|^2-2\left\langle \mathbb{E}_{\xi}\left[ \sum_{i=1}^q p_{x,\mu}^{\xi}(i)\nabla \tilde h_{i,\mu}(x,\xi)\right] ,\nabla\tilde{h}_{\mu}(x)\right\rangle \\
	%		&&\quad	=\mathbb{E}_{\xi}\left[ \sum_{i=1}^q p_{x,\mu}^{\xi}(i)\|\nabla \tilde h_{i,\mu}(x,\xi)\|^2\right]  -\|\nabla\tilde{h}_{\mu}(x)\|^2\\
	%		&&\quad	\leq\mathbb{E}_{\xi,i}\left[ \|\nabla \tilde h_{i,\mu}(x,\xi)\|^2\right] 
	%		\leq\sigma^2,
	%	\end{eqnarray*}
%	where the last inequality follows from $\nabla \tilde h_{i,\mu}(x,\xi)\in\partial h_{i}(x,\xi)$ by Proposition 1 (3) of \cite{chen2012smoothing}.
\qed

\end{appendices}

%%===========================================================================================%%
%% If you are submitting to one of the Nature Portfolio journals, using the eJP submission   %%
%% system, please include the references within the manuscript file itself. You may do this  %%
%% by copying the reference list from your .bbl file, paste it into the main manuscript .tex %%
%% file, and delete the associated \verb+\bibliography+ commands.                            %%
%%===========================================================================================%%

%\bibliography{sn-bibliography}% common bib file
\bibliography{references}
%% if required, the content of .bbl file can be included here once bbl is generated
%%\input sn-article.bbl

\end{document}